\documentclass[a4paper, 12pt]{article} 
\usepackage{amssymb,latexsym,color} \author{Johannes Sj\"ostrand\\
\small IMB, Universit\'e de Bourgogne\\
\small 9, Av. A. Savary, BP 47870\\
\small FR-21078 Dijon C\'edex\\
\footnotesize  and UMR 5584, CNRS\\
\footnotesize johannes.sjostrand@u-bourgogne.fr
} \date{}
\title{Spectral properties of non-self-adjoint operators.}

\newtheorem{dref}{Definition}[section] \newtheorem{lemma}[dref]{Lemma}
\newtheorem{theo}[dref]{Theorem} \newtheorem{prop}[dref]{Proposition}
\newtheorem{remark}[dref]{Remark} \newtheorem{ex}[dref]{Example}

\newenvironment{proof}{\par\noindent{{\bf Proof.}}}{\hfill$\Box$
\medskip} 
\newenvironment{proofof}{\par\noindent{{\bf Proof} of }}{\hfill$\Box$
\medskip} 
\newenvironment{outline}{\par\noindent{{\bf Outline of the Proof.}}}{\hfill$\Box$
\medskip} 
 \newcommand\R{\mathbb{R}}

\newcommand{\ekv}[2]{\begin{equation}\label{#1}#2\end{equation}}
\newcommand{\eekv}[3]{\begin{eqnarray}\label{#1}#2 \\ #3
\nonumber\end{eqnarray}}
\newcommand{\eeekv}[4]{\begin{eqnarray}\label{#1}#2 \\ #3
\nonumber\\#4\nonumber\end{eqnarray}}
\newcommand{\eeeekv}[5]{\begin{eqnarray}\label{#1}#2 \\ #3
\nonumber\\#4\nonumber\\#5\nonumber\end{eqnarray}}

  \newcommand\iint{\int\hskip -2mm\int}

\newcommand{\no}[1]{(\ref{#1})} \newcommand\trans[1]{{^t\hskip -2pt
#1}}
\begin{document}

\maketitle
\begin{abstract} This text is a slightly expanded version of my 6 hour mini-course at the PDE-meeting in \'Evian-les-Bains in June 2009. The first part gives some old and recent results on non-self-adjoint differential operators. The second part is devoted to recent results about Weyl distribution of eigenvalues of elliptic operators with small random perturbations. 

\medskip\centerline{\bf R\'esum\'e}

\smallskip
Ce texte est une version l\'eg\`erement complet\'ee de mon cours de 6 heures au colloque d'\'equations aux d\'eriv\'ees partielles \`a \'Evian-les-Bains en juin 2009. Dans la premi\`ere partie on expose quelques r\'esultats anciens et r\'ecents sur les op\'erateurs non-autoadjoints. La deuxi\`eme partie est consacr\'ee aux r\'esultats r\'ecents sur la distribution de Weyl des valeurs propres des op\'erateurs elliptiques avec des petites perturbations al\'eatoires. 
\end{abstract}

\tableofcontents

\section{Introduction}\label{In}
\setcounter{equation}{0}
\par For self-adjoint and more generally normal operators on some complex 
Hilbert space ${\cal H}$ we have a nice theory, including the spectral 
theorem and the wellknown and important resolvent estimate,
\ekv{0.1}
{\Vert (z-P)^{-1}\Vert \le ({\rm dist\,}(z,\sigma (P)))^{-1},}
where $\sigma (P)$ denotes the spectrum of $P$. The spectral theorem also gives very nice control over functions of self-adjoint operators, so for instance if $P$ is self-adjoint with spectrum contained in the half interval $[\lambda _0,+\infty [$, then 
\ekv{0.2}
{\Vert e^{-tP}\Vert \le e^{-\lambda _0t},\ t\ge 0 .}

However, non-normal operators appear frequently in different problems; Scattering poles, 
Convection-diffu\-sion problems, Kramers-Fokker-Planck equation, damped wave 
equations, linearized operators in fluid dynamics. Then typically, $\Vert 
(z-P)^{-1}\Vert $ may be very large even when $z$ is far from the 
spectrum and this implies mathematical difficulties:\\
-- When studying the distribution of eigenvalues,\\
-- When studying functions of the operator, like $ 
e^{-tP}$ and its norm.

The largeness of the norm of the resolvent far away from the spectrum also makes the eigenvalues very unstable under small perturbations of the operator and this is a source of mathematical and numerical difficulties.  

\par There are two natural reactions to this problem:
\begin{itemize}
\item Change the Hilbert space norm to make the operators look 
more normal. This is quite natural to do when there is no clear unique choice of the ambient Hilbert space, like in problems for scattering poles (resonances).    
\item Recognize that the region of the $z$-plane where $\Vert 
(z-P)^{-1}\Vert $ is large, has its own interest. One can then introduce the $\epsilon $-pseudospectrum which is the set of points $z\in {\bf C}$, which are either in the spectrum or such that $\| (z-P)^{-1}\|>\epsilon $. and study this set in its own right.
\end{itemize}
\par The first point of view will be illustrated in the first part of these notes, and is at the basis of a whole range of methods from that of analytic ditations in the study of resonances to more microlocal methods.
\par The second point of view has been promoted by numerical analysts like L.N.~Trefethen and then made its way into analysis through contributions by E.B.~Davies, M.~Zworski and others. That spectral instability is not only a nuisance, but can be at the origin of nice and previously unexpected results, will hopefully be clear from the second and main part of these lecture notes, where we shall describe some results about Weyl asymptotics of the distribution of eigenvalues of elliptic operators with small random perturbations. These results have been obtained in recent works by M.~Hager \cite{Ha05, Ha06a, Ha06b}, Hager and the author \cite{HaSj08}, W.~Bordeaux Montrieux \cite{Bo}, the author \cite{Sj08a, Sj08b} and Bordeaux Montriex and the author \cite{BoSj09}.

\par The results in the second part of the lectures rely on microlocal analysis, combined with quite classical methods for non-self-adjoint operators and holomorphic functions of one variable and some elementary probability theory. It therefore was natural to include a first part that treats some older and newer results and methods about non-self-adjoint operators. Many of these methods combined with microlocal analysis have been used with great succes in resonance theory, but that is beyond the scope of these lectures.. Let us nevertheless mention results by T.~Christiansen \cite{Chr05, Chr06} and Christiansen--P.~Hislop \cite{ChrHi05} that establish Weyl type lower bounds for the number of scattering poles in large discs in generic situations. 

Here is the plan of the notes:

\medskip\par
Part I is devoted to some old and recent general results for non-self-adjoint differential operators.

\smallskip\par
In Section \ref{ell} we describe some classical results for the distribution of eigenvalues of elliptic operators, starting with a result of T.~Carleman about Weyl asymptotics of the eigenvalues for operators with real principal symbol. We also give a result by S.~Agmon about
the completeness of the set of generalized eigenvalues as a well as further results by M.S.~Agranovich, A.S.~Markus, V.I.~Matseev.

In Section \ref{ps}, we start by giving some basic definitions and facts about the $\epsilon $-pseudospectrum, then go on by describing the Davies--H\"ormander quasi-mode construction under the assumption that a certain Poisson bracket is non-zero and of suitable sign.

In Section \ref{bdy}, we describe some estimates on the size of norm of the resolvent when the spectral parameter is close to the range of the semi-classical principal symbol. These estimates are closely related to subellipticity estimates for operators of principal type and the results here are due to Dencker--Sj\"ostrand--Zworski with some recent partial improvement of the author. 

In Section \ref{sur}, we discuss some recent results from two different areas; the Kramers-Fokker-Planck operator and spectral asymptotics for analytic operators in two dimensions. The first topic (based on joint works with F.~H\' erau--C.~Stolk, F.~H\' erau--M.~Hitrik and inspired by works of H\' erau--F.~Nier, B.~Helffer--F.~Nier) was chosen because it has a very concrete importance, while the second topic (based on joint works with M.~Hitrik and S.~V{\~ u} Ng{\d o}c) is of interest as an example of how  to get precise information about individual eigenvalues. In both situations the methods exploit suitable changes of the Hilbert space norms.

\smallskip
Part II is devoted to Weyl asymptotics for the eigenvalues of elliptic operators with small random perturbations. 

\smallskip
\par In Section \ref{ze}, we give a result, that generalizes and improves earlier results by Hager and Hager--Sj\"ostrand about the number of zeros of holomorphic functions of exponential growth, which is clearly close to classical results for entire functions but that we have not found in the literature. This result is used in an essential way in the sequel.

In Section \ref{one} we treat the one-dimensional semi-classical case very much in the spirit of Hager. This case is easier than the general case, and it has some special features that permit to have more precise results, and is most likely the first testing case for more refined questions about statistics and correlation of eigenvalues.

In Section \ref{mult} we establish Weyl asymptotics in the multi-dimensional semi-classical case, by combination of complex analysis (Section \ref{ze}), microlocal analysis spectral theory and probabilistic arguments.

In Section
\ref{alm}, we consider the large eigenvalues of elliptic operators. In the semi-classical case the results say that we have Weyl asymptotics with a probability tending to 1 very fast when Planck's constant tends to zero. The study of large eigenvalues of elliptic operators can often be reduced to a semi-classical study, and by performing such a reduction and applying the Borel-Cantelli lemma, we show that Weyl asymptotics holds almost surely. Such a result was obtained by W.~Bordeaux Montrieux for elliptic operators on $S^1$ using the results and the approach of Hager. Here we mainly describe a corresponding multi-dimensional result, obtained jointly with Bordeaux Montrieux,  where the semi-classical part is the one described in Section \ref{mult}.

In Section \ref{open} we formulate some open problems.

\part{Some general results}

\section{Elliptic non-self-adjoint operators, some classical results}\label{ell}
\setcounter{equation}{0}

This is a classical area. T.~Carleman \cite{Ca36} considered the Dirichlet realization $P$ of a second order elliptic operator in a bounded domain $\Omega \Subset {\bf R}^3$, assuming enough smoothnes on the coefficients and on the boundary, and he also assumed that the principal symbol is strictly positive so that the non-self-adjointness can come only from the lower order symbols. In this case we see easily that the spectrum consists of isolated eigenvalues of finite algebraic multiplicity and we will allways count the eigenvalues with their multiplicity. Carleman showed that the eigenvalues are contained in a parabolic neighborhood of the positive real axis and that the real parts are distributed according to the Weyl asymptotics, i.e.
\ekv{ell.1}
{
\# \{\mu  \in \sigma (P);\, \Re \mu \le \lambda \} = \frac{1}{2\pi }\lambda ^{3/2}(\mathrm{vol\,} \{ (x,\xi )\in T^*\Omega ;\, p(x,\xi )\le 1\} +o(1)),\ \lambda \to \infty 
}
His method of proof consisted in studying the trace
of $(P+\kappa ^2)^{-1}-(P+\kappa _0^2)^{-1}$ for a fixed $\kappa _0$ in the limit when $\kappa \to \infty $ and to apply a Tauberian argument.

After Carleman there have been important results of Keldysch which have inspired later workers in the field, like Agmon, Agranovich, Markus and Matseev. In a spirit similar to that of Carleman, M.S.~Agranovich and A.S.~Markus \cite{AgMa89} considered a non-self-adjoint elliptic classical pseudodifferential operator on a compact manifold $\Omega $ of dimension $n$, of order $m>0$. Let $p$ denote the principal symbol. Assume that the range of $p$ is in a sector $\{ z\in {\bf C};\, |\mathrm{arg\,}p|<\theta \}$ where $\theta <\pi $. If the quantity
\ekv{ell.2}
{
d:= \frac{1}{(2\pi )^nn}\int _{\Omega }\int_{|\xi |=1}p(x,\xi )^{-n/m}S(d\xi) dx
} 
is $\ne 0$ then the authors show that 
\ekv{ell.3}
{
N(\lambda )\asymp \lambda ^{n/m},\ \lambda \to \infty 
}
where $N(\lambda )=\# (\sigma (P)\cap D(0,\lambda ))$ and $D(0,\lambda )$ denotes the open disc of radius $\lambda $ centered at 0. They also obtain some more precise inequalities for the quantities $\limsup_{\lambda \to \infty }\lambda ^{-n/m}N(\lambda )$ and $\liminf_{\lambda \to \infty }\lambda ^{-n/m}N(\lambda )$. Especially when these two limits are equal, then the corresponding quantity is in the interval $[|d|,\Delta ]$, where 
\ekv{ell.4}
{
\Delta := \frac{1}{(2\pi )^nn}\int _{\Omega }\int_{|\xi |=1}|p(x,\xi )|^{-n/m}d\xi dx.
}
Here it is interesting to notice that $d\ne 0$ when the angle $2\theta $ is smaller than $\pi m/n$. To prove these results, the authors first establish the asymptotic trace formula
\ekv{ell.4.2}
{
\mathrm{tr\,}((\mu +A)^{-\ell})\sim \mathrm{Const.}\mu ^{n/m-\ell},\ \mu \to +\infty 
}
for $\ell m>n$, where the constant can be expressed with the help of $d$, and then apply a Tauberian argument. Here $(\mu +A)^{-\ell}$ is of trace class when $\ell m>n$.

The operator $P=f(x)D_x+g(x)$ on $S^1$ is elliptic when $f\ne 0$, and its spectral behaviour can be studied explicitly. When the range of $f$ is in a sector of angle less than $\pi $, then we have a nice spectrum sitting on a line, while for larger values, strange things may happen (\cite{Bo}, \cite{Se86}) and the spectrum may be either equal to ${\bf C}$ or empty.

Another interesting and potentially important question is that of the completeness of the set of generalized eigenfunctions. This question was studied by Keldysh and Agmon. Agmon studied elliptic boundary value problems, let us here formulate his result in the case of elliptic differential operators on manifolds without boundary \cite{Ag62} (see also \cite{Ag65}). Assume that $P$ is such an operator of even order $m>0$ and assume that the symbol $p$ takes its values away from a finite union of closed half-rays $e^{i\theta _j}[0,+\infty [$, where $\theta _1<\theta _2<...<\theta _N$ belong to $[0,2\pi [$ and assume that the angles $\theta _{j+1}-\theta _j$ are all strictly smaller than $m\pi /n$ (with the convention that $\theta _{N+1}=\theta _1+2\pi $). Then the generalized eigenfunctions of $P$ are complete in $L^2$, in the sense that they span a dense subspace. 

The proof uses the following ingredients: 
\begin{itemize}
\item First we show that the resolvent $(z -P)^{-1}$ is well-defined and of norm ${\cal O}(|z|^{-1})$ when $z$ tends to infinity along one of the half rays given by $\mathrm{arg\,}z=\theta _j$. In particular the resolvent exists for at least one value of $z$ and by analytic Fredholm theory this implies that the spectrum of $P$ consists of isolated eigenvalues and each such value is of finite algebraic multiplicity. 
\item Using a suitable functional determinant which is an entire function of fractional exponential grow and whose zeros are the eigenvalues, we get a polynomial control over the number $N(\lambda )$ when $\lambda \to \infty $ and a corresponding exponential control over the resolvent on a family of circles $\partial D(0,r _j)$, $r_j\to \infty $, when $j\to \infty $: 
\ekv{ell.6}
{
\Vert (z-P)^{-1}\Vert \le {\cal O}(1)\exp (|z|^{\frac{n}{m}+\epsilon }),\ |z|=\rho _j, 
}
where $\epsilon >0$ can be chosen arbitrarily small. 
\item Of course, $(z-P)^{-1}$ will have poles at the eigenvalues, but if $f_0\in L^2$ is orthogonal to all the generalized eigenfunctions, then $F_f(z):=((z-P)^{-1}f|f_0)_{L^2}$ turns out to be an entire function for every $f\in L^2$. Thanks to the condition on the angles, and the exponential control of the resolvent, we can apply the Phragm\'en Lindel\"of theorem to conclude that $F_f(z)$ is actually constant. By restricting to a ray of minimal growth we see that the constant has to be zero, and varying $f$ it is not hard to see from this that $f_0$ has to be zero.
\end{itemize}

Let us mention that rays of minimal growth of the resolvent have also been used by R.~Seeley \cite{Se66}. We could also point out that under the same ellipticity condition and the assumption that the range of $p$ is not equal to all of ${\bf C}$, we have the classical upper bound 
\ekv{ell.7}
{
N(\lambda )={\cal O}(\lambda ^{n/m}),\ \lambda \to +\infty .
}
This can be proved in a similar manner using a relative determinant and the Jensen formula. A more elementary proof is the following, that was pointed out to me by M.S.~Agranovich \cite{Ag09}: After replacing $P$ by $P-\lambda _0$ for a sufficiently large $\lambda _0$ on a ray of minimal growth, we may assume that $P$ is bijective. Let $s_1(P^{-1})\ge s_2(P^{-1})\ge ...$ be the singular values of the inverse $P^{-1}$ (i.e. the decreasing sequence of eigenvalues of $(P^*P)^{-1/2}$ which is bounded from $L ^2$ to $H^m$). If $C>0$ is large enough we have $(P^*P)^{-1}\le C(1-\Delta )^{-m}$ where $\Delta $ is the Laplace Beltrami operator on $X$ for some smooth Riemannian metric and this implies corresponding inequalities for the eigenvalues. Applying the Weyl asymptotics for $-\Delta $, we conclude that $s_j(P^{-1})\le {\cal O}(j^{-m/n})$. Then Corollary 3.2 in Chapter III of \cite{GoKr} implies that $\lambda _j(P^{-1})={\cal O}(j^{-m/n})$, where $\lambda _j(P^{-1})$ denote the eigenvalues of $P^{-1}$ arranged so that $j\mapsto |\lambda _j(P^{-1})|$ is decreasing. Then (\ref{ell.7}) follows.

Markus and V.I.~Matseev \cite{MaMa79} have established interesting estimates on the difference of the counting function for a self-adjoint (or more generally normal) operator and the counting function for the real parts of the eigenvalues of a small perturbation of that operator. The proofs are based on the use of relative determinants. We will not state the general results, but simply mention a corollary of the general result: Let $P$ be an elliptic differential operator of order $m$ on a compact manifold with positive principal symbol $p(x,\xi )$. Then the eigenvalues are contained in a thin neighborhood of the positive real axis and if $N(\lambda )$ denotes the number of such eigenvalues with real part $\le \lambda $, then $N(\lambda )=(2\pi )^{-n}\mathrm{Vol\,}p^{-1}([0,1])\lambda ^{n/m}+{\cal O}(\lambda ^{(n-1)/m})$. Notice that the remainder estimate is the same as in the general result of Avakumovi\'c \cite{Av56}, Levitan \cite{Lev56} ($m=2$) and H\"ormander \cite{Ho68} (general $m$) for self-adjoint elliptic operators which in turn depend on trace formulas, not for resolvents as in Carleman's approach but for hyperbolic evolution problems.

\section{Pseudospectrum, quasi-modes and spectral instability}\label{ps}
\setcounter{equation}{0}
Let ${\cal H}$ be a complex Hilbert space and let $P:{\cal H}\to {\cal
H}$ be a closed densely defined operator. Recall that the resolvent
set is defined as
$$
\rho (P)=\{ z\in {\bf C};\, P-z:{\cal D}(P)\to {\cal H}\hbox{ has a bounded
2-sided inverse}\}.
$$
It is an open set, if $z\in \rho (P)$ and $\Vert (z-P)^{-1}\Vert =
1/\epsilon $, then the open disc $D(z,\epsilon )$ is contained in
$\rho (P)$. The spectrum of $P$ is the closed set 
$$
\sigma (P)={\bf C}\setminus \rho (P).
$$
Following Trefethen--M.~Embree \cite{TrEm} we define, for $\epsilon >0$,
the $\epsilon $-pseudospectrum  as the open set
\ekv{ps.1}
{
\sigma _\epsilon (P)=\sigma (P)\cup \{ z\in \rho (P);\, \Vert
(z-P)^{-1}\Vert>1/\epsilon \} .
}
Unlike the spectrum, the $\epsilon $-pseudospectrum will change if we replace the given norm on ${\cal H}$ by an equivalent one.

\par $\sigma _\epsilon (P)$ can be characterized as a set of spectral
instability, by the following simplified version of a theorem of Roch
and Silberman:
\ekv{ps.2}
{
\sigma _\epsilon (P)=\bigcup_{Q\in {\cal L } ({\cal H},{\cal H})\atop 
\Vert Q\Vert <\epsilon }\sigma (P+Q).}
The result becomes more sublte if we use the more traditional
definition with a non-strict inequality in (\ref{ps.1}).

\begin{proof}
Let $\widetilde{\sigma }_\epsilon (P)$ denote the right hand side in
(\ref{ps.2}). If $z\in {\bf C}\setminus \sigma _\epsilon (P)$, then by
a perturbation argument, we see that $z\in {\bf C}\setminus
\widetilde{\sigma }_\epsilon (P)$.

\par Let $z\in \sigma _\epsilon (P)$. If $z\in \sigma (P)$ we also have $z\in \widetilde{\sigma }_\epsilon (P)$, so we may assume that $z\in \rho (P)$. Then $\exists$ $u\in {\cal
D}(P)$, $v\in {\cal H}$ such that $\Vert u\Vert =1$, $\Vert v\Vert
<\epsilon $, $(P-z)u=v$. Let $Q$ be the rank one operator from ${\cal
H}$ to ${\cal H}$, given by $Q\phi =-(\phi |u)v$. Then $\Vert Q\Vert =
\Vert u\Vert \Vert v\Vert <\epsilon $ and
$
(P+Q-z)u=v+Qu=v-v=0
$, so $z\in \sigma (P+Q)$, and $z\in \widetilde{\sigma }_\epsilon (P)$.
\end{proof}

Using the subharmonicity of the function $z\mapsto \Vert
(z-P)^{-1}\Vert$ we notice that every bounded connected component of
$\sigma _\epsilon (P)$ contains an element of $\sigma (P)$.

We next discuss the construction of {\it quasimodes} for non-normal
differential operators which shows that very often we get large
$\epsilon $-pseudospectra. The background and starting point is a
result by E.B.~Davies \cite{Da} for non-selfadjoint Schr\"o\-dinger
operators in dimension 1. M.~Zworski \cite{Zw} observed that this is
essentially an old result of H\"ormander \cite{Ho60a, Ho60b}(1960), and that we have the
following generalization, with $\{ a,b\} =a'_\xi \cdot b'_x-a'_x \cdot
b'_\xi =H_a(b) $ denoting the Poisson bracket of $a=a(x,\xi )$,
$b(x,\xi )$.  
\begin{theo}\label{ps1}  Let 
\ekv{ps.3}
{
P(x,hD_x)=\sum_{\vert \alpha \vert \le m}a_\alpha (x)(hD_x)^\alpha ,\
D_x=\frac{1}{i}{\partial \over\partial x}
}
have smooth coefficients in the open set $\Omega \subset{\bf R}^n$. Put 
$p(x,\xi )=\sum_{\vert \alpha \vert \le m}a_\alpha (x)\xi ^\alpha $. Assume 
$z=p(x_0,\xi _0)$ with  
${1\over i}\{ p,\overline{p}\}(x_0,\xi_0) >0$. Then $\exists$ $u=u_h$,
with $\Vert u\Vert=1$, 
$\Vert (P-z)u\Vert ={\cal O}(h^\infty )$, when $h\to 0$.

\end{theo}

In the case when the coefficients are all analytic we can replace ``$h^\infty $'' by ``$e^{-1/Ch}$ for some $C>0$''.

\par 
Notice that this implies that if the resolvent $(P-z)^{-1}$ exists
then its norm is greater than any negative power of $h$ when $h\to 0$ (and even exponentially large in the analytic case).

In the case $n\ge
2$, we noticed with A.~Melin in \cite{MeSj02} that if $z=p(\rho )$ and $\Re p$, $\Im
p$ are independent at $\rho $, then $\frac{1}{i}\{ p,\overline{p}\}$
times the natural Liouville measure is equal to a constant times the
restriction to $p^{-1}(z)$ of $\sigma ^{n-1}$ which is a closed
form. It follows that if $\Gamma $ is a compact connected component of
$p^{-1}(z)$ on which $d\Re p$ and $d\Im p$ are pointwise independent,
then the average of $\frac{1}{i}\{ p,\overline{p}\}$ over $\Gamma $
with respect to the Liouville measure has to vanish. Hence if there is
a point on $\Gamma $ where the Poisson bracket is $\ne 0$ then there
is also point where it is positive. In the case $n=1$ we have a
similar phenomenon: If (for instance thanks to suitable ellipticity
assumption) we know that $p^{-1}(z)$ is finite and that $\frac{1}{i}\{
p,\overline{p}\}$ is $\ne 0$ everywhere on that set, then this set is
finite and if it is contained in  the interior of a connected bounded
set $\Omega $ in
phase space with smooth boundary such that the variation of
$\mathrm{arg\,}(p-z)$ along that boundary is equal to zero, then we
have to have an equal number of points in $p^{-1}(z)$ where
$\frac{1}{i}\{ p,\overline{p}\}$ is positive and where it is
negative. This follows from the observation that the argument
variation of $p-z$ along a small positively oriented circle around a
point in $p^{-1}(z)$ is $\mp 2\pi $ when $\pm \frac{1}{i}\{
p,\overline{p}\}>0$ at that point.

\begin{ex}\label{ps2} \rm $P=-h^2\Delta +V(x)$, $p(x,\xi )=\xi^2+V(x)$,
${1\over i}\{p,\overline{p}\}=-4\xi \cdot \Im V'(x)$.\end{ex}

More recently K.~Pravda-Starov \cite{Pr} improved this result by
adapting a more refined quasi-mode construction of R.~Moyer (in 2
dimensions) and H\"ormander \cite{Ho} for adjoints of operators that
do not satisfy the Nirenberg-Tr\`eves condition $(\Psi )$ for local 
solvability. 

The proof in the $C^\infty $-case in \cite{Zw} is by a standard
reduction of semi-classical results to classical results in ordinary
microlocal analysis. In \cite{DeSjZw} we gave a direct proof and also
treated the case of analytic coefficients, which is also essentially
quite old. Here is a brief outline of a
\begin{proof} In the following we use the notation $\mathrm{neigh\, }(a,A)$ for ``some neighborhood of $a$ in $A$''. If $\phi \in C^\infty
(\mathrm{neigh\,}(x_0,{\bf R}^n))$ satisfies $\phi '(_0):\xi _0\in {\bf
R}^n$, and 
\ekv{ps.4}{\Im \phi ''(x_0)>0,} then we can define the
complex Lagrangian manifold \ekv{ps.5} { \Lambda _\phi :=\{ (x,\phi
'(x));\, x\in \mathrm{neigh\,}(x_0,{\bf C}^n)\} } where we extend $\phi
$ to a complex neighborhood by taking an almost holomorphic extension,
i.e. a smooth extension such that $\overline{\partial }\phi ={\cal
O}((\Im x)^\infty )$. In this case we can content ourselves with
working with formal Taylor expansions at $x_0$, and then $\Lambda
_\phi $ can be viewed as an equivalence class of real submanifolds of
the complexified phase space ${\bf C}^{2n}$ where two submanifolds are
equivalent if they agree to infinite order at $(x_0,\xi _0)$. As
observed by H\"ormander \cite{Ho71b} and developped a lot by the
author with A.~Melin in \cite{MeSj76} and in other works, the
positivity assumption (\ref{ps.4}) can be formulated equivalently by
saying that 
\ekv{ps.6} { \frac{1}{i}\sigma (t,\overline{t})>0,\ 0\ne t\in
T_{(x_0,\xi _0) } (\Lambda _\phi ), } 
where $\sigma $ denotes the
symplectic 2-form, here viewed as a bilinear form on the complexified
tangentspace of the cotangent space at $(x_0,\xi _0)$.  

\par Let $z,
p, (x_0,\xi _0)$ be as in the theorem. Then we observe that
$\frac{1}{i}\sigma (H_p,\overline{H_p})=\frac{1}{i}\{
p,\overline{p}\}>0$. Moreover, the real set $\Sigma :=p^{-1}(z)$ is a
smooth symplectic manifold near $(x_0,\xi _0)$ and using the Darboux
theorem, we can identify it with ${\bf R}^{2(n-1)}$ and hence find a
Lagrangian submanifold $\Lambda '$ in its compexification passing
through $(x_0,\xi _0)$ that satisfies the positivity condition
(\ref{ps.6}). Viewing the complexification of $\Sigma $ as a
submanifold of ${\bf C}^{2n}$, we can take $\Lambda =\{ \exp sH_p(\rho
);\, s\in \mathrm{neigh\,}(0,{\bf C}),\, \rho \in
\mathrm{neigh\,}((x_0,\xi _0),{\bf C}^{2n})\}$. Using that $H_p$ is
symplectically orthogonal to the tangent space of $\Sigma $ it is then
quite easy to verify that $\Lambda $ is a complex Lagrangian manifold
to $\infty $ order at $(x_0,\xi _0)$ contained (to infinite order) in
the complex characteristic hypersurface $\{ \rho \in
\mathrm{neigh\,}((x_0,\xi _0),{\bf C}^{2n});\, p(\rho )=0\}$ and
satisfying the positivity condition (\ref{ps.6}). Hence to infinite
order, $\Lambda $ is of the form $\Lambda _\phi $ for a function $\phi
$ as in (\ref{ps.4}), (\ref{ps.5}), which also fulfills the eiconal
equation
\ekv{ps.7}
{
p(x,\phi '(x))={\cal O}(|x-x_0|^\infty ).
}
We normalize $\phi $ by requiring that $\phi (x_0)=0$. Then the
function $e^{i\phi (x)/h}$ is rapidly decreasing with all its
derivatives away from any neighborhood of $x_0$, and by a complex
version of the standard WKB-construction we can construct an elliptic
symbol $a(x;h)\asymp a_0(x)+ha_1(x)+...,$ by solving the suitable
transport equations to infinite order at $x_0$, such that if $\chi \in
C_0^\infty (\mathrm{neigh\,}(x_0,{\bf R}^n))$ is equal to 1 near
$x_0$, then $u(x;h)=\chi (x)h^{-n/4}a(x;h)e^{i\phi (x)/h}$ has the
required properties. \end{proof}

\section{Boundary estimates of the resolvent}\label{bdy}
\setcounter{equation}{0}

\subsection{Introduction}\label{inbdy}

In this section we are interested in bounds on the resolvent of an $h$-pseudo\-differential operator when $z$ is close to the boundary of the range of $p$. As with the quasi-mode construction this question is closely related to classical results in the general theory of linear PDE, and with N.~Dencker and Zworski (\cite{DeSjZw}) we were able to find quite general results closely related to the classical topic of subellipticity for pseudodifferential operators of principal type, studied by Egorov, H\"ormander and others. See \cite{Ho}.

In \cite{DeSjZw} we obtained resolvent estimates at certain boundary points,\\
(A) under a non-trapping condition,\\
and\\
(B) under a stronger ``subellipticity condition''.

\par In case (A) we could apply quite general and simple arguments related to the propagation of regularity and in case (B) we were able to adapt general Weyl-H\"ormander calculus and H\"ormander's treatment of subellipticity for operators of principal type (\cite{Ho}). 
In the first case we obtained that the resolvent extends and has temperate growth in $1/h$ in discs of radius ${\cal O}(h\ln 1/h)$ centered at the appropriate boundary points, while in case (B) we got the corresponding extension up to distance ${\cal O}(h^{k/(k+1)})$, where the integer $k\ge 2$ is determined by a condition of ``subellipticity type''. 

Using a method based on semi-groups led to a strengthened result in case (B): The resolvent can be extended to a disc of radius ${\cal O}((h\ln 1/h)^{k/(k+1)})$ around the appropriate boundary points.

\par Let $X$ be equal to ${\bf R}^n$ or equal to a compact smooth
manifold of dimension $n$. 

\par In the first case, let $m\in C^\infty ({\bf R}^{2n};[1,+\infty
[)$ be an order function (see \cite{DiSj99} for more details about the
pseudodifferential calculus) in the sense that for some $C_0,N_0>0$,
\ekv{inbdy.1}
{
m(\rho )\le C_0\langle \rho -\mu \rangle^{N_0}m(\mu ),\ \rho ,\mu \in
{\bf R}^{2n},
}
where $\langle \rho -\mu \rangle= (1+|\rho -\mu |^2)^{1/2}$. Let
$P=P(x,\xi ;h)\in S(m)$, meaning that $P$ is smooth in $x,\xi $ and
satisfies 
\ekv{inbdy.2}
{
|\partial _{x,\xi }^\alpha P(x,\xi ;h)|\le C_\alpha m(x,\xi ),\ (x,\xi
)\in {\bf R}^{2n},\, \alpha
\in {\bf N}^{2n},
}
where $C_\alpha $ is independent of $h$. We also assume that 
\ekv{inbdy.3}
{
P(x,\xi ;h)\sim p_0(x,\xi )+hp_1(x,\xi )+...,\hbox{ in }S(m),
}
and write $p=p_0$ for the principal symbol. We impose the ellipticity
assumption
\ekv{inbdy.4}
{
\exists w\in {\bf C},\, C>0,\hbox{ such that }|p(\rho )-w|\ge m(\rho )/C,\
\forall \rho \in {\bf R}^{2n},\,|\rho |\ge C.
}
In this case we let 
\ekv{inbdy.5}
{
P=P^w(x,hD_x;h)=\mathrm{Op}(P(x,h\xi ;h))
}
be the Weyl quantization of the symbol $P(x,h\xi ;h)$ that we can view as a
closed unbounded operator on $L^2({\bf R}^n)$.

\par In the second case when $X$ is compact manifold, we let $P\in
S^m_{1,0}(T^*X)$ (the classical H\"ormander symbol space ) of order
$m\ge 0$, meaning that
\ekv{inbdy.6}
{
|\partial _x^\alpha \partial _{\xi }^\beta P(x,\xi ;h)|\le C_{\alpha
  ,\beta }\langle \xi \rangle^{m-|\beta |},\ (x,\xi )\in T^*X, } where $C_{\alpha ,\beta }$ are independent of $h$. We also assume that we have an expansion of the type (\ref{inbdy.3}), now in the sense that 
\ekv{inbdy.7} { P(x,\xi ;h)-\sum_0^{N-1}h^jp_j(x,\xi )\in h^NS^{m-N}_{1,0}(T^*X),\ N=1,2,...  } and we quantize the symbol $P(x,h\xi ;h)$ in the standard (non-unique) way, by doing it for various local coordinates and paste the quantizations together by means of a partition of unity. In the case $m>0$ we impose the ellipticity condition \ekv{inbdy.8} { \exists C>0,\hbox{ such that }|p(x,\xi )|\ge \frac{\langle
  \xi \rangle^m}{C},\ |\xi |\ge C.
}
\par Let $\Sigma (p)=\overline{p^*(T^*X)}$ and let $\Sigma _\infty
(p)$ be the set of accumulation points of $p(\rho _j)$ for all
sequences $\rho _j\in T^*X$, $j=1,2,3,..$ that tend to infinity.
By pseudodifferential calculus, if $K\subset {\bf C}\setminus \Sigma (p)$ is compact, then $(z-P)^{-1}$ exists and has uniformly bounded operator norm for $z\in K$, $0<h\ll 1$. 

 The following
theorem (\cite{Sj09a}) is a partial improvement of corresponding results in
\cite{DeSjZw}.
\begin{theo}\label{inbdy1}
We adopt the general assumptions above. Let $z_0\in \partial \Sigma
(p)\setminus \Sigma _\infty (p)$ and assume that $dp\ne 0$ at every
point of $p^{-1}(z_0)$. Then for every such point $\rho $ there exists
$\theta \in {\bf R}$ (unique up to a multiple of $\pi $) such that
$d(e^{-i\theta }(p-z_0))$ is real at $\rho $. We write $\theta
=\theta (\rho )$. Consider the following two cases:
\begin{itemize}
\item (A) For every $\rho \in p^{-1}(z_0)$, the maximal integral curve
  of $H_{\Re (e^{-i\theta (\rho )}p)}$ through the point $\rho $ is not
  contained in $p^{-1}(z_0)$. 
\item (B) There exists an integer $k\ge 1$ such that for every $\rho
  \in p^{-1}(z_0)$, there exists $j\in \{ 1,2,..,k\}$ such that 
$$p^*(\exp tH_p(\rho ))=at^j+{\cal O}(t^{j+1}),\ t\to 0,$$
where $a=a(\rho )\ne 0$. Here $p$ also denotes an almost holomorphic extension to a complex neighborhood of $\rho $ and we put $p^*(\mu )=\overline{p(\overline{\mu })}$. Equivalently, $H_p^j(\overline{p})(\rho )/(j!)=a\ne 0$.
\end{itemize}

\par Then, in case (A), there exists a constant $C_0>0$ such that for
every constant $C_1>0$ there is a constant $C_2>0$ such that the
resolvent $(z-P)^{-1}$ is well-defined for $|z-z_0|<C_1h\ln
\frac{1}{h}$, $h<\frac{1}{C_2}$, and satisfies the estimate 
\ekv{inbdy.9}
{
\| (z-P)^{-1}\|\le \frac{C_0}{h}\exp ( \frac{C_0}{h}|z-z_0|).
}

\par In case (B), there exists a constant $C_0>0$ such that for
every constant $C_1>0$ there is a constant $C_2>0$ such that the
resolvent $(z-P)^{-1}$ is well-defined for $|z-z_0|<C_1(h\ln
\frac{1}{h})^{k/(k+1)}$, $h<\frac{1}{C_2}$ and satisfies the estimate 
\ekv{inbdy.10}
{
\| (z-P)^{-1}\|\le \frac{C_0}{h^{\frac{k}{k+1}}}
\exp (\frac{C_0}{h}|z-z_0|^{\frac{k+1}{k}}).
}
\end{theo} 

\par In \cite{DeSjZw} we obtained (\ref{inbdy.10}) for
$z=z_0$, implying that the resolvent exists and satisfies the same
bound for $|z-z_0|\le h^{k/(k+1)}/{\cal O}(1)$ in case (B) and with
$k/(k+1)$ replaced by 1 in case (A). In case (A) we also showed that
the resolvent  exists with norm bounded by a negative power of $h$
in any disc $D(z_0,C_1h\ln (1/h))$. (The condition in case (B) was formulated a little differently in \cite{DeSjZw}, but the two conditions lead to the same microlocal models and hence they are equivalent.) 
The case (A) of the theorem is basically identical with the corresponding result in \cite{DeSjZw} and was proved using weighted estimates with weights that have at most polynomial growth in $h$. We will not disuss that in detail here and instead we concentrate on the partially new result in the case (B).

When $k=2$ more direct methods are available and more precise bounds can be given, at least in special cases. Such results have been obtained by J.~Martinet \cite{Mart09}, Y.~Almog, B.~Helffer, X.~Pan, see \cite{He09} and W.~Bordeaux Montrieux \cite{Bo}.

Let us now consider the special situation of potential interest for
evolution equations, namely the case when
\ekv{inbdy.11}{z_0\in i{\bf R},}
\ekv{inbdy.12}{
\Re p(\rho )\ge 0\hbox{ in }\mathrm{neigh\,}(p^{-1}(z_0),T^*X).
}
\begin{theo}\label{inbdy2}
We adopt the general assumptions above. Let $z_0\in \partial \Sigma
(p)\setminus \Sigma _\infty (p)$ and assume (\ref{inbdy.11}),
(\ref{inbdy.12}). Also assume that
$dp\ne 0$ on $p^{-1}(z_0)$, so that $d\Im p\ne 0$, $d\Re p=0$ on that set.
 Consider the two cases of Theorem \ref{inbdy1}:
\begin{itemize}
\item (A) For every $\rho \in p^{-1}(z_0)$, the maximal integral curve
  of $H_{\Im p}$ through the point $\rho $ contains a point where $\Re
  p>0$. 
\item (B) There exists an integer $k\ge 1$ such that for every $\rho
  \in p^{-1}(z_0)$, we have $H_{\Im p}^j\Re p(\rho )\ne 0$ for some
  $j\in \{ 1,2,...,k\}$.
\end{itemize}

\par Then, in case (A), there exists a constant $C_0>0$ such that for
every constant $C_1>0$ there is a constant $C_2>0$ such that the
resolvent $(z-P)^{-1}$ is well-defined for 
$$
|\Im (z-z_0)|<\frac{1}{C_0},\ \frac{-1}{C_0}<\Re
z<C_1h\ln\frac{1}{h},\ h<\frac{1}{C_2},
$$
 and satisfies the estimate 
\ekv{inbdy.13}
{
\| (z-P)^{-1}\|\le \cases{ \frac{C_0}{|\Re z|},\ \Re z\le -h,\cr
  \frac{C_0}{h}\exp (\frac{C_0}{h}\Re z), \Re z\ge -h.
 }
}

\par In case (B), there exists a constant $C_0>0$ such that for
every constant $C_1>0$ there is a constant $C_2>0$ such that the
resolvent $(z-P)^{-1}$ is well-defined for 
\ekv{inbdy.13.5}
{|\Im (z-z_0)|<\frac{1}{C_0},\ \frac{-1}{C_0}<\Re
z<C_1(h\ln\frac{1}{h})^{\frac{k}{k+1}},\ h<\frac{1}{C_2},}
 and satisfies the estimate 
\ekv{inbdy.14}
{
\| (z-P)^{-1}\|\le \cases{ \frac{C_0}{|\Re z|},\ \Re z\le -h^{\frac{k}{k+1}},\cr
  \frac{C_0}{h^{\frac{k}{k+1}}}\exp (\frac{C_0}{h}(\Re z)_+^{^{\frac{k}{k+1}}}), \Re z\ge -h^{\frac{k}{k+1}}.
 }
}
\end{theo} 

\subsection{Outline of the proof in case (B)}

Away from the set $p^{-1}(z_0)$ we can use ellipticity, so the problem is to obtain microlocal estimates near a point $\rho_0 \in p^{-1}(z_0)$. After a standard factorization of $P-z$ in such a region, we can further reduce the proof of the first theorem to that of the second one.

\par The main (quite standard) idea of the proof of Theorem \ref{inbdy2} is to study $\exp (-tP/h)$ (microlocally) for $0\le t\ll 1$ and to show that in this case
\ekv{inbdy.15}
{
\| \exp -\frac{tP}{h}\| \le C\exp (-\frac{t^{k+1}}{Ch}),
}
for some constant $C>0$. Noting that that implies that $\| \exp -\frac{tP}{h}\| ={\cal O}(h^\infty )$ for $t\ge h^\delta $ when $\delta (k+1)<1$, and using the formula
\ekv{inbdy.16}
{
(z-P)^{-1}=-\frac{1}{h}\int_0^\infty \exp (\frac{t(z-P)}{h}) dt,
}
we get to (\ref{inbdy.14}).

The most direct way of studying $\exp (-tP/h)$, or rather a microlocal version of that operator, is to view it as a Fourier integral operator with complex phase (\cite{Mas, Ku, MeSj76, Ma}) of the form
\ekv{inbdy.17}
{
U(t)u(x)=\frac{1}{(2\pi h)^n}\iint e^{\frac{i}{h}(\phi (t,x,\eta )-y\cdot \eta )}a(t,x,\eta ;h) u(y)dy d\eta ,
}
where the phase $\phi $ should have a non-negative imaginary part and satisfy the Hamilton-Jacobi equation:
\ekv{inbdy.18}
{
i\partial _t\phi +p(x,\partial _x\phi )={\cal O}((\Im \phi )^\infty ), \hbox{ locally uniformly,}
}
with the initial condition 
\ekv{inbdy.19}
{
\phi (0,x,\eta )=x\cdot \eta .
}
The amplitude $a$ will be bounded with all its derivatives and has an asymptotic expansion where the terms are determined by transport equations. This can indeed be carried out in a classical manner for instance by adapting the method of \cite{MeSj76} to the case of non-homogeneous symbols following a reduction used in \cite{MenSj, Ma}. It is based on making estimates on the fonction 
$$
S_\gamma (t)=\Im (\int_0^t \xi (s)\cdot dx(s))-\Re \xi (t)\cdot \Im x(t)+\Re \xi (0)\cdot \Im x(0)
$$ 
along the complex integral curves $\gamma :[0,T]\ni s\mapsto (x(s),\xi (s))$ of the Hamilton field of $p$. Notice that here and already in (\ref{inbdy.18}), we need to take an almost holomorphic extension of $p$. Using the property (B) one can show that $\Im \phi (t,x,\eta )\ge C^{-1}t^{k+1}$ and from that we can obtain (a microlocalized version of) (\ref{inbdy.15}) quite easily. 

Finally, we prefered a variant: Let 
$$
Tu(x)=Ch^{-\frac{3n}{4}}\int e^{\frac{i}{h}\phi (x,y)}u(y)dy,
$$
be an FBI -- or (generalized) Bargmann-Segal transform that we treat in the spirit of Fourier integral operators with complex phase as in 
\cite{Sj82}. Here $\phi $ is holomorphic in a neighborhood of $(x_0,y_0)\in {\bf C}^n\times {\bf R}^n$, and $-\phi '_y(x_0,y_0)=\eta _0\in {\bf R}^n$, $\Im \phi ''_{y,y}(x_0,y_0)>0$, $\det \phi ''_{x,y}(x_0,y_0)\ne 0$. Let $\kappa _t:\, (y,-\phi '_y(x,y))\mapsto (x,\phi '_x(x,y))$ be the associated canonical transformation. Then microlocally, $T$ is bounded $L^2\to H_{\Phi _0}:=\mathrm{Hol\,}(\Omega )\cap L^2(\Omega ,e^{-2\Phi _0/h}L(dx))$ and has (microlocally) a bounded inverse, where $\Omega $ is a small complex neighborhood of $x_0$ in ${\bf C}^n$. Here the weight $\Phi _0$ is smooth and strictly pluri-subharmonic. If $\Lambda _{\Phi _0}:=\{ (x,\frac{2}{i}\frac{\partial \Phi _0}{\partial x});\, x\in \mathrm{neigh\,}(x_0)\}$, then (in the sense of germs) $\Lambda _{\Phi _0}=\kappa _T(T^*X)$. The conjugated operator $\widetilde{P}=TPT^{-1}$ can be defined locally modulo ${\cal O}(h^\infty )$ (see also \cite{LaSj}) as a bounded operator from $H_\Phi \to H_\Phi $ provided that the weight $\Phi $ is smooth and satisfies $\Phi '-\Phi _0'={\cal O}(h^\delta) $ for some $\delta >0$. (In the analytic frame work this condition can be relaxed.) Egorov's theorem applies in this situation, so the leading symbol $\widetilde{p}$ of $\widetilde{P}$ is given by $\widetilde{p}\circ \kappa _T=p$. Thus (under the assumptions of Theorem \ref{inbdy2}) we have ${{\Re \widetilde{p}}_\vert}_{\Lambda _{\Phi _0}}\ge 0$, which in turn can be used to see that for $0\le t\le h^\delta $, we have $e^{-t\widetilde{P}/h}={\cal O}(1)$: 
$H_{\Phi _0}\to H_{\Phi _t}$, where $\Phi _t\le \Phi _0$ is determined by the real Hamilton-Jacobi problem
\ekv{inbdy.20}
{
\frac{\partial \Phi _t}{\partial t}+\Re \widetilde{p}(x,\frac{2}{i}\frac{\partial \Phi _t}{\partial x})=0,\ \Phi _{t=0}=\Phi _0.
}

Here is a somewhat formal derivation of (\ref{inbdy.20}):

\par Consider formally:
$$
(e^{-t\widetilde{P}/h}u|e^{-t\widetilde{P}/h}u)_{H_{\Phi _t}}=(u_t|u_t)_{H_{\Phi _t}},\ u\in H_{\Phi _0},
$$
and try to choose $\Phi _t$ so that the time derivative of this expression
vanishes to leading order. We get
\begin{eqnarray*}0&\approx&h\partial _t\int u_t\overline{u}_te^{-2\Phi
    _t/h}L(dx)\\
&=&-\left( (\widetilde{P}u_t|u_t)_{H_{\Phi _t}}+
(u_t|\widetilde{P}u_t)_{H_{\Phi _t}}+\int 2\frac{\partial \Phi
  _t}{\partial t}(x)|u|^2 e^{-2\Phi _t/h}L(dx)
\right).
\end{eqnarray*}
Here 
$$
(\widetilde{P}u_t|u_t)_{H_{\Phi _t}}=\int
(\widetilde{p}_{\vert_
{\Lambda _{\Phi _t}}}+{\cal
  O}(h))|u_t|^2e^{-2\Phi _t/h}L(dx),
$$
and similarly for $(u_t|\widetilde{P}u_t)_{H_{\Phi _t}}$, so we would
like to have
$$
0\approx
\int (2\frac{\partial \Phi _t}{\partial t}+2\Re 
\widetilde{p}_{\vert_
{\Lambda _{\Phi _t}}}+{\cal O}(h))|u_t|^2
e^{-2\Phi _t/h}L(dx).
$$

\par We choose $\Phi _t$ to be the solution of (\ref{inbdy.20}). 
 Then the preceding
discussion again shows that $e^{-t\widetilde{P}/h}={\cal O}(1): H_{\Phi
  _0}\to H_{\Phi _t}$. 

\par To get (\ref{inbdy.15}), it suffices to show that $\Phi _t\le \Phi _0-t^{k+1}/C$ for $0\le t\ll 1$. By a geometric discussion, this follows from 
\ekv{inbdy.21}
{
G_t(\rho )\le -t^{k+1}/C,
}
where $G_t$ is a smooth function in a real neighborhood of $\rho _0$, given by
\ekv{inbdy.22}
{
\frac{\partial G_t(\rho )}{\partial t}+\Re p(\rho +iH_{G_t}(\rho ))=0,\ G_0=0.
}
The behaviour of $G_t$ is easy to understand by means of Taylor expansion of the first equation in (\ref{inbdy.22}) at the point $\rho $.

\subsection{Examples}\label{ex}

Consider 
\ekv{ex.1}
{
P=-h^2\Delta +iV(x),\ V\in C^\infty (X;{\bf R}),
}
where either $X$ is a smooth compact manifold of dimension $n$ or $X={\bf R}^n$. In the second case we assume that $p=\xi ^2+iV(x)$ belongs to a symbol space $S(m)$ where $m\ge 1$ is an order function. It is easy to give quite general sufficient condition for this to happen, let us just mention that if $V\in C_b^\infty ({\bf R}^2)$ then we can take $m=1+\xi ^2$ and if $\partial ^\alpha V(x)={\cal O}((1+|x|)^2)$ for all $\alpha \in {\bf N}^n$ and satisfies the ellipticity condition $|V(x)|\ge C^{-1}|x|^2$ for $|x|\ge C$, for some constant $C>0$, then we can take $m=1+\xi ^2+x^2$.

We have $\Sigma (p)=[0,\infty [+i\overline{V(X)}$. When $X$ is compact then $\Sigma _\infty (p)$ is empty and when $X={\bf R}^n$, we have $\Sigma _\infty (p)=[0,\infty [+i\Sigma _\infty (V)$, where $\Sigma _\infty (V)$ is the set of accumulation points at infinity of $V$. 

Let $z_0=x_0+iy_0\in \partial \Sigma (p)\setminus \Sigma _\infty (p)$.
\begin{itemize}
\item In the case $x_0=0$ we see that Theorem \ref{inbdy2} (B) is applicable with $k=2$, provided that $y_0$ is not a critical value of $V$. This is close to problems from fluid dynamics, studied by I.~Gallagher, T.~Gallay, F.~Nier \cite{GaGaNi08}.
\item Now assume that $x_0>0$ and that $y_0$ is either the maximum or the minimum of $V$. In both cases, assume that $V^{-1}(y_0)$ is finite and 
that each element of that set is a non-degenerate maximum or minimum. Then Theorem \ref{inbdy2} (B) is applicable to $\pm iP$. By allowing a more complicated behaviour of $V$ near its extreme points, we can produce examples where \ref{inbdy2} (B) applies with $k>2$.
\end{itemize}

\par Now, consider the non-self-adjoint harmonic oscillator 
\ekv{ex.2}{
Q=-\frac{d^2}{dy^2}+iy^2
}
on the real line, studied by Boulton \cite{Bou} and Davies \cite{Da2}, K.~Pravda Starov \cite{Pr06}. Consider a large spectral parameter $E=i\lambda +\mu $ where $\lambda \gg 1$ and $|\mu |\ll \lambda $. The change of variables $y=\sqrt{\lambda }x$ permits us to identify $Q$ with 
$Q=\lambda P$, where $P=-h^2\frac{d^2}{dx^2}+ix^2$ and $h=1/\lambda \to 0$. 
Hence $Q-E=\lambda (P-(i+\frac{\mu }{\lambda }))$ and Theorem \ref{inbdy2} (B)
 is applicable with $k=2$. We conclude that $(Q-E)^{-1}$ is well-defined and of polynomial growth in $\lambda $ (which can be specified further) when 
$$
\frac{\mu }{\lambda }\le C_1 (\lambda ^{-1}\ln \lambda )^{\frac{2}{3}},
$$ 
for any fixed $C_1>0$,
i.e. when 
\ekv{ex.3}
{
\mu \le C_1\lambda ^{\frac{1}{3}}(\ln \lambda )^{\frac{2}{3}}.
}

M.~Hitrik and K.~Pravda Starov \cite{HiPr07} have obtained interesting results on the characterization of the exponential decay of the the semi-groups generated by differential operators with quadratic symbols.

\section{Survey of some recent results}\label{sur}
\setcounter{equation}{0}

\subsection{Introduction}

In this section we survey some recent results about Kramers-Fokker-Planck type operators and about Bohr-Sommerfeld quantization conditions in dimension 2. In both cases our approach uses quite essentially the possibility of modifying the Hilbert space structure by means of exponential weights. 

In recent years there has been many works that apply the commutator methods developed by Kohn for subelliptic operators to equations of Fokker-Planck type and non-equilibriaum stat physics models. Especially there is the work by F.~H\'erau--F.~Nier \cite{HeNi04} devoted to the Kramers-Fokker-Planck operator 
\ekv{sur.1}
{
P=y\cdot h\partial _x-V'(x)\cdot h\partial _y+
\frac{1}{2}(y-h\partial _y)\cdot (y+h\partial _y),\ (x,y)\in {\bf R}^{2n}={\bf R}_x^n\times {\bf R}^n_y.
}  
Symbol: $p=ip_2+p_1$, where
\ekv{sur.2}
{
p_1(x,y,\xi ,\eta )=\frac{1}{2}(y^2+\eta ^2),\ p_2(x,y,\xi ,\eta )= y\cdot \xi -V'(x)\cdot \eta 
}
Here we have suppressed some physical parameters and only kept $h$ which is proportional to the temperature. Using commutator techniques H\'erau and Nier establish an interesting link to a Witten Laplacian and under various general symbol type assumtions on $V$ they show:
\begin{itemize}
\item global subellipticity ($P$ is not elliptic even locally), and $m$-accretivity,
\item absence of large eigenvalues and corresponding power-decay of the resolvent in certain ``parabolic'' neighborhoods of $i{\bf R}$,
\item Estimates on the smallest non-vanishing eigenvalue (relating it to the corresponding quantity for the Witten laplacian), and especially estimates on this quantity in the high  and low temperature limit.  
\item Precise estimates on the return to equilibrium (when $0$ is an eigenvalue) or simply decay when $0$ is not an eigenvalue, including estimates on the exponential rate of convergence. 
\end{itemize}
With F.~H\'erau and C.~Stolk \cite{HeSjSt05} we applied microlocal methods in order to study the semi-classical (low temperature) limit, 
the results are fewer than the ones in \cite{HeNi04} and for me easier to describe.  

Assume that $V$ is a Morse function, such that
\ekv{sur.3}
{
|V'(x)|\ge 1/C \hbox{ when }|x|>C,\quad \partial ^\alpha V''(x)={\cal O}(1), \forall \alpha \in {\bf N}^n.
}
We showed that $\exists c>0$ such that for any $C\ge 1$ we have have for $h>0$ small enough:
\begin{itemize}
\item The eigenvalues in the disc $D(0,Ch)$ are of the form $\mu =\mu (\lambda )=\lambda h+o(h)$, where $\lambda $ are the  eigenvalues of the quadratic approximation of ${{P}_\vert}_{h=1}$ at the critical points of $V$. 
\item For $|z|\ge Ch$ and $\Re z\le c|z|^{1/3}h^{2/3}$ the resolvent exists and satisfies the estimate 
$$
\| (z-P)^{-1}\|\le \frac{C}{|z|^{1/3}h^{2/3}}.
$$
\end{itemize}

This is similar to Theorem \ref{inbdy2} (B) in the case $k=2$. with an essential difference: $p$ has no values in $i{\bf R}\setminus \{0\}$, so $\partial \Sigma \subset \Sigma _\infty $! Nevertheless we can compute 
$H_{p_2}^2p_1=V'(x)^2+\xi ^2-(V''(x)y\cdot y+V''(x)\eta \cdot \eta )$
and see that when $p_1$ is small, then $H_{p_2}^2p_1>0$ except near points where $\xi =V'(x)=0$. This means that where $p_1$ is very small, the short time averages are not small. This can be exploited by conjugating the operator by an operator which is bounded and has a bounded inverse and for which the conjugated operator has a symbol with a larger real part. 

Below we shall discuss a more general supersymmetric situation and also give detailed results on the return to equilibrium for the associated heat equation, following two joint works with F.~H\'erau and M.~Hitrik.

In the case of the KFP equation we make a moderate change of the Hilbert space norm in order to increase the real part of the operator away from certain critical points of the symbol. When we have analyticity assumptions such changes of the norm can be much larger and sometime allow us to study all eigenvalues also at fixed distance inside the range of $p$. Here is one such result by M.~Hitrik \cite{Hi04}: 

{Let $P=P(x,hD_x;h)$ on ${\bf R}$ satisfying the general conditions of section \ref{bdy}. Also assume that $P$ has a holomorphic extension to a tubular neighborhood of ${\bf R}^2$ in ${\bf C}^2$ which is still ${\cal O}(m(\Re (x,\xi ))$ Let $z_0\in \partial \Sigma \setminus \Sigma _\infty $ be a point such that $p^{-1}(z_0)=\{(x_0,\xi _0)\}$ and such that $|p(x,\xi )-z_0|\asymp |(x-x_0,\xi -\xi _0)|^2$ for $(x,\xi )$ in a neighborhood of $(x_0,\xi _0)$. Also assume that there is a truncated sector $z_0+]0,\epsilon _0]e^{i[\theta _0-\epsilon _0,\theta _0+\epsilon _0]}$ which is disjoint from $\Sigma (p)$.} {Then in a small but fixed neighborhood of $z_0$, the spectrum is given by the set of values $z_0+G((k+\frac{1}{2})h;h)+{\cal O}(h^{\infty })$, where $G(\cdot ;h)$ is holomorphic and $\sim \sum_0^\infty h^jG_j(\cdot ;h)$, $k=0,1,2,...$, where $G_0(0)=0 $, $G_0'(0)\ne 0$.}

 The main idea is to construct
an IR-manifold $\Lambda $ (a complex deformation of real phase space), containing $(x_0,\xi _0)$, such that on $\Lambda $, $p-z_0$ is elliptic outside an arbitrarily
small neighborhood of $(x_0,\xi _0)$ and such that near that point, the restriction of the quadratic
part of $p$ to $\Lambda$  takes its values along a ray in the complex plane. To implement this picture, we make a Bargmann transform, mapping $L^2$ into an exponentially weighted space of holomorphic functions, then deform the weight. 

In the second part of this section we discuss precise Bohr-Sommerfeld rules in dimension 2 for non-self-adjoint operators in the semi-classical limit. What is remarkable here is that thanks to the non-self-adjointness we get better results than what would be possible in the self-adjoint case. Again analyticity assumptions and the use of exponential weights on the Bargmann transform side is essential. We will follow recent works with M.~Hitrik and S.~V\~u Ng\d{o}c \cite{HiSjVu07, HiSj08}

\subsection{Kramers-Fokker-Planck type operators, spectrum and return to
equilibrium}\label{kfp} \setcounter{equation}{0}
\subsubsection{Introduction}
\par There has been a renewed interest in the problem of {``return to
equilibrium''} for various 2nd order operators.  One example is the
Kramers-Fokker-Planck operator: \ekv{0.1kfp} {P= y\cdot h\partial
_x-V'(x)\cdot h\partial _y+{\gamma \over 2}(-h\partial _y+y) \cdot
(h\partial _y+y), } where $x,y\in {\bf R}^n$ correspond respectively
to position and speed of the particles and $h>0$ corresponds to
temperature. The constant $\gamma >0$ is the friction. (Since we will
only discuss $L^2$ aspects we here present right away an adapted
version of the operator, obtained after conjugation by a Maxwellian
factor.)

\par The associated evolution equation is:
$$
(h\partial _t+P)u(t,x,y)=0.
$$
{\it Problem of return to equilibrium:} Study the rate of convergence
of $u(t,x,y)$ to a multiple of the ``ground state''
$u_0(x,y)=e^{-(y^2/2+V(x))/h}$ when $t\to +\infty $, assuming that
$V(x)\to +\infty $ sufficiently fast when $x\to \infty $ so that
$u_0\in L^2({\bf R}^{2n})$. Notice here that $P(u_0)=0$ and that the
vector field part of $P$ is $h$ times the Hamilton field of
$y^2/2+V(x)$, when we identify ${\bf R}_{x,y}^{2n}$ with the cotangent
space of ${\bf R}^n_x$.

\par A closely related problem is to study the difference between the
first eigenvalue (0) and the next one, $\mu (h)$. (Since our operator
is non-self-adjoint, this is only a very approximate formulation
however.)

Some contributions: L.~Desvillettes--C.~Villani \cite{DeVi},
J.P.~Eckmann--M.~Hairer \cite{EcHa}, F.~H\'erau--F.~Nier \cite{HeNi04},
B.~Helffer--F.~Nier \cite{HelNi}, Villani \cite{Vi}. In the work
\cite{HeNi04} precise estimates on the exponential rates of return to
equilibrium were obtained with methods close to those used in
hypoellipticity studies and this work was our starting point. With
H\'erau and C.Stolk \cite{HeSjSt05} we made a study in the
semi-classical limit and studied small eigenvalues modulo ${\cal
O}(h^{\infty })$. More recently with H\'erau and M.~Hitrik
\cite{HeHiSj08a} we have made a precise study of the exponential decay of
$\mu (h)$ when $V$ has two local minima (and in that case $\mu (h)$
turns out to be real). This involves tunneling, i.e. the study of the
exponential decay of eigenfunctions. As an application we have a
precise result on the return to equilibrium \cite{HeHiSj08b}.  This has
many similarities with older work on the tunnel effect for
Schr\"odinger operators in the semi-classical limit by
B.~Helffer--Sj\"ostrand \cite{HeSj84, HelSj2} and B.~Simon \cite{Si84}
but for the Kramers-Fokker-Planck operator the problem is richer and
more difficult since $P$ is neither elliptic nor self-adjoint. We have
used a supersymmetry observation of J.M.~Bismut \cite{Bi05} and
J.~Tailleur--S.~Tanase-Nicola--J.~Kurchan \cite{TaTaKu06}, allowing
arguments similar to those for the standard Witten complex
\cite{HelSj2}.
\subsubsection{Statement of the main results}\label{kfpst}

 Let $P$ be given by (\ref{0.1kfp}) where $V\in C^\infty ({\bf R}^n;{\bf
R})$, and \ekv{1.1} {\partial ^\alpha V(x)={\cal O}(1),\ \vert \alpha
\vert\ge 2,} \ekv{1.2} { \vert \nabla V(x) \vert\ge 1/C,\ \vert x
\vert \ge C, } \ekv{1.3} { V\hbox{ is a Morse function.}  } We also
let $P$ denote the graph closure of $P$ from ${\cal S}({\bf R}^{2n})$
which coincides with the maximal extension of $P$ in $L^2$ (see
\cite{HeNi04, HelNi, HeHiSj08a}).  We have $\Re P\ge 0$ and the spectrum
of $P$ is contained in the right half plane. In \cite{HeSjSt05} the
spectrum in any strip $0\le \Re z\le Ch$ (and actually in a larger
parabolic neighborhood of the imaginary axis, in the spirit of
\cite{HeNi04}) was determined asymptotically ${\rm mod\,}({\cal
O}(h^\infty ))$. It is discrete and contained in a sector $\vert \Im z
\vert\le C\Re z+{\cal O}(h^\infty )$:
\begin{theo}\label{kfp1} The eigenvalues in the strip $0\le \Re
z\le Ch$ are of the form \ekv{1.4} { \lambda
_{j,k}(h)\sim h(\mu _{j,k}+h^{1/N_{j,k}}\mu _{j,k,1}+ h^{2/N_{j,k}}\mu
_{j,k,2}+..)  } where $\mu _{j,k}$ are the eigenvalues of the
quadratic approximation (``non-selfadjoint oscillator'')
$$
y\cdot \partial _x-V''(x_j)x\cdot \partial _y+{\gamma \over
2}(-\partial _y+y)\cdot (\partial _y+y),
$$
at the points $(x_j,0)$, where $x_j$ are the critical points of $V$.
\end{theo}

 The $\mu _{j,k}$ are known explicitly and it follows that when $x_j$
is not a local minimum, then $\Re\lambda _{j,k}\ge h/C$ for some
$C>0$.  When $x_j$ is a local minimum, then precisely one of the
$\lambda_{j,k}$ is ${\cal O}(h^\infty )$ while the others have real
part $\ge h/C$. Furthermore, when $V\to +\infty $ as $x\to \infty $,
then $0$ is a simple eigenvalue. {\it In particular, if $V$ has only
one local minimum, then
$$
\inf \Re (\sigma (P)\setminus \{ 0\})\sim h(\mu _1+h\mu _2+\dots
),\quad \mu _1>0.
$$
}(or possibly an expansion in fractional powers) and we obtained a
corresponding result for the problem of return to equilibrium. It
should be added that when $\mu _{j,k}$ is a simple eigenvalue of the
quadratic approximation then $N_{j,k}=1$ so there are no fractional
powers of $h$ in (\ref{1.4}).

\par The following is the main new result that we obtained with
F.~H\'erau and M.~Hitrik in \cite{HeHiSj08a}:
\begin{theo}\label{kfp2} Assume that $V$ has precisely 3 critical
points; 2 local minima, $x_{\pm 1}$ and one ``saddle point'', $x_0$ of
index 1. Then for $C>0$ sufficiently large and $h$ sufficiently small,
$P$ has precisely 2 eigenvalues in the strip $0\le \Re z\le h/C$,
namely $0$ and $\mu (h)$, where $\mu (h)$ is real and of the form
\ekv{1.5} { \mu (h)=h(a_1(h)e^{-2S_1/h}+a_{-1}(h)e^{-2S_{-1}/h}), }
where $a_j$ are real,
$$
a_j(h)\sim a_{j,0}+ha_{j,1}+...,\ h\to 0,\quad a_{j,0}>0,
$$
$$S_j=V(x_0)-V(x_j).$$\end{theo}

\par As for the problem of return to equilibrium, we obtained the
following result with F.~H\'erau and M.~Hitrik in \cite{HeHiSj08b}:
\begin{theo}\label{kfp2.5} We make the same assumptions as in Theorem
\ref{kfp2} and let $\Pi_j$ be the spectral projection associated with
the eigenvalue $\mu_j$, $j=0,1$, where $\mu _0=0$, $\mu _1=\mu
(h)$. Then we have
\begin{equation}
\label{eq04} \Pi_j={\cal O}(1):L^2\to L^2,\quad h\rightarrow 0.
\end{equation} We have furthermore, uniformly as $t\geq 0$ and
$h\rightarrow 0$,
\begin{equation}
\label{eq05} e^{-tP/h}=\Pi_0+e^{-t\mu_1/h}\Pi_1+{\cal O}(1) e^{-t/C},
\mbox{ in }{\cal L}(L^2,L^2),
\end{equation} where $C>0$ is a constant.
\end{theo}

\par Actually, as we shall see in the outline of the proofs, these
results (as well as (\ref{1.4})) hold for more general classes of
supersymmetric operators.

\par Very recently we observed with F.~H\'erau and M.~Hitrik that we can actually treat the case of any (finite) number of local minima. The basic observation here is that there is a Hermitian product which becomes a scalar product on the space spanned by the $N_0$ lowest eigenvalues, where $N_0$ denotes the number of local minima, and for which our operator is formally self-adjoint. This makes it possible to apply very much the same methods as for the standard Witten complex.

\subsubsection{A partial generalization of \cite{HeSjSt05}}\label{kfpgen}

Consider on ${\bf R}^n$ ($2n$ is now replaced by $n$):
\begin{eqnarray*} P&=&\sum_{j,k}hD_{x_j}b_{j,k}(x)hD_{x_k}+\\
&&{1\over 2}\sum_j(c_j(x)h\partial _{x_j} +h\partial _{x_j}\circ
c_j(x))+p_0(x)\\ &=&P_2+iP_1+P_0,
\end{eqnarray*} where $b_{j,k}, c_j, p_0$ are real and smooth.  The
associated symbols are:
\begin{eqnarray*} p(x,\xi )&=&p_2(x,\xi )+ip_1(x,\xi )+p_0(x),\\
p_2&=&\sum b_{j,k}\xi _j\xi _k,\ p_1=\sum c_j\xi _j.
\end{eqnarray*} Assume,
$$
p_2\ge 0,\ p_0\ge 0,
$$
\begin{eqnarray*}
\partial _x^\alpha b_{j,k}&=&{\cal O}(1),\ \vert \alpha \vert\ge 0,\\
\partial _x^\alpha c_j&=&{\cal O}(1),\ \vert \alpha \vert\ge 1,\\
\partial _x^\alpha p_0&=&{\cal O}(1),\ \vert \alpha \vert\ge 2.
\end{eqnarray*} {Assume that} $$\{x;\,p_0(x)=c_1(x)=..=c_n(x)=0\}$$ is
finite $=\{ x_1,...,x_N\}$ and put ${\cal C}=\{ \rho _1,...,\rho
_n\}$, $\rho _j=(x_j,0)$.  {Put}
$$\widetilde{p}(x,\xi )=\langle \xi  \rangle^{-2}
p_2(x,\xi )+p_0(x),\ \langle \xi \rangle =\sqrt{1+\vert \xi \vert
^2}$$
$$
\langle \widetilde{p} \rangle_{T_0}={1\over T_0}\int_{-T_0/2}^{T_0/2}
\widetilde{p} \circ \exp (tH_{p_1})dt,\ T_0>0 \hbox{ fixed.}
$$
Here in general we let $H_a=a'_\xi \cdot \frac{\partial }{\partial x}
-a'_x \cdot \frac{\partial }{\partial \xi } $ denote the Hamilton
field of the $C^1$-function $a=a(x,\xi )$.

\par Dynamical assumptions: Near each $\rho _j$ we have $\langle
\widetilde{p} \rangle_{T_0}\sim \vert \rho -\rho _j \vert^2$ and in
any compact set disjoint from ${\cal C}$ we have $\langle
\widetilde{p} \rangle_{T_0}\ge 1/C$. (Near infinity this last
assumption has to be modified slightly and we refer to \cite{HeHiSj08a}
for the details.) The following result from \cite{HeHiSj08a} is very
close to the main result of \cite{HeSjSt05} and generalizes Theorem
\ref{kfp1}:

\begin{theo} \label{kfp3} Under the above assumptions, the spectrum
of $P$ is discrete in any band $0\le \Re z\le Ch$ and the eigenvalues
have asymptotic expansions as in (\ref{1.4}).\end{theo}

\par Put $$q(x,\xi )=-p(x,i\xi )=p_2(x,\xi )+p_1(x,\xi )-p_0(x).$$

\par The linearization of the Hamilton field $H_q$ at $\rho _j$ (for
any fixed $j$) has eigenvalues $\pm \alpha _k$, $k=1,..,n$ with real
part $\ne 0$.  Let $\Lambda _+=\Lambda _{+,j}$ be the unstable
manifold through $\rho _j$ for the $H_q$-flow. Then $\Lambda _+$ is
Lagrangian and of the form $\xi =\phi '_+(x)$ near $x_j$ ($\phi_+
=\phi _{+,j}$), where
$$\phi_+ (x_j)=0,\ \phi_+' (x_j)=0,\ 
\phi_+'' (x_j)>0.$$ The next result is from \cite{HeHiSj08a}:
\begin{theo}\label{kfp4} Let $\lambda _{j,k}(h)$ be a simple
eigenvalue as in (\ref{1.4}) and assume there is no other eigenvalue
in a disc $D(\lambda _{j,k},h/C)$ for some $C>0$.  Then, in the $L^2$ sense, the
corresponding eigenfunction is of the form $e^{-\phi
_+(x)/h}(a(x;h)+{\cal O}(h^\infty ))$ near $x_j$, where $a(x;h)$ is
smooth in $x$ with an asymptotic expansion in powers of $h$. Away from
a small neighborhood of $x_j$ it is exponentially decreasing.
\end{theo}

The proof of the first theorem uses microlocal weak exponential
estimates, while the one of the last theorem also uses local
exponential estimates.

\subsubsection{Averaging and exponential weights.}\label{kfpav}

The basic idea of the proof of Theorem \ref{kfp3} is taken from
\cite{HeSjSt05}, but we reworked it in order to allow for
non-hypoelliptic operators. We will introduce a weight on $T^*{\bf
R}^n$ of the form \ekv{pR.1} { \psi _\epsilon =-\int
J(\frac{t}{T_0})\widetilde{p}_\epsilon \circ \exp (tH_{p_1}) dt, } for
$0<\epsilon \ll 1$. Here $J(t)$ is the {odd} function given by
\ekv{pR.2} { J(t)= \left\{ \begin{array}{ll} & 0,\ \vert t\vert \ge
\frac{1}{2}, \\ & \frac{1}{2}-t,\ 0<t\le \frac{1}{2},
\end{array} \right.  } and we choose $\widetilde{p}_\epsilon (\rho )$
to be equal to $\widetilde{p}(\rho )$ when ${\rm dist\,}(\rho ,{\cal
C})\le \epsilon $, and flatten out to $\epsilon \widetilde{p}$ away
from a fixed neighborhood of ${\cal C}$ in such a way that
$\widetilde{p}_\epsilon ={\cal O}(\epsilon )$.  Then \ekv{pR.3} {
H_{p_1}\psi _\epsilon =\langle \widetilde{p}_\epsilon \rangle _{T_0}
-\widetilde{p}_\epsilon .  }

We let $\epsilon =Ah$ where $A\gg 1$ is independent of $h$. {Then the
weight $\exp (\psi _\epsilon /h)$ is uniformly bounded when $h\to 0$.}
Indeed, $\psi _\epsilon ={\cal O}(h)$.

\par Using Fourier integral operators with complex phase, we can
define a Hilbert space of functions that are ``microlocally ${\cal
O}(\exp (\psi _\epsilon /h))$ in the $L^2$ sense''. The norm is
uniformly equivalent to the one of $L^2$, but the natural leading
symbol of $P$, acting in the new space, becomes { \ekv{pR.4}{p(\exp
(iH_{\psi_\epsilon })(\rho )),\ \rho \in T^*{\bf R}^n} }which by
Taylor expansion has {real part $\approx p_2(\rho )+p_0(\rho )+\langle
\widetilde{p}_\epsilon \rangle-\widetilde{p}_\epsilon $.}
\par Very roughly, {the real part of the new symbol is $\ge \epsilon $
away from ${\cal C}$ and behaves like ${\rm dist\,}(\rho ,{\cal C})^2$
in a $\sqrt{\epsilon }$-neighborhood of ${\cal C}$.} This can be used
to show that the spectrum of $P$ (viewed as an operator on the
weighted space) in a band $0\le \Re z <\epsilon /C $ comes from an
$\sqrt{\epsilon }$-neighborhood of ${\cal C}$. In such a neighborhood,
we can treat $P$ as an elliptic operator and the spectrum is to
leading order determined by the quadratic approximation of the dilated
symbol (\ref{pR.4}). This gives Theorem \ref{kfp3}.

We next turn to the proof of Theorem \ref{kfp4}, and we work near a
point $\rho _j=(x_j,\xi _j)\in {\cal C}$. Recall that $\Lambda _+:\xi
=\phi _+'(x)$ is the unstable manifold for the $H_q$-flow, where
$q(x,\xi )=-p(x,i\xi )$. We have {$q(x,\phi _+'(x))=0$}.

\par In general, if $\psi \in C^\infty $ is real, then $P_\psi
:=e^{\psi /h}\circ P\circ e^{-\psi /h}$ has the symbol \ekv{pR.5} {
p_\psi (x,\xi )=p_2(x,\xi ){-q(x,\psi '(x))}+i(q'_\xi (x,\psi
'(x))\cdot \xi }

\begin{itemize}
\item As long as $q(x,\psi '(x))\le 0$, we have $\Re p_\psi \ge 0$ and
we may hope to establish good apriori estimates for $P_\psi $.
\item This is the case for $\psi =0$ and for $\psi =\phi _+$. Using
the convexity of $q(x,\cdot )$, we get suitable weights $\psi $ with
$q(x,\psi '(x))\le 0$, equal to $\phi _+(x)$ near $x_j$, strictly
positive away from $x_j$ and constant outside a neighborhood of that
point.
\item It follows that the eigenfunction in Theorem \ref{kfp4} is
(roughly) ${\cal O}(e^{-\phi _+(x)/h})$ near $x_j$ in the $L^2$ sense.
\item On the other hand, we have quasi-modes of the form
$a(x;h)e^{-\phi _+(x)/h}$ as in \cite{HeSj84}.
\item Applying the exponentially weighted estimates, indicated above,
to the difference of the eigenfunction and the quasi-mode, we then get
Theorem \ref{kfp4}.
\end{itemize}

\subsubsection{Supersymmetry and the proof of Theorem
\ref{kfp2}}\label{kfpss}

\par We review the supersymmetry from \cite{Bi05}, \cite{TaTaKu06}, see
also G.~Lebeau \cite{Le04}.  Let $A(x):T_x^*{\bf R}^n\to T_x{\bf R}^n$
be linear, invertible and smooth in $x$. Then we have the
nondegenerate bilinear form
$$
\langle u\vert v \rangle_{A(x)}=\langle \wedge^kA(x)u|v \rangle,\
u,v\in\wedge^kT_x^*{\bf R}^n,
$$
and we also write $(u|v)_{A(x)}=\langle u|\overline{v}
\rangle_{A(x)}$.

\par If $u,v$ are smooth $k$-forms with compact support, put
$$
(u|v)_A=\int (u(x)|v(x))_{A(x)}dx.
$$
The formal ``adjoint'' $Q^{A,*}$ of an operator $Q$ is then given by
$$
(Qu|v)_A=(u|Q^{A,*}v)_A.
$$

 Let $\phi :{\bf R}^n\to {\bf R}$ be a smooth Morse function with
$\partial ^\alpha \phi $ bounded for $\vert \alpha \vert\ge 2$ and
with $|\nabla \phi |\ge 1/C$ for $|x|\ge C$. Introduce the
Witten-De Rham complex:
$$
d_\phi =e^{-{\phi \over h}}\circ hd\circ e^{{\phi \over h}}= \sum_j
(h\partial _{x_j}+\partial _{x_j}\phi )\circ dx_j^\wedge ,
$$
where $d$ denotes exterior differentiation and $dx_j^\wedge$ left
exterior multiplication with $dx_j$.  The corresponding Laplacian is
then: $-\Delta_A=d_\phi ^{A,*}d_\phi +d_\phi d_\phi^{A,*}$.  Its
restriction to $q$-forms will be denoted by $-\Delta _A^{(q)}$. Notice
that:
$$-\Delta ^{(0)}_A(e^{-\phi /h})=0.$$

Write $A=B+C$ with $B^t=B$, $C^t=-C$. $-\Delta_A$ is a second order
differential operator with scalar principal symbol in the
semi-classical sense (${h\over i}{\partial \over \partial x_j}\mapsto
\xi _j$) of the form:
$$p(x,\xi )=\sum_{j,k}b_{j,k}(\xi _j\xi _k+\partial _{x_j}\phi
\partial _{x_k}\phi )+2i\sum_{j,k}c_{j,k}\partial _{x_k}\phi \,\xi _j.
$$
\par\noindent {\bf Example.} Replace $n$ by $2n$, $x$ by $(x,y)$, let
$$
A={1\over 2}\left(\begin{array}{ccc}0 &I\\ -I
&\gamma \end{array}\right).
$$
Then
\begin{eqnarray*} -\Delta _A^{(0)}&=&h(\phi '_y\cdot \partial _x-\phi
'_x\cdot \partial _y) \\ &&+ {\gamma \over 2}\sum_{j}(-h\partial
_{y_j}+\partial _{y_j}\phi ) (h\partial _{y_j}+\partial _{y_j}\phi ).
\end{eqnarray*} When $\phi ={y^2/2}+V(x)$ we recover the KFP operator
(\ref{0.1kfp})

\par The results of Subsubsection \ref{kfpgen} apply, if we make the
additional dynamical assumptions there; $-\Delta ^{(q)}_A$ has an
asymptotic eigenvalue $=o(h)$ associated to the critical point $x_j$
precisely when the index of $x_j$ is equal to $q$ (as for the Witten
complex and analogous complexes in several complex variables). In
order to cover the cases $q>0$ we also assume that \ekv{3.1}{A={\rm
Const.}}  {\it The Double well case.}  Keep the assumption
(\ref{3.1}).  Assume that $\phi $ is a Morse function with $|\nabla
\phi |\ge 1/C$ for $|x|\ge C$ such that $-\Delta _A$ satisfies the
extra dynamical conditions of Subsubsection \ref{kfpgen} and having
precisely three critical points, two local minima $U_{\pm 1}$ and a
saddle point $U_0$ of index 1.

\par Then $-\Delta _A^{(0)}$ has precisely 2 eigenvalues: $0,\, \mu $
that are $o(h)$ while $-\Delta _A^{(1)}$ has precisely one such
eigenvalue: $\mu $. (Here we use as in the study of the Witten
complex, that $d_\phi $ and $d_{\phi }^{A,*}$ intertwine our
Laplacians in degeree 0 and 1. The detailed justification is more
complicated however.)  $e^{-\phi /h}$ is the eigenfunction of $\Delta
_A^{(0)}$ corresponding to the eigenvalue 0. Let $S_j=\phi (U_0)-\phi
(U_j)$, $j=\pm 1$, and let $D_j$ be the connected component of $\{
x\in {\bf R}^n;\, \phi (x) < \phi (U_0)\}$ containing $U_j$ in its
interior.

\par

\par Let $E^{(q)}$ be the corresponding spectral subspaces so that
${\rm dim\,}E^{(0)}=2$, ${\rm dim\,}E^{(1)}=1$. Truncated versions of
the function $e^{-\phi (x)/h}$ can be used as approximate
eigenfunctions, and we can show:
\begin{prop}\label{kfp5} $E^{(0)}$ has a basis $e_1,e_{-1}$, where
$$e_j=\chi_j(x)e^{-{1\over h}(\phi (x)-\phi (U_j))}+{\cal
O}( e^{-{1\over h}(S_j-\epsilon )}),\mbox{ in the $L^2$-sense.} $$ Here,
we let $\chi _j\in C_0^\infty (D_j)$ be equal to 1 on $\{ x\in D_j;\,
\phi (x)\le \phi (U_0)-\epsilon \}$.\end{prop} The theorems
\ref{kfp3}, \ref{kfp4} can be adapted to $-\Delta _A^{(1)}$ and lead
to:
\begin{prop}\label{kfp6} $E^{(1)}={\bf C}e_0$, where
$$e_0(x)=\chi _0(x)a_0(x;h)e^{-{1\over h}\phi _+(x)}+{\cal O}
(e^{-\epsilon _0/h}),$$ $\phi _+(x)\sim (x-U_0)^2,$ $\epsilon _0>0$ is
small enough, $a_0$ is an elliptic symbol, $\chi_0\in C_0^\infty ({\bf
R}^n)$, $\chi_0=1$ near $U_0$.
\end{prop}

\par Let the matrices of $d_\phi :E^{(0)}\to E^{(1)}$ and
$d_\phi^{A,*} :E^{(1)}\to E^{(0)}$ with respect to the bases $\{
e_{-1}, e_1\}$ and $\{ e_0\}$ be
$$
\left(\begin{array}{ccc} \lambda_{-1} & \lambda_{1}\end{array}\right)
\hbox{ and } \left (\begin{array}{ccc}\lambda _{-1}^* \\ \lambda
_1^*\end{array}\right)\hbox{ respectively}.
$$
Using the preceding two results in the spirit of tunneling estimates
and computations of Helffer--Sj\"ostrand (\cite{HeSj84, HelSj2}) we
can show:

\begin{prop}\label{kfp7} Put $S_j=\phi (U_0)-\phi (U_j)$, $j=\pm
1$. Then we have
$$
\left(\begin{array}{ccc} \lambda _{-1}\\ \lambda _1\end{array}\right)
=h^{1\over 2}(I+{\cal O}(e^{-{1\over Ch}})) \left(\begin{array}{ccc}
\ell_{-1}(h)e^{-S_{-1}/h}\\ \ell_1(h)e^{-S_1/h}
\end{array}\right),
$$

$$
\left(\begin{array}{ccc} \lambda ^*_{-1}\\ \lambda
_1^*\end{array}\right)=h^{1\over 2}(I+{\cal O}(e^{-{1\over Ch}}))
\left(\begin{array}{ccc} \ell_{-1}^*(h)e^{-S_{-1}/h}\\
\ell_1^*(h)e^{-S_1/h}
\end{array}\right),
$$
where $\ell_{\pm 1}$, $\ell_{\pm 1}^*$ are real elliptic symbols of
order $0$ such that $\ell_{j}\ell_j^*>0$, $j=\pm 1$.
\end{prop}

\par From this we get Theorem \ref{1.2}, since $\mu =\lambda
^*_{-1}\lambda _{-1}+ \lambda _1^*\lambda
_1$.\hfill{$\square$}\medskip

Thanks to the fact that we have only two local minima, certain
simplifications were possible in the proof. In particular it was
sufficent to control the exponential decay of general eigenfunctions
in some small neighborhood of the critical points.

\subsubsection{Return to equilibrium, ideas of the proof of Theorem
\ref{kfp2.5}}
\label{kfpre}

 Keeping the same assumptions, let $\Pi_0,\,\Pi _1 $ be the rank 1
spectral projections corresponding to the eigenvalues $\mu _0:=0,\,
\mu _1:=\mu $ of $-\Delta _A^{(0)}$ and put $\Pi =\Pi _0+\Pi _1$. Then
$e_{-1,},e_1$ is a basis for ${\cal R}(\Pi )$ and the restriction of
$P$ to this range, has the matrix \ekv{4.1} {
\left(\begin{array}{ccc}\lambda _{-1}^*\\ \lambda _1^*
\end{array}\right) \left(\begin{array}{ccc} \lambda _{-1} &\lambda
_1\end{array}\right) = \left(\begin{array}{ccc} \lambda _{-1}^*\lambda
_{-1} & \lambda _{-1}^*\lambda _{1}\\ \lambda _{1}^*\lambda _{-1}&
\lambda _{1}^*\lambda _{1}
\end{array}\right) } with the eigenvalues $0$ and $\mu =\lambda
_{-1}^*\lambda _{-1}+ \lambda _{1}^*\lambda _{1}$. A corresponding
basis of eigenvectors is given by
\begin{eqnarray}\label{4.2} v_0&=&\frac{1}{\sqrt{\mu _1}}(\lambda
_1e_{-1}-\lambda _{-1}e_{-1}) \\ v_1&=&\frac{1}{\sqrt{\mu _1}}(\lambda
_{-1}^*e_{-1}+\lambda _{1}^* e_{-1}).\nonumber
\end{eqnarray} The corresponding dual basis of eigenfunctions of $P^*$
is given by \eekv{4.3} { v_0^*&=&\frac{1}{\sqrt{\mu }}(\lambda
_1^*e_{-1}^*- \lambda _{-1}^*e_{-1}^*) }
{v_1^*&=&\frac{1}{\sqrt{\mu}}(\lambda _{-1}e_{-1}^*+\lambda _{1}
e_{1}^*),} where $e_{-1}^*,e_1^* \in {\cal R}(\Pi ^*)$ is the basis
that is dual to $e_{-1},e_1$. It follows that $v_j,v_j^*={\cal O}(1)$
in $L^2$, when $h\to 0$.
\par From this discussion we conclude that $\Pi _j=(\cdot |v_j^*)v_j$,
\emph{are uniformly bounded when $h\to 0$}.  A non-trivial fact, based
on the analysis described in Subsubsections \ref{kfpgen}, \ref{kfpav}, is
that after replacing the standard norm and scalar product on $L^2$ by
certain uniformly equivalent ones, we have \ekv{4.4} {\Re (Pu|u)\ge
\frac{h}{C}\Vert u\Vert^2,\quad \forall u\in {\cal
R}(1-\widetilde{\Pi} ),} where $\widetilde{\Pi }$ is the spectral
projection corresponding to the spectrum of $P$ in $D(0,Bh)$ for some
$B\gg 1$.

This can be applied to the {study of $u(t):=e^{-tP/h}u(0)$,} where the
initial state $u(0)\in L^2$ is arbitrary: {Write \ekv{4.5} { u(0)=\Pi
_0u(0)+\Pi _1u(0)+(1-\Pi )u(0)=:u^0+u^1+u^\perp .  }Then
\begin{eqnarray}\label{11} \Vert u^0 \Vert, \Vert u^1 \Vert, \Vert
u^\perp\Vert&\le& {\cal O}(1) \Vert u(0)\Vert \\ \Vert
e^{-tP/h}u^\perp\Vert&\le& Ce^{-t/C}\Vert u(0)\Vert \label{12}\\
e^{-tP/h}u_j&=&e^{-t\mu _j/h}u_j,\ j=0,1. \label{13}
\end{eqnarray}} Here (\ref{12}) follows if we write $u^\perp=
(1-\widetilde{\Pi })u+(\Pi -\widetilde{\Pi })u$, apply (\ref{4.4}) to
the evolution of the first term, and use that the last term is the
(bounded) spectral projection of $u$ to a finite dimensional spectral
subspace of $P$, for which the corresponding eigenvalues all have real
part $\ge h/C$.\hfill{$\Box$}

\subsection{Spectral asymptotics in 2 dimensions}\label{2d}
\setcounter{equation}{0}

\subsubsection{Introduction}\label{2dint}

This subsection is mainly based on recent joint works with S.~V{\~ u} Ng{\d o}c and M.~Hitrik \cite{HiSj08} \cite{HiSjVu07}, but we shall start by recalling some earlier results that we obtained with A.~Melin \cite{MeSj03} where we discovered that in the two dimensional case one often can have Bohr-Sommerfeld conditions to determine all the individual eigenvalues in some region of the spectral plane, provided that we have analyticity. In the self-adjoint case such results are known (to the author) only in 1 dimension and in very special cases for higher dimensions.

Subsequently, with M. Hitrik we have studied small perturbations of
self-adjoint operators. First we studied the case when the classical
flow of the unperturbed operator is periodic, then also with S.~V{\~
u} Ng{\d o}c we looked at the more general case when it is completely
integrable, or just when the energy surface contains some invariant
Diophantine Lagrangian tori.

\subsubsection{Bohr-Sommerfeld rules in two dimensions}\label{2dbs}

For (pseudo-)differential operators in dimension 1, we often have a
Bohr-Sommerfeld rule to determine the asymptotic behaviour of the
eigenvalues. Consider for instance the semi-classical Schr\"odinger
operator
$$
P=-h^2\frac{d^2}{dx^2}+V(x),\mbox{ with symbol } p(x,\xi )=\xi
^2+V(x),
$$
where we assume that $V\in C^\infty ({\bf R};{\bf R})$ and $V(x)\to
+\infty $, $|x|\to \infty $. Let $E_0\in {\bf R}$ be a non-critical
value of $V$ such that (for simplicity) $\{ x\in {\bf R}; V(x)\le
E_0\}$ is an interval. Then in some small fixed neighborhood of $E_0$
and for $h>0$ small enough, the eigenvalues of $P$ are of the form
$E=E_k$, $k\in {\bf Z}$, where
$$
\frac{I(E)}{2\pi h} = k-\theta (E; h), \quad I(E)=\int_{p^{-1}(E)}\xi
\cdot dx, \ \theta (E;h)\sim \theta _0(E)+\theta _1(E)h+...$$ In the
non-self-adjoint case we get the same results, provided that $\Im V$
is small and $V$ {\it is analytic}. The eigenvalues will then be on a
curve close to the real axis.

\par
For self-adjoint operators in
dimension $\ge 2$ it is generally admitted that Bohr-Sommerfeld rules
do not give all eigenvalues in any fixed domain except in certain
(completely integrable) cases. Using the KAM theorem one can sometimes
describe some fraction of the eigenvalues.

\par With A.~Melin \cite{MeSj03}: we considered an
$h$-pseudodifferential operator with leading symbol $p(x, \xi )$ that
is bounded and holomorphic in a tubular neighborhood of ${\bf R}^4$ in
${\bf C}^4 = {\bf C}^2_x\times{\bf C}^2_\xi $.  Assume that
\ekv{bs.1}{{\bf R}^4 \cap p^{-1}(0)\ne \emptyset \mbox{ is
connected.}}  \ekv{bs.2}{\mbox{On } {\bf R}^4\mbox{ we have }|p(x, \xi
)| \ge 1/C,\mbox{ for }|(x, \xi )| \ge C,} for some $C > 0$,
\ekv{bs.3}{d\Re p(x, \xi ), d \Im p(x, \xi ) \mbox{ are linearly
independent for all } (x, \xi ) \in p^{-1}(0) \cap {\bf R}^4.}  (Here
the boundedness assumption near $\infty $ and (\ref{bs.2}) can be
replaced by a suitable ellipticity assumption.)  It follows that
$p^{-1}(0)\cap {\bf R}^4$ is a compact (2-dimensional) surface.

\par Also assume that \ekv{bs.4}{|\{ \Re p, \Im p\} |\mbox{ is
sufficiently small on } p^{-1}(0) \cap {\bf R}^4.}  
``Sufficiently small'' here refers to some positive bound that can be
defined whenever the the other conditions are satisfied uniformly.

When the Poisson bracket vanishes on $p^{-1}(0)$, this set becomes a
Lagrangian torus, and more generally it is a torus. The following is a
complex version of the KAM theorem without small divisors (cf
T.W.~Cherry \cite{Ch}, J.~Moser \cite{Mo}),
\begin{theo}\label{bs1} (\cite{MeSj03}) There exists a smooth
2-dimensional torus $\Gamma \subset p^{-1}(0)\cap {\bf C}^4$, close to
$p^{-1}(0)\cap {\bf R}^4$ such that ${{\sigma }_\vert}_{\Gamma } = 0$
and $I_j(\Gamma )\in {\bf R},$ $j = 1, 2$. Here $ I_j(\Gamma ):=
\int_{\gamma _j} \xi \cdot dx$ are the actions along the two
fundamental cycles $\gamma _1,\gamma _2\subset \Gamma $, and $\sigma =
\sum_1^2 d\xi _j \wedge dx_j$ is the complex symplectic (2,0)-form.
\end{theo}

Replacing $p$ by $p-z$ for $z$ in a neighborhood of $0\in {\bf C}$, we
get tori $\Gamma (z)$ depending smoothly on $z$ and a corresponding
smooth action function $I(z)=(I_1(\Gamma (z)),I_2(\Gamma (z)))$, which
are important in the Bohr-Sommerfeld rule for the eigen-values near
$0$ in the semi-classical limit $h\to 0$:
\begin{theo}\label{bs2} (\cite{MeSj03}) Under the above assumptions,
there exists $\theta_0 \in ( \frac{1}{2} {\bf Z})^2$ and $\theta (z;
h) \sim \theta_ 0 + \theta_ 1(z)h + \theta_ 2(z)h^2 + ..$ in $C^\infty
( \mathrm{neigh\,} (0,{\bf C}))$, such that for $z$ in an
$h$-independent neighborhood of 0 and for $h > 0$ sufficiently small,
we have that $z$ is an eigenvalue of $P=p(x,hD_x)$ iff
$$\frac{I(z)}{2\pi h}
= k-\theta (z; h),\mbox{ for some }k ∈\in {\bf Z}^2.\quad (BS)$$
\end{theo}

Recently, a similar result was obtained by S.~Graffi, C.~Villegas Bas
\cite{GrVi}.

\par An application of this result is that we get all resonances
(scattering poles) in a {fixed} neighborhood of $0\in {\bf C}$ for
$-h^2\Delta +V(x)$ if V is an analytic real potential on ${\bf R}^2$
with a nondegenerate saddle point at $x=0$, satisfying $V(0)=0$ and
having $\{ (x,\xi )=(0,0)\}$ as its classically trapped set in the
energy surface $\{ p(x,\xi )=0\}$.

\medskip

\subsubsection{Diophantine case}\label{2ddi}

In this and the next subsubsection we describe a result from
\cite{HiSjVu07} and the main result of \cite{HiSj08} about individual
eigenvalues for small perturbations of a self-adjoint operator with a
completely integrable leading symbol. We start with the case when only
Diophantine tori play a role.

Let $P_\epsilon (x,hD;h)$ on ${\bf R}^2$ have the leading symbol
$p_\epsilon (x,\xi )=p(x,\xi )+i\epsilon q(x,\xi )$ where $p$, $q$ are
real and extend to bounded holomorphic functions on a tubular
neighborhood of ${\bf R}^4$. 
Assume that $p$ fulfills the ellipticity condition (\ref{bs.2})
near infinity and that \ekv{bs.5}{P_{\epsilon =0}=P(x,hD)} is
self-adjoint. (The conditions near infinity can be
modified and we can also replace ${\bf R}^2_x$ by a compact
2-dimensional analytic manifold.)

\par Also, assume that $P_\epsilon (x,\xi ;h)$ depends smoothly on
$0\le \epsilon \le \epsilon _0$ with values in the space of bounded
holomorphic functions in a tubular neighborhood of ${\bf R}^4$, and
$P_\epsilon \sim p_\epsilon +hp_{1,\epsilon} +h^2p_{2,\epsilon }+...$,
when $h\to 0$.

Assume \ekv{bs.6}{p^{-1}(0) \mbox{ is connected and }dp\ne 0 \mbox{ on
that set.}}

Assume complete integrability for $p$: There exists an analytic real
valued function $f$ on $T^*{\bf R}^2$ such that $H_pf = 0$, with the
differentials $df$ and $dp$ being linearly independent almost
everywhere on $p^{-1}(0)$. ($H_p=p'_\xi \cdot \frac{\partial }{\partial x}-p'_x \cdot
\frac{\partial }{\partial \xi }$ is the Hamilton field.)

\par Then we have a disjoint union decomposition
\ekv{bs.7}{p^{-1}(0)\cap T^*{\bf R}^2 = \bigcup_{\Lambda \in J}\Lambda
,} where $\Lambda $ are compact connected sets, invariant under the
$H_p$ flow.  We assume (for simplicity) that $J$ has a natural
structure of a graph whose edges correspond to families of regular
leaves; Lagrangian tori (by the Arnold-Mineur-Liouville theorem \cite{Vu06}).  The
union of edges $J\setminus S$ possesses a natural real analytic
structure.

Each torus $\Lambda \in J\setminus S$ carries real analytic
coordinates $x_1, x_2$ identifying $\Lambda $ with ${\bf T}^2 = {\bf
R}^2/2\pi {\bf Z}^2$, so that along $\Lambda $, we have \ekv{bs.8}{H_p
= a_1\frac{\partial }{\partial x_1} + a_2\frac{\partial }{\partial
x_2},} where $a_1, a_2 \in {\bf R}$. The rotation number is defined as
the ratio $\omega (\Lambda ) = [a_1 : a_2] \in {\bf R}{\bf P}^1$, and
it depends analytically on $\Lambda \in J\setminus S$. We assume that
$\omega (\Lambda )$ is not identically constant on any open edge.

\par We say that $\Lambda \in J\setminus S$ is respectively rational,
irrational, diophantine if $a_1/a_2$ has the corresponding
property. Diophantine means that there exist $\alpha >0$, $d>0$ such
that \ekv{bs.15} { |(a_1,a_2)\cdot k|\ge \frac{\alpha }{|k|^{2+d}},\
0\ne k\in {\bf Z}^2, }

\par We introduce \ekv{bs.9}{\langle q\rangle_T =
\frac{1}{T}\int_{-T/2}^{T/2} q \circ \exp (tH_p) dt,\ T > 0, } and
consider the compact intervals $Q_\infty (\Lambda ) \subset {\bf R}$,
$\Lambda \in J$, defined by, \ekv{bs.10}{Q_\infty (\Lambda ) = [\lim_
{T\to →\infty} \inf_ {\Lambda } \langle q\rangle _T , \lim_ {T\to
\infty} \sup_ {\Lambda} \langle q\rangle _T ] .}

\par { A first localization of the spectrum $\sigma (P_\epsilon
(x,hD_x;h))$ (\cite{HiSjVu07}) is given by \ekv{bs.14} { \Im (\sigma
(P_\epsilon ) \cap \{z ; |\Re z|\le \delta \} ) \subset \epsilon [\inf
\bigcup_ {\Lambda \in J} Q_\infty (\Lambda )-o(1), \sup \bigcup_
{\Lambda \in J} Q_\infty (\Lambda ) + o(1) ] , } when $\delta
,\epsilon, h \to 0$.}

\par For each torus $\Lambda \in J\setminus S$, we let {$\langle
q\rangle (\Lambda )$} be the average of ${{q}_\vert}_{\Lambda }$ with
respect to the natural smooth measure on $\Lambda $, and assume that
the analytic function $J\setminus S ∋\ni \Lambda \mapsto \langle
q\rangle (\Lambda )$ is not identically constant on any open edge.

By
combining (\ref{bs.8}) with the Fourier series representation of $q$,
we see that
when $\Lambda$ is irrational then $Q_\infty (\Lambda ) = \{\langle
q\rangle (\Lambda )\}$, while 
in the rational case,    
\ekv{bs.11}{Q_\infty (\Lambda ) \subset \langle q\rangle
(\Lambda ) + {\cal O}( \frac{1} {(|n|+|m|)^\infty}) [ -1, 1 ],}
when $\omega (\Lambda ) = \frac{m}{n}$ and $m \in {\bf Z}$, $n \in {\bf N}$ are relatively prime.

\par{Let $F_0\in \cup_{\Lambda \in J}Q_\infty (\Lambda )$ and assume
that there exists a Diophantine torus $\Lambda _d$ (or finitely many),
such that \ekv{bs.15.5}{ \langle q\rangle (\Lambda _d)=F_0, \quad
d_\Lambda \langle q\rangle (\Lambda _d)\ne 0.}} With M.~Hitrik and
S.~V{\~u} Ng{\d o}c we obtained:

\begin{theo}\label{bs3}(\cite{HiSjVu07}) Assume also that $F_0$ does not
belong to $Q_\infty (\Lambda )$ for any other $\Lambda \in J$. Let
$0<\delta <K<\infty $. Then $\exists C>0$ such that for $h>0$ small
enough, and $k^K\le \epsilon \le h^\delta $, the eigenvalues of
$P_\epsilon $ in the rectangle $ |\Re z|<h^\delta /C,\ |\Im z-\epsilon
\Re F_0|<\epsilon h^\delta /C $ are given by
$$P^{(\infty )}(h(k-\frac{k_0}{4})-\frac{S}{2\pi },\epsilon ;h)+{\cal
  O}(h^\infty ),\ k\in {\bf Z}^2,$$ Here $P^{(\infty )}(\xi ,\epsilon
;h)$ is smooth, real-valued for $\epsilon =0$ and when $h\to 0$ we
have \ekv{bs.16}{P^{(\infty )}(\xi ,\epsilon ;h)\sim \sum_{\ell
=0}^\infty h^\ell p_\ell^{(\infty )}(\xi ,\epsilon ),\ p_0^{(\infty
)}=p_\infty (\xi )+i\epsilon \langle q\rangle (\xi )+{\cal O}(\epsilon
^2),} corresponding to action angle coordinates.
\end{theo}

In \cite{HiSjVu07} we also considered applications to small
non-self-adjoint perturbations of the Laplacian on a surface of
revolution. Thanks to (\ref{bs.11}) the total measure of the union of
all $Q_\infty (\Lambda )$ over the rational tori is finite and
sometimes small, and we could then show that there are plenty of
values $F_0$, fulfilling the assumptions in the theorem.

With M.~Hitrik we are currently studying the distribution of eigenvalues in sub-bands that are delimited by two different values ``$F_0$"
as in the theorem.

\subsubsection{The case with rational tori}\label{2drt}

{ Let $F_0$ be as in (\ref{bs.15.5}) but now also allow for the
possibility that there is a rational torus (or finitely many) $\Lambda
_r$, such that \ekv{bs.17} { F_0\in Q_\infty (\Lambda _r),\quad
F_0\ne\langle q\rangle (\Lambda _r), } \ekv{bs.17.5} {d_\Lambda
(\langle q\rangle)(\Lambda _r)\ne 0,\ d_\Lambda (\omega )(\Lambda
_r)\ne 0.  } { Assume also that \ekv{bs.18} { F_0\not\in Q_\infty
(\Lambda ),\hbox{ for all }\Lambda \in J\setminus \{ \Lambda _d,
\Lambda _r \} .  }}}
With M.~Hitrik we showed the
following result:

\begin{theo}\label{bs4}(\cite{HiSj08}) Let $\delta >0$ be small and
assume that $h\ll \epsilon \le h^{\frac{2}{3}+\delta }$, or that the
subprincipal symbol of $P$ vanishes and that $h^2 \ll \epsilon \le
h^{\frac{2}{3}+\delta }$ . Then the spectrum of $P_\epsilon $ in the
rectangle
$$
[-\frac{\epsilon }{C},\frac{\epsilon }{C}]+i\epsilon
[F_0-\frac{\epsilon ^\delta }{C},F_0+\frac{\epsilon ^\delta }{C}]
$$
is the union of two sets: $E_d \cup E_r$, where the elements of $E_d$
form a distorted lattice, given by the Bohr-Sommerfeld rule
(\ref{bs.16}), with horizontal spacing $\asymp h$ and vertical spacing
$\asymp \epsilon h$. The number of elements $\# (E_r)$ of $E_r$ is
${\cal O}({\epsilon ^{3/2}}/{h^2})$.
\end{theo} NB that $\# (E_d) \asymp \epsilon ^{1+\delta }/h^2$.

This result can be applied to the damped wave equation on surfaces of
revolution.

\subsubsection{Outline of the proofs of Theorem \ref{bs3} and \ref{bs4}}
\label{2dpr}

The principal symbol of $P_\epsilon $ is $p_\epsilon =p+i\epsilon
q+{\cal O}(\epsilon ^2)$. Put
$$
\langle q\rangle_T=\frac{1}{T}\int_{-T/2}^{T/2}q\circ \exp (tH_p)dt.
$$
As in Subsection \ref{kfp} we will use an averaging of the imaginary part
of the symbol.  Let $J(t)$ be the piecewise affine function with
support in $[-\frac{1}{2},\frac{1}{2}]$, solving
$$
J'(t)=\delta (t)-1_{[-\frac{1}{2},\frac{1}{2}]}(t),
$$
and introduce the weight
$$G_T(t)=\int J(-\frac{t}{T})q\circ \exp (tH_p)dt.$$
Then $H_pG_T=q-\langle q\rangle_T$, implying \ekv{2dpr.1} {p_\epsilon
\circ \exp (i\epsilon H_{G_T})=p+i\epsilon \langle q\rangle_T+{\cal
O}_T(\epsilon ^2).}

The left hand side of (\ref{2dpr.1}) is the principal symbol of the
isospectral operator $e^{-\frac{\epsilon }{h}G_T(x,hD_x)}\circ
P_\epsilon \circ e^{\frac{\epsilon }{h}G_T(x,hD_x)}$ and under the
assumptions of Theorem \ref{bs3} resp. \ref{bs4} its imaginary part
will not take the value $i\epsilon F_0$ on $p^{-1}(0)$ away from
$\Lambda _d$ resp.~$\Lambda _d\cup \Lambda _r$. This means that we
have localized the spectral problem to a neighborhood of $\Lambda _d$
resp.~$\Lambda _d\cup\Lambda _r$.

Near $\Lambda _d$ we choose action-angle coordinates so that $\Lambda
_d$ becomes the zero section in the cotangent space of the 2-torus,
and \ekv{2dpr.2} {p_\epsilon (x,\xi )=p(\xi )+i\epsilon q(x,\xi
)+{\cal O}(\epsilon ^2).}  We follow the quantized Birkhoff normal
form procedure in the spirit of V.F.~Lazutkin and Y.~Colin de
Verdi\`ere \cite{La, Co}: solve first \ekv{2dpr.3}{H_pG=q(x,\xi
)-\langle q(\cdot ,\xi )\rangle ,} where the bracket indicates that we
take the average over the torus with respect to $x$. Composing with
the corresponding complex canonical transformation, we get the new
conjugated symbol
$$p(\xi )+i\epsilon \langle q(\cdot ,\xi )\rangle+{\cal O}(\epsilon
^2+\xi ^\infty ).$$ Here the Diophantine condition is of course
important.

Iterating the procedure we get for every $N$,
$$
p_\epsilon \circ \exp (H_{G^{(N)}})=\underbrace{p(\xi )+i\epsilon
(\langle q\rangle (\xi )+{\cal O}(\epsilon ,\xi
))}_{\mathrm{independent\ of\ }x}+{\cal O}((\xi ,\epsilon )^{N+1})
$$
This procedure can be continued on the operator level, and up to a
small error we see that $P_\epsilon $ is microlocally equivalent to an
operator $P_\epsilon (hD_\xi ,\epsilon ;h)$. At least formally,
Theorem \ref{bs3} then follows by considering Fourier series
expansions, but in order to get a full proof we also have take into
account that we have constructed complex canonical transformations
that are quantized by Fourier integral operators with complex phase
and study the action of these operators on suitable exponentially
weighted spaces.

Near $\Lambda _r$ we can still use action-angle coordinates as in
(\ref{2dpr.2}) but the homological equation (\ref{2dpr.3}) is no
longer solvable. Instead, we use {secular perturbation theory} (cf the
book \cite{LiLi}), which amounts to making a {partial Birkhoff
reduction.}

\par After a linear change of $x$-variables, we may assume that $p(\xi
)=\xi _2+{\cal O}(\xi ^2)$ and in order to fix the ideas $=\xi _2+\xi
_1^2$. Then we can make the averaging procedure only in the
$x_2$-direction and reduce $p_\epsilon $ in (\ref{2dpr.2}) to
$$
\widetilde{p}_\epsilon (x,\xi )=\underbrace{\xi _2+\xi _1^2+{\cal
O}(\epsilon )}_{\mathrm{independent\ of\ }x_2,\atop \approx \xi _2+\xi
_1^2+i\epsilon \langle q\rangle_2(x_1,\xi )}+{\cal O} ((\epsilon ,\xi
)^\infty ),
$$
where $\langle q\rangle_2(x_1,\xi )$ denotes the average with respect
to $x_2$.

\par
Carrying out the reduction on the operator level, we obtain up to
small errors an operator $\widetilde{P}_\epsilon
(x_1,hD_{x_1},hD_{x_2};h)$ and after passing to Fourier series in
$x_2$, a family of non-self-adjoint operators on $S_{x_1}^1$:
$\widetilde{P}_\epsilon (x_1,hD_{x_1},hk;h)$, $k\in {\bf Z}$.

The non-self-adjointness and the corresponding possible wild growth of
the resolvent makes it hard to go all the way to study individual
eigenvalues. However, it can be shown that in the region $|\xi _1|\gg
\epsilon ^{1/2}$ (inside the energy surface $p=0$) we can go further
and (as near $\Lambda _d$) get a sufficiently good elimination of the
$x$-dependence. This leads to the conclusion that the contributions
from a vicinity of $\Lambda _r$ to the spectrum of $P_\epsilon $ in
the rectangle
$$
|\Re z|\le \frac{\epsilon }{C},\ |\Im z-\epsilon F_0|\le
\frac{\epsilon ^{1+\delta }}{C},
$$ 
come from a neighborhood of $\Lambda _r$ of phase space volume ${\cal
O}(\epsilon ^{3/2})$.

This explains heuristically why the rational torus will contribute
with ${\cal O} (\epsilon ^{3/2}/h^2)$ eigenvalues in the rectangle.

The actual proof is more complicated. We use a Grushin problem
reduction in order to reduce the study near $\Lambda _r$ to that of a
square matrix of size ${\cal O}(\epsilon ^{3/2}/h^2)$. However, even
if we avoid the eigenvalues of such a matrix, the inverse can only be
bounded by \ekv{2dpr.4}{ \exp {\cal O}(\epsilon ^{3/2}/h^2).} What
saves us is that away from $\Lambda _r\cup \Lambda _d$, we can
conjugate the operator with exponential weights and show that the
resolvent has an ``off-diagonal decay'' like $\exp (-1/(Ch))$. This
implies that we can confine the growth in (\ref{2dpr.4}) to a small
neighborhood of $\Lambda _r$, if
$$
\frac{1}{Ch}\gg \frac{\epsilon ^{\frac{3}{2}}}{h^2},
$$ 
leading to the assumption $\epsilon \ll h^{2/3}$ in Theorem \ref{bs4}.

\part{Non-self-adjoint operators with random perturbations.}

\section{Zeros of holomorphic functions of exponential growth}\label{ze}
\setcounter{equation}{0}

We will need a result on the number of zeros in a domain of holomorphic functions $u(z)=u_h(z)$ that satisfy an exponential upper bound near the boundary of the domain as well as corresponding lower bounds at finitely many points, distributed along the boundary. Such a result (related to classical results for the zeros of entire functions, cf \cite{Lev80}, Chapter III, Section 3, Theorem 3) was obtained by Hager \cite{Ha06a, Ha06b} under a rather strong regularity assumption on the exponent. With Hager \cite{HaSj08} we obtained a more general result with a logarithmic loss  however.  Recently I revisited the proofs and was able to get a result that includes the earlier ones and allows to eliminate such losses, see \cite{Sj09b}.

Let $\Gamma \Subset {\bf C}$ be an open set and let $\gamma =\partial
 \Gamma $ be the boundary of $\Gamma $. Let $r:\gamma \to ]0,\infty [$
 be a Lipschitz function of Lipschitz modulus $\le 1/2$:
\ekv{ze.1}
{
\vert r(x)-r(y)\vert \le \frac{1}{2}\vert x-y\vert,\ x,y\in \gamma .
}
We further assume that $\gamma $ is Lipschitz in the following precise
sense, where $r$ enters:

For every $x\in \gamma $ there exist new affine coordinates $\widetilde{y}=(\widetilde{y}_1,\widetilde{y}_2)$ of the form $\widetilde{y}=U(y-x)$, $y\in {\bf C}\simeq {\bf R}^2$ being the old coordinates, where $U=U_x$ is orthogonal, such that the intersection of $\Gamma $ and the rectangle $R_x:= \{ y\in {\bf C} ;\, |\widetilde{y}_1|<r(x),\, |\widetilde{y}_2|<C_0r(x)\} $ takes the form \ekv{ze.2} {\{ y\in R_x;\, \widetilde{y}_2>f_x(\widetilde{y}_1),\ |\widetilde{y}_1|<r(x),\} } where $f_x(\widetilde{y}_1)$ is uniformly Lipschitz on $[-r(x),r(x)]$, and $C_0$ is a fixed constant, which is larger than the Lipschitz moduli of the functions $f_x$.

\par Notice that our assumption (\ref{ze.2}) remains valid if we
decrease $r$. It will be convenient to extend the function to all of
${\bf C}$, by putting 
\ekv{ze.3}
{
r(x)=\inf_{y\in \gamma }(r(y)+\frac{1}{2}|x-y|).
}
The extended function is also Lipschitz with modulus $\le
\frac{1}{2}$: 
$$
|r(x)-r(y)|\le \frac{1}{2}|x-y|,\ x,y\in {\bf C}.
$$
Notice that 
\ekv{ze.4}
{
r(x)\ge \frac{1}{2}\mathrm{dist\,}(x,\gamma ),
}
and that 
\ekv{ze.5}
{
|y-x|\le r(x)\Rightarrow \frac{r(x)}{2}\le r(y)\le \frac{3r(x)}{2}.
}

\begin{theo}\label{ze1}
Let $\Gamma \Subset {\bf C}$ be simply connnected, and have Lipschitz boundary $\gamma $ with an
associated Lipschitz weight $r$ as in (\ref{ze.1}), (\ref{ze.2}), 
(\ref{ze.3}).
 Put $\widetilde{\gamma }_{\alpha r}=\cup_{x\in \gamma
}D(x,\alpha r(x))$ for any constant $\alpha >0$. Let $z_j^0\in \gamma $, $j\in {\bf Z}/N{\bf Z} $ be distributed along 
the  boundary in the positively oriented sense such that 
$$r(z_j^0)/4\le |z_{j+1}^0-z_j^0|\le r(z_j^0)/2 .$$ (Here ``4'' can be
replaced by any fixed constant $>2$.) 
Then if
$C_1>0$ is large enough, depending only on the constant $C_0$ in (\ref{ze.2}) and if
$C_1\ge C_0$, $z_j\in D(z_j^0,r(z^0_j)/(2C_1))$, we have the following:

\par Let $\phi $ be a continuous subharmonic function on
$\widetilde{\gamma }_{r/C_1}$ with a distribution extension to
$\Gamma \cup \widetilde{\gamma }_{r/C_1}$ that will be denoted by the
same symbol. Then there exists a constant $C_2>0$ such that
if $u=u_h(z)$, $0<h\le 1$, is a holomorphic function on $\Gamma \cup \widetilde{\gamma
}_{r/{C_1}}$ satisfying
\ekv{ze.6}
{
h\ln |u|\le \phi (z)\hbox{ on }\widetilde{\gamma }_{r/{C_1}},
}
\ekv{ze.7}
{
h\ln |u(z_j)|\ge \phi (z_j)-\epsilon _j,\hbox{ for }j=1,2,...,N,
}
where $\epsilon _j\ge 0$, then the number of zeros of $u$ in $\Gamma $
satisfies
\eekv{ze.8}
{
&&|\# (u^{-1}(0)\cap \Gamma )-\frac{1}{2\pi h}\mu (\Gamma )|\le
}
{&&
\frac{C_2}{h}\left( 
\mu (\widetilde{\gamma }_{r/{C_1}})+\sum_1^N \left(\epsilon _j+\int
_{D(z_j,\frac{r(z_j)}{4C_1})}|\ln \frac{|w-z_j|}{r(z_j)}|\mu (dw)\right)
\right).
}
Here $\mu :=\Delta \phi \in {\cal D}'(\Gamma \cup \widetilde{\gamma }_{r/C_1})$ is a positive measure on $\widetilde{\gamma }_{r/C_1}$ so that $\mu (\Gamma )$ and $\mu (\widetilde{\gamma }_{r/C_1})$ are well-defined. Moreover,
the constant $C_2$ only depends on $C_0$ in (\ref{ze.2}) and on $C_1$.
\end{theo}

We next discuss the elimination of the logarithmic integrals in (\ref{ze.8}).

\smallskip\par
Using (\ref{ze.5}), we get 
\begin{eqnarray*}
&&\int_{D(z_j^0,\frac{r(z_j^0)}{2C_1})}\int
_{D(z,\frac{r(z)}{4C_1})}|\ln \frac{|w-z|}{r(z)}|\mu (dw)\frac{L(dz)}{L(D(z_j^0,\frac{r(z_j^0)}{2C_1}))}\le\\
&&\int_{D(z_j^0,\frac{r(z_j^0)}{2C_1})}\int
_{D(z^0_j,\frac{r(z^0_j)}{C_1})}|\ln \frac{|w-z|}{r(z)}|\mu
(dw)\frac{L(dz)}{L(D(z_j^0,\frac{r(z_j^0)}{2C_1}))},
\end{eqnarray*}
where $L$ denotes the Lebesgue measure.
Here we use Fubini's theorem and the fact that 
$$
\int_{D(z_j^0,\frac{r(z_j^0)}{2C_1})} |\ln \frac{|z-w|}{r(z)}|L(dz)\le {\cal O}(1)L(D(z_j^0,\frac{r(z_j^0)}{2C_1}))
$$
to conclude that the mean-value of
$$
D(z_j^0,\frac{r(z_j^0)}{2C_1})\ni z\mapsto \int_{D(z,\frac{r(z)}{4C_1})}
|\ln \frac{|w-z|}{r(z)}|\mu (dw)
$$
is ${\cal O}(1)\mu (D(z_j^0,\frac{r(z_j^0)}{C_1}))$. Thus we can find
$\widetilde{z}_j\in D(z_j^0,\frac{r(z_j^0)}{2C_1})$ such that 
$$
\sum_{j=1}^N \int_{D(\widetilde{z}_j,\frac{r(\widetilde{z}_j)}{4C})}
|\ln \frac{|w-\widetilde{z}_j|}{r(\widetilde{z}_j)}|\mu (dw)={\cal
  O}(1)\mu (\widetilde{\gamma }_{r/C}).
$$
This leads to the following variant of Theorem \ref{ze1}, where
(\ref{ze.8}) is simplified but where the choice of $z_j$ is no more arbitrary.
\begin{theo}\label{ze2}
Let $\Gamma $, $\gamma =\partial \Gamma $, $r$, $z_j^0$, $C_0$, $C_1$ be as in
Theorem \ref{ze1}. Then $\exists\,\widetilde{z}_j\in
D(z_j^0,\frac{r(z_j^0)}{2C_1})$ such that if $\phi $, $u$ are as in Theorem
\ref{ze1}, satisfying (\ref{ze.6}), and
\ekv{ze.9}
{
h\ln |u(\widetilde{z}_j)|\ge \phi (\widetilde{z}_j)-\epsilon _j,\ j=1,2,...,N,
} 
instead of (\ref{ze.7}), then
\ekv{ze.10}
{
|\# (u^{-1}(0)\cap \Gamma )-\frac{1}{2\pi h}\mu (\Gamma )|
\le \frac{C_2}{h}(\mu (\widetilde{\gamma }_{r/C_1})+\sum
\epsilon _j ).
}
\end{theo}

\smallskip
\par Of course, if we already know that 
\ekv{ze.11}
{
\int_{D(z_j,\frac{r(z_j)}{4C})}|\ln \frac{|w-z_j|}{r(z_j)}| \mu (dw)=
{\cal O}(1)\mu (D(z_j,\frac{r(z_j)}{4C}),
}
then we can keep $\widetilde{z}_j=z_j$ in (\ref{ze.8}) and get
(\ref{ze.10}). This is the case, if we assume that $\mu $ is
equivalent to the Lebesgue measure in the following sense:
\ekv{ze.12}
{
\frac{\mu (dw)}{\mu (D(z_j,\frac{r(z_j)}{2C_1}))} \asymp 
\frac{L (dw)}{L (D(z_j,\frac{r(z_j)}{2C_1}))}
\hbox{ on }D(z_j,\frac{r(z_j)}{2C_1}),\hbox{ uniformly for }j=1,2,...,N.
}
Then we get,
\begin{theo}\label{d3}
Make the assumptions of Theorem \ref{ze1} as well as (\ref{ze.11}) or the
stronger assumption (\ref{ze.12}). Then from (\ref{ze.6}), (\ref{ze.7}),
we conclude (\ref{ze.10}).
\end{theo}
In particular, we recover the counting proposition of M.~Hager
\cite{Ha06a, Ha06b}, where $\phi $ was independent of $h$ and of class $C^2$ 
in a fixed neighborhood of $\gamma $. Then $\mu
\asymp L$ and if we choose $r\ll 1$ constant and assume (\ref{ze.6}),
(\ref{ze.7}), we get from (\ref{ze.9}):
\ekv{ze.13}
{
|\# (u ^{-1}(0)\cap \Gamma )-\frac{1}{2\pi h}\mu (\Gamma )|
\le
\frac{\widetilde{C}}{h}(r+\sum_1^N \epsilon _j ).
}
Hager had $\epsilon _j=\epsilon $ independent of $j$,
$r=\sqrt{\epsilon }$, $N\asymp \epsilon ^{-1/2}$, so the remainder in
(\ref{ze.13}) is ${\cal O}(\frac{\sqrt{\epsilon }}{h})$.

We next outline the proof of Theorem \ref{ze1}.
Using a locally finite covering with discs $D(x,r(x))$ and a
subordinated partition of unity, it is standard to find a smooth function
$\widetilde{r}(x)$ satisfying 
\ekv{ze.14}
{
\frac{1}{C}r(x)\le \widetilde{r}(x)\le r(x),\ |\nabla
\widetilde{r}(x)|\le \frac{1}{2},\ \partial ^\alpha
\widetilde{r}(x)={\cal O}(\widetilde{r}^{1-|\alpha |}).
}

\par From now on, we replace $r(x)$ by $\widetilde{r}(x)$ 
and the drop the
tilde. The general estimates on $r$ remain valid and we have
$$
r(x)\ge \frac{1}{C}\mathrm{dist\,}(x,\gamma ).
$$.

\par Consider the signed distance to $\gamma $:
\ekv{ze.15}
{
g(x)=\cases{\mathrm{dist\,}(x,\gamma ),\ x\in \Gamma \cr
  -\mathrm{dist\,}(x,\gamma ),\ x\in {\bf C}\setminus \Gamma }
}

\par In the set $\cup_{x\in \gamma }R_x$, we consider the
regularized function 
\ekv{ze.16}
{
g_{\epsilon }(x)=\int \frac{1}{(\epsilon r(x))^2}\chi
(\frac{x-y}{\epsilon r(x)})g(y) L(dy),
}
where $0\le \chi \in C_0^\infty (D(0,1))$, $\int \chi
(x)L(dx)=1$. Here $\epsilon >0$ is small and we notice that
$r(x)\asymp r(y)$, $g(y)={\cal O}(r(y))$, when $\chi ((x-y)/\epsilon
r(x))\ne 0$. It follows that $g_\epsilon (x)={\cal O}(r(x))$ and more
precisely, since $g$ is Lipschitz, that 
\ekv{ze.17}
{
g_\epsilon (x)-g(x)={\cal O}(\epsilon r(x)).
}
Moreover one can show that if $(\nabla g)_\epsilon $ denotes the regularization of $\nabla g$, obtained as in (\ref{ze.16}), then 
\ekv{ze.18}{\nabla _xg_\epsilon (x)-(\nabla g)_\epsilon (x)={\cal O}(1)
\sup_{y\in D(x,\epsilon r(x))}\frac{|g(y)|}{r(x)},}
\ekv{ze.19}
{
\partial ^\alpha g_\epsilon (x)={\cal O}_{\alpha }((\epsilon
r(x))^{1-|\alpha |}),\ |\alpha |\ge 1.
}

\par Let $C>0$ be large enough but independent of $\epsilon $. Put 
\ekv{ze.20}
{
\widetilde{\gamma }_{C\epsilon r}=\{ x\in \cup_{y\in \gamma
}R_y;\, |g_\epsilon (x)|<C\epsilon r(x)\}.
}

\par If $C>0$ is sufficiently large, then in the coordinates
associated to (\ref{ze.2}), $\widetilde{\gamma }_{C\epsilon r}$ takes
the form
\ekv{ze.21}
{
f_x^-(\widetilde{y}_1)<\widetilde{y}_2<f_x^+(\widetilde{y}_1),\ |\widetilde{y}_1|<r(x),
}
where $f^\pm _x$ are smooth on $[-r(x),r(x)]$ and satisfy 
\ekv{ze.22}
{
\partial _{\widetilde{y}_1}^kf_x^\pm ={\cal O}_k((\epsilon
r(x))^{1-k}),\ k\ge 1,
}
\ekv{ze.23}
{
0< f_x^+-f_x,\, f_x-f_x^-\asymp \epsilon r(x).
}
Later, we will fix $\epsilon >0$ small enough and write $\gamma
_r=\widetilde{\gamma }_{C\epsilon r}$ and more generally, $\gamma
_{\alpha r}=\widetilde{\gamma }_{C\alpha \epsilon r}$.

\par We shall next establish an exponentially weighted estimate for
the Dirichlet Laplacian in $\gamma _r$ by adating the general approach of Agmon estimates to thin tubes (cf \cite{HeSj84}, \cite{Je00}):
\begin{prop}\label{ze4}
Let $C>0$ be sufficiently large and $\epsilon >0$ sufficiently small. Then if $\phi \in
C^2(\overline{\gamma }_r)$ and 
\ekv{ze.24}
{
|\phi '_x|\le \frac{1}{Cr},
}
we have 
\ekv{ze.25}
{
\Vert e^{\phi }Du\Vert+\frac{1}{C}\Vert \frac{1}{r}e^{\phi }u\Vert
\le C\Vert re^\phi \Delta u\Vert,\ u\in (H_0^1\cap H^2)(\gamma _r),
}
where we use the natural $L^2$ norms.
\end{prop}

\begin{outline} Let $\phi \in C^2(\overline{\gamma
  }_r;{\bf R})$ and put 
$$
-\Delta _\phi = e^{\phi }\circ (-\Delta )\circ e^{-\phi }=D_x^2-(\phi '_x)^2+i(\phi '_x\circ D_x+D_x\circ \phi '_x),
$$
where we make the usual observation that the last term is formally
anti-self-adjoint. Then for every $u\in (H^2\cap H_0^1)(\gamma _r)$:
\ekv{ze.26}
{
(-\Delta _\phi u|u)=\Vert D_xu\Vert^2-((\phi '_x)^2u|u).
}

\par We need an apriori estimate for $D_x$. Let $v:\overline{\gamma
}_r\to {\bf R}^n$ be sufficiently smooth. We sometimes consider $v$ as
a vector field. Then for $u\in (H^2\cap H_0^1)(\gamma _r)$:
$$
(Du|vu)-(vu|Du)=i(\mathrm{div\,}(v)u|u).
$$

Assume $\mathrm{div\,}(v)>0$. Recall that if $v=\nabla w$, then
$\mathrm{div\,}(v)=\Delta w$, so it suffices to take $w$ strictly
subharmonic. Then 
$$
\int \mathrm{div\,}(v)|u|^2dx\le 2\Vert vu\Vert\Vert Du\Vert\le \Vert
Du\Vert^2+\Vert vu\Vert^2,
$$
which we write
$$
\int (\mathrm{div\,}(v)-|v|^2)|u|^2 dx\le \Vert Du\Vert^2.
$$
Using this in (\ref{ze.26}), we get
\begin{eqnarray*}
&&\frac{1}{2}\Vert Du\Vert^2+\int
(\frac{1}{2}(\mathrm{div\,}(v)-|v|^2)-(\phi '_x)^2)|u|^2 dx\le\\
&& \Vert \frac{1}{k}(-\Delta _\phi )u\Vert\Vert ku\Vert\le 
\frac{1}{2}\Vert \frac{1}{k}(-\Delta _\phi )u\Vert^2+\frac{1}{2}
\Vert ku\Vert^2,
\end{eqnarray*}
where $k$ is any positive continuous function on $\overline{\gamma
}_r$. We write this as
\ekv{ze.26.5}
{
\frac{1}{2}\Vert Du\Vert^2+\int
(\frac{1}{2}(\mathrm{div\,}(v)-|v|^2-k^2)-(\phi _x')^2)dx \le
\frac{1}{2}\Vert \frac{1}{k}(-\Delta _\phi )u\Vert^2.
}

\par The remaining work is then to see that we can choose $v$ so that 
\ekv{ze.27}{
\mathrm{div\,}(v)\ge r^{-2},\ |v|\le {\cal O}(r^{-1}).
}
and it turns out that this is possible with
\ekv{ze.28}
{
v=\nabla (e^{\lambda g/r}),
}
where $\lambda >0$ is sufficiently large and $g=g_\epsilon $. Then replace $v$ by a small multiple and finally choose $k$ to be a small multiple of $1/r$.\end{outline}
 
\par If $\Omega \Subset {\bf C}$ has smooth boundary, let $G_\Omega $,
$P_\Omega $ denote the Green and the Poisson kernels of $\Omega $, so
that the Dirichlet problem,
$$
\Delta u=v,\ {{u}_\vert}_{\partial \Omega }=f,\quad u,v\in C^\infty
(\overline{\Omega }),\ f\in C^{\infty }(\partial \Omega ),
$$
has the unique solution
$$
u(x)=\int_\Omega G_\Omega (x,y)v(y)L(dy)+\int_{\partial \Omega
}P_\Omega (x,y)f(y)|dy|.
$$
Recall that $-G_\Omega \ge 0$, $P_\Omega \ge 0$. We have 
\ekv{ze.29}
{
-G_{\Omega }(x,y)\le C-\frac{1}{2\pi }\ln |x-y|,
}
where $C>0$ only depends on the diameter of $\Omega $. 

\par We also have the scaling property:
\ekv{ze.30}
{
G_{\Omega }(\frac{x}{t},\frac{y}{t})=G_{t\Omega }(x,y),\ x,y\in
t\Omega , t>0.
}
Moreover, $-G_{\Omega }$ is an increasing function of $\Omega $
in the natural sense. Using these facts with Proposition \ref{ze4} one can show the following result:
\begin{prop}\label{ze5}
For all $x,y\in \gamma _r$ (and $\epsilon >0$ small enough), we have 
\ekv{ze.31}
{
-G_{\gamma _r}(x,y)\le C-\frac{1}{2\pi }\ln \frac{|x-y|}{r(y)},\hbox{
  when }|x-y|\le \frac{r(y)}{C},
}
\ekv{ze.32}
{
-G_{\gamma _r}(x,y)\le C \exp (-\frac{1}{C}\int_{\pi _\gamma (y)}^{\pi _\gamma (x)}\frac{1}{r(t)}|dt|),\hbox{
  when }|x-y|\ge \frac{r(y)}{C},
}
where it is understood that the integral is evaluated along $\gamma $
from $\pi _\gamma (y)\in \gamma $ to $\pi _\gamma (x)\in \gamma $,
where $\pi _\gamma (y)$, $\pi _\gamma (x)$ denote points in $\gamma $
with $|x-\pi _\gamma (x)|=\mathrm{dist\,}(x,\gamma )$, 
$|y-\pi _\gamma (y)|=\mathrm{dist\,}(y,\gamma )$, and we choose these
two points (when they are not uniquely defined) and the intermediate
segment in such a way that the integral is as small as possible.
\end{prop}

\par We will also need a lower bound on $G_{\gamma _r}$ on suitable
subsets of $\gamma _r$. For $\epsilon >0$ fixed and sufficiently small,
we say that $M\Subset \gamma _r$ is an
elementary piece of $\gamma _r$ if 
\begin{itemize}
\item $M\subset \gamma _{(1-\frac{1}{Cr})}$,
\item $\frac{1}{C}\le \frac{r(x)}{r(y)}\le C$, $x,y\in M$,
\item $\exists y\in M$ such that $M=y+r(y)\widetilde{M}$, where
  $\widetilde{M}$ belongs to a bounded set of relatively compact
  subsets of ${\bf C}$ with smooth boundary.
\end{itemize}
In the following, it will be tacitly understood that we choose our
elementary pieces with some uniform control ($C$ fixed and uniform
control on the $\widetilde{M}$). Using Harnack's inequality one can show:
\begin{prop}\label{a3}
If $M$ is an elementary piece in $\gamma _r$, then 
\ekv{ze.33}
{
-G_{\gamma _r}(x,y)\asymp 1+|\ln \frac{|x-y|}{r(y)}|,\ x,y\in M.
}
\end{prop}

\par Let $\phi $ be a continuous subharmonic function defined in some neighborhood of 
$\overline{\gamma _r}$. Let 
\ekv{cz.9}
{
\mu =\mu _\phi =\Delta \phi 
}
be the corresponding locally finite  positive measure.

Let $u$ be a holomorphic function defined in a neighborhood of 
$\Gamma \cup\overline{\gamma _r}$. 
We assume that 
\ekv{cz.10}
{
h\ln \vert u(z)\vert \le \phi (z),\ z\in\overline{\gamma _r}. 
}

\begin{lemma}\label{cz1}
Let $z_0\in M$, where $M$ is an elementary piece, such that
\ekv{cz.11}
{h\ln \vert u(z_0)\vert \ge \phi (z_0)-\epsilon ,\ 0<\epsilon \ll 1.}
Then the number of zeros of $u$ in $M$ is 
\ekv{cz.12}{\le {C\over h}(\epsilon +\int_{\gamma _r}-G_{\gamma_r}(z_0,w)\mu (dw)).}
\end{lemma}
\begin{proof}
Writing $\phi $ as a uniform limit of an increasing sequence of smooth 
functions, we may assume that $\phi \in C^\infty $.
Let 
$$
n_u(dz)=\sum 2\pi \delta (z-z_j),
$$
where $z_j$ are the zeros of $u$ counted with their multiplicity. We may 
assume that no $z_j$ are situated on $\partial \gamma _r$. Then, since 
$\Delta \ln \vert u\vert =n_u$,
\eeekv{cz.13}
{
&&\hskip -2cm h\ln \vert u(z)\vert = \int_{\gamma _r} G_{\gamma_r}(z,w)h n_u (dw)+\int_{\partial 
\gamma _r}P_{\gamma_r}(z,w)h\ln \vert u(w)\vert \vert dw\vert}   
{&&\le \int_{\gamma _r}G_{\gamma_r}(z,w)hn_u(dw)+\int_{\partial \gamma 
_r}P_{\gamma_r}(z,w)\phi (w)\vert dw\vert} 
{&&= \int_{\gamma _r}G_{\gamma_r}(z,w)hn_u(dw)+\phi (z)-\int_{\gamma 
_r}G_{\gamma_r}(z,w)\mu (dw).}
Putting $z=z_0$ in \no{cz.13} and using \no{cz.11}, we get
$$
\int_{\gamma _r}-G_{\gamma_r}(z_0,w)hn_u(dw)\le \epsilon +\int_{\gamma 
_r}-G_{\gamma_r}(z_0,w)\mu (dw).
$$
Now $$
-G_{\gamma_r}(z_0,w)\ge {1 \over C},\ w\in M,
$$
and we get \no{cz.12}.
\end{proof}

\par Notice that this argument is basically the same as when using 
Jensen's 
formula to estimate the number of zeros of a holomorphic function in a disc. 

\par Now we sharpen the assumption \no{cz.11} and assume 
\ekv{cz.14}
{
h\ln \vert u(z_j)\vert \ge \phi (z_j)-\epsilon_j ,
}
where $z_1,...,z_N\in \gamma _{(1-{1\over C_1})r}$
are points such that with the cyclic convention $N+1=1$:
 
\ekv{cz.15}
{
|z_{j+1}-z_j|\asymp r(z_j),\ \frac{r(z_{j+1})}{r(z_j)}\asymp 1.
}
We also assume that $z_1,z_2,...,z_N$ are arranged in such a way that 
\ekv{cz.16}
{
{\bf Z}/N{\bf Z}\mapsto \pi_\gamma  (z_j) \hbox{ runs through the oriented boundary in the positive sense.}
}
Let $M_j\subset \gamma _r$ be elementary pieces such that 
\ekv{cz.16.5}
{z_j\in M_j,\ \mathrm{dist\,}(z_j,M_k)\ge \frac{r(z_j)}{C}\hbox{ when
  }
k\ne j,\
  \gamma _{\widetilde{r}}\subset \cup_j M_j,\ \widetilde{r}=(1-\frac{1}{C_1})r.}
We will also assume for a while that $\phi $ 
is smooth. 

\par According to Lemma \ref{cz1}, we have 
\ekv{cz.17}
{
\# (u^{-1} (0)\cap M_j)\le 
{C_3\over h}(\epsilon_j +\int_{\gamma _r}-G_{\gamma_r}(z_j,w)\mu (dw)).
}

\par Consider the harmonic functions on $\gamma _{\widetilde{r}}$,
\ekv{cz.19}
{
\Psi (z)=h(\ln \vert u(z)\vert +\int_{\gamma 
_{\widetilde{r}}}-G_{\gamma _{\widetilde{r}}}(z,w)n_u(dw)),
}
\ekv{cz.20}
{
\Phi (z)=\phi (z)+\int_{\gamma _{\widetilde{r}}}-
G_{\gamma _{\widetilde{r}}}(z,w)\mu (dw).
}
Then $\Phi (z)\ge \phi (z)$ with equality on $\partial \gamma 
_{\widetilde{r}}$. Similarly, $\Psi (z)\ge h\ln \vert u(z)\vert $ with 
equality on $\partial \gamma _{\widetilde{r}}$.

\par Consider the harmonic function
\ekv{cz.21}
{
H(z)=\Phi (z)-\Psi (z),\ z\in \gamma _{\widetilde{r}}.
}
Then on $\partial \gamma _{\widetilde{r}}$, we have by \no{cz.10} that
$$
H(z)=\phi (z)-h\ln \vert u(z)\vert \ge 0,
$$
so by the maximum principle,
\ekv{cz.22}
{
H(z)\ge 0,\hbox{ on }\gamma _{\widetilde{r}}.
}
By \no{cz.14}, we have 
\eeeekv{cz.23}
{
H(z_j)&=& \Phi (z_j)-\Psi (z_j)
}
{
&=&\phi (z_j)-h\ln \vert u(z_j)\vert }{&&+
\int_{\gamma _{\widetilde{r}}}-G_{\gamma _{\widetilde{r}}}(z_j,w)\mu 
(dw)-\int_{\gamma _{\widetilde{r}}} -G_{\gamma _{\widetilde{r}}}(z_j,w)hn_u(dw)
}
{
&\le & \epsilon_j +\int_{\gamma _{\widetilde{r}}} -
G_{\gamma _{\widetilde{r}}}(z_j,w)\mu (dw).
}

\par Harnack's inequality implies that 
\ekv{cz.24}
{
H(z)\le {\cal O}(1)(\epsilon_j +\int -G_{\gamma _{\widetilde{r}}}(z_j,w)\mu (dw))\hbox{ 
on }M_j\cap \gamma _{\widehat{r}},\ \widehat{r}=(1-\frac{1}{C_1})\widetilde{r}.}

\par Now assume that $u$ extends to a holomorphic function in a neighborhood of 
$\Gamma \cup \overline{\gamma _r}$. We then would like to evaluate the 
number of zeros of $u$ in $\Gamma $. Using \no{cz.17}, we first have 
\ekv{cz.25}
{
\# (u^{-1}(0)\cap \gamma _{\widetilde{r}})\le {C\over h}\sum_{j=1}^N
\left(\epsilon _j
+ \int_{\gamma _r}-G_{\gamma_r}(z_j,w)\mu (dw)\right).
}
\par Let $\chi \in C_0^\infty (\Gamma \cup \gamma _{\widehat{r}};[0,1])$ be equal to 1 on $\Gamma $. Of course $\chi $ 
will have to depend on $r$ but we may assume that for all $k \in{\bf N}$,
\ekv{cz.26}
{
\nabla ^k\chi ={\cal O}(r^{-k}).
}
We are interested in 
\ekv{cz.27}
{
\int \chi (z)hn_u(dz)=\int_{\gamma _{\widehat{r}}}h\ln \vert u(z)\vert 
\Delta \chi (z)L(dz).
}
Here we have on $\gamma _{\widetilde{r}}$
\eeeekv{cz.28}
{h\ln \vert u(z)\vert &=&\Psi (z)-\int_{\gamma 
_{\widetilde{r}}}-G_{\gamma _{\widetilde{r}}}(z,w)hn_u(dw)}
{&=&\Phi (z)-H(z)-\int_{\gamma 
_{\widetilde{r}}}-G_{\gamma _{\widetilde{r}}}(z,w)hn_u(dw)}
{
&=&\phi (z)+\int_{\gamma _{\widetilde{r}}}-G_{\gamma _{\widetilde{r}}}
(z,w)\mu 
(dw)-H(z)-\int_{\gamma _{\widetilde{r}}}-G_{\gamma _{\widetilde{r}}}(z,w)hn_u(dw)
}
{&=& \phi (z)+R(z),}
where the last equality defines $R(z)$.

\par Inserting this in \no{cz.27}, we get 
\ekv{cz.29}
{
\int\chi (z)hn_u(dz)=\int \chi (z)\mu (dz)+\int R(z)\Delta \chi (z)L(dz).
}
(Here we also used some extension of $\phi $ to $\Gamma $ 
with $\mu =\Delta \phi $.) The task is now to estimate $R(z)$ and the 
corresponding integral in \no{cz.29}. Put 
\ekv{cz.30}
{
\mu _j=\mu (M_j\cap \gamma _{\widetilde{r}}).
}
Using the exponential decay property (\ref{ze.32}) (equally valid for 
$G_{\gamma _{\widetilde{r}}}$) we get for $z\in M_j\cap \gamma 
_{\widetilde{r}}$, ${\rm dist\,}(z,\partial M_k)\ge r(z_j)/{\cal
  O}(1)$, $k\ne j$:
\ekv{cz.31}
{
\int_{\gamma _{\widetilde{r}}}-G_{\gamma _{\widetilde{r}}}
(z,w)\mu (dw)\le 
\int_{M_j\cap \gamma _{\widetilde{r}}}
-G_{\gamma _{\widetilde{r}}}(z,w)\mu 
(dw)+{\cal O}(1)\sum_{k\ne j}\mu _ke^{-{1\over C_0}\vert j-k\vert },
}
where $|j-k|$ denotes the natural distance from $j$ to $k$ in ${\bf
  Z}/N{\bf Z}$.
Similarly from \no{cz.24}, we get 
\ekv{cz.32}
{
H(z)\le {\cal O}(1)(\epsilon_j +
\int_{M_j\cap\gamma_{\widetilde{r}}}-G_{\gamma _{\widetilde{r}}}
(z_j,w)\mu (dw)+\sum_{k\ne j}e^{-{1\over 
C_0}\vert j-k\vert }\mu _k),
}
for $z\in M_j\cap \gamma _{\widetilde{r}}$. 

\par This gives the following estimate on the contribution from the first 
two terms in $R(z)$ to the last integral in \no{cz.29}:
\eeekv{cz.32.5}
{
&&\int_{\gamma _{\widetilde{r}}}\left(\int_{\gamma _{\widetilde{r}}}
-G_{\gamma _{\widetilde{r}}}(z,w)\mu (dw)-H(z)\right)\Delta \chi (z)L(dz)
}{&&={\cal O}(1)\sum_j(\epsilon _j+\int_{M_j\cap\gamma_{\widetilde{r}}}-G_{\gamma
    _{\widetilde{r}}}(z_j,w)\mu (dw))+\sum_{k\ne
    j}e^{-\frac{1}{C_0}|j-k|}\mu _k)
}{&&+{\cal O}(1)\sum_j\int_{M_j\cap\gamma_{\widetilde{r}}}\int_{M_j\cap\gamma_{\widetilde{r}}}-G_{\gamma _{\widetilde{r}}}
(z,w)\mu (dw)|\Delta \chi (z)|L(dz).
}
Here,
\ekv{cz.32.7}{
\int_{M_j\cap\gamma_{\widetilde{r}}}-G_{\gamma _{\widetilde{r}}} |\Delta \chi (z)|L(dz)={\cal O}(1), 
}
so (\ref{cz.32.5}) leads to 
\eekv{cz.33}
{
&&\int_{\gamma _{\widetilde{r}}}\left(\int_{\gamma _{\widetilde{r}}}
-G_{\gamma _{\widetilde{r}}}(z,w)\mu (dw)-H(z)\right)
\Delta \chi (z)L(dz)}
{&&={\cal O}(1)\left(\mu (\gamma _{\widetilde{r}})+\sum_j \epsilon _j
+\sum_j \int_{M_j\cap\gamma_{\widetilde{r}}} -G_{\gamma _{\widetilde{r}}}(z_j,w)\mu (dw)\right).
}

\par The contribution from the last term in $R(z)$ (in \no{cz.28}) to the 
last integral in \no{cz.29} is 
\ekv{cz.34}
{
\int_{z\in \gamma _{\widehat{r}}}\int_{w\in \gamma 
_{\widetilde{r}}}G_{\gamma _{\widetilde{r}}}(z,w)hn_u(dw)\Delta \chi (z)L(dz).
}
Here, by using an estimate similar to (\ref{cz.31}) with $\mu (dw)$
replaced by $L(dz)$ together with (\ref{cz.32.7}), we get 
$$\int_{z\in \gamma _{\widehat{r}}}G_{\gamma _{\widetilde{r}}}(z,w)(\Delta \chi )
(z)L(dz)={\cal O}(1),
$$
so the expression \no{cz.34} is by (\ref{cz.25})
\eeekv{cz.35}
{
&&{\cal O}(h)\# (u^{-1} (0)\cap \gamma _{\widetilde{r}})
}
{
&=&{\cal O}(1)\sum_{j=1}^N (\epsilon _j+ \int_{\gamma 
_r}(-G_{\gamma_r}(z_j,w))\mu (dw))
}
{&=&
{\cal O}(1)(\mu (\gamma _r)+\sum_{j=1}^N(\epsilon _j+ \int_{M 
_j}-G_{\gamma_r}(z_j,w)\mu (dw))).
}
This is quite similar to (\ref{cz.33}). Using Proposition \ref{ze5}, we have 
$$
\int_{M_j\cap\gamma_{\widetilde{r}}}-G_{\gamma _{\widetilde{r}}}(z_j,w)\mu (dw)\le
{\cal O}(1)(\int_{|w-z_j|\le \frac{r(z_j)}{C}}|\ln
\frac{|z_j-w|}{r(z_j)}|\mu (dw)+\mu (M_j\cap\gamma_{\widetilde{r}})
$$
and similarly for the last integral in (\ref{cz.35})
Using all this in \no{cz.29}, we get
\eekv{cz.36}
{&&\hskip -5truemm
\int\chi (z)hn_u(dz)=\int \chi (z)\mu (dz)
}
{
&&+{\cal O}(1)(\mu (\gamma _r)+\sum_j(\epsilon _j+\int _{|w-z_j|\le
  r(z_j)/C}|\ln (\frac{|z_j-w|}{r(z_j)})|\mu (dw)
).
}
We replace the smoothness assumption on $\phi $ by the assumption that 
$\phi $ is continuous near $\Gamma $ and keep \no{cz.14}. Then by 
regularization, we still get \no{cz.36}.

Here, we observe that 
$$
|\# (u^{-1}(0)\cap \Gamma ) -\frac{1}{2\pi h}\int \chi (z) hn_u (dz)|
\le
\# (u^{-1}(0)\cap \gamma _{\widetilde{r}}),
$$
which can be estimated by means of (\ref{cz.35}),
and combining this with (\ref{cz.36}), we get 
\eekv{d.2}
{
&&|\# (u^{-1}(0)\cap \Gamma )-\frac{1}{2\pi h}\mu (\Gamma )|\le
}
{&& \frac{{\cal O}(1)}{h}
\left( \mu (\gamma _r)+\sum_j (\epsilon _j +\int_{|w-z_j|\le
    \frac{r(z_j)}{C}} |\ln \frac{|z_j-w|}{r(z_j)}| \mu (dw))
\right) . }

\section{The one-dimensional semi-classical case}\label{one}
\setcounter{equation}{0}

In this section we consider a simple model operator in dimension 1 and show how random perturbations give rise to Weyl asymptotics in the interior of the range of $p$. We follow rather closely the work of Hager \cite{Ha06b} with some inputs also from Bordeaux Montrieux \cite{Bo} and Hager--Sj \cite{HaSj08}. Some of the general ideas appear perhaps more clearly in this special situation.

\par Let $P=hD_x+g(x)$, $g\in C^\infty (S^1)$ with symbol $p(x,\xi )=\xi +g(x)$, and assume that $\Im g$
has precisely two critical points; a unique maximum and a unique
minimum. 

\par Let $\Omega \Subset \{z\in {\bf C};\, \min \Im g<\Im z<\max \Im
g\}$
 be open. Put 
\eekv{3.1ny}{P_\delta =P_{\delta ,\omega }=hD_x+g(x)+\delta Q_\omega ,}
{Q_\omega u(x)=\sum_{|k|,|\ell|\le \frac{C_1}{h}}\alpha _{j,k}(\omega
  )(u|e^k)e^{\ell}(x),}
where $C_1>0$ is sufficiently large, $e^k(x)=(2\pi )^{-1/2}e^{ikx}$,
$k\in {\bf Z}$, and $\alpha _{j,k}\sim {\cal N}(0,1)$ are independent
complex Gaussian random variables, centered with variance 1. $Q_\omega $ is compact, so $P_\delta
$ has discrete spectrum.
Let $\Gamma \Subset \Omega $ have smooth boundary. 
\begin{theo}\label{31}
Let $\kappa >5/2$ and let $\epsilon _0>0$ be sufficiently small. Let
$\delta =\delta (h)$ satisfy $e^{-\epsilon _0/h}\ll \delta \ll
h^\kappa $ and put $\epsilon =\epsilon (h)=h\ln (1/\delta )$. Then for
$h>0$ small enough, we have with probability $\ge 1-{\cal
  O}(\frac{\delta ^2}{\sqrt{\epsilon }h^5})$ that the number of
eigenvalues of $P_\delta $ in $\Gamma $ satisfies
\ekv{3.2}
{
|\#(\sigma (P_\epsilon )\cap\Gamma )-\frac{1}{2\pi
  h}\mathrm{vol}(p^{-1}(\Gamma ))|\le
\mathrm{Const.\,}\frac{\sqrt{\epsilon }}{h}.
}
\end{theo}
If instead, we let $\Gamma $ vary in a set of subsets that satisfy the
assumptions uniformly, then with probability $\ge 1-{\cal
  O}(\frac{\delta ^2}{\epsilon h^5})$  we have (\ref{3.2}) uniformly
for all $\Gamma $ in that subset. The remainder of the section is devoted
to the (outline of) the proof of this result.
\subsection{Preparations for the unperturbed operator}
For $z\in \Omega $, let $x_+(z), x_-(z)\in S^1$ be the solutions of the 
equation $\Im g(x)=\Im z$, with $\pm \Im g'(x_\pm)<0$, define $\xi
_{\pm}(z)$ by $\xi _\pm +\Re g(x_{\pm})=\Re z$. Then, with $\rho
_{\pm}=(x_\pm, \xi _\pm )$, we have
$$
p(\rho _{\pm})=z,\quad \pm \frac{1}{i}\{ p,\overline{p}\} (\rho _\pm)>0.
$$
We introduce quasimodes of the form 
$$e_{\mathrm{wkb}}(x)=h^{-1/4}a(h)\chi
 (x-x_+(z))e^{\frac{i}{h}\phi _+(x)},$$ where $a(h)\sim a_0+ha_1+..,\ a_0\ne 0$, $\phi _+(x)=\int_{x_+(z)}^x (z-g(y))dy$, $\chi \in C_0^\infty
 (\mathrm{neigh\,}(0,{\bf R}))$ and $\chi (x)=1$ in a neighborhood of $0$. We can choose $a$ depending smoothly on $z$ such that all
 derivatives with respect to $z,\overline{z}$ are bounded when $h\to 0$ and $\Vert e_{\mathrm{wkb}}\Vert =1$ where we take the $L^2$ norm over $]x_-(z),x_+(z)+2\pi [$. We can
 assume that $e_{\mathrm{wkb}}$ is normalized in $L^2$ and 
$$
(P-z)e_{\mathrm{wkb}}={\cal O}(e^{-\frac{1}{Ch}}). 
$$

\par
Define $z$-dependent elliptic self-adjoint operators
$$
Q=(P-z)^*(P-z),\, \widetilde{Q}=(P-z)(P-z)^*:\, L^2(S^1)\to L^2(S^1), 
$$
with domain ${\cal D}(Q),\, {\cal D}(\widetilde{Q})=H^2(S^1)$.
They have discrete spectrum $\subset [0,+\infty [$. Using that $P-z$
is Fredholm of index zero, we see that $\mathrm{dim\,}{\cal N}(Q)=
\mathrm{dim\,}{\cal N}(\widetilde{Q})$. If $\mu \ne 0$ is an
eigenvalue of $Q$, with the corresponding eigenfunction $e\in C^\infty
$, then $f:=(P-z)e $ is an eigenfunction for $\widetilde{Q}$ with the
same eigenvalue $\mu $. Pursuing this observation, we see that 
$$
\sigma (Q)=\sigma(\widetilde{Q})=\{ t_0^2,t_1^2,...\},\ 0\le t_j\nearrow
+\infty . 
$$ 
\begin{prop}\label{41}
There exists a constant $C>0$ such that $t_0^2={\cal O}(e^{-1/(Ch)})$,
$t_1^2-t_0^2\ge h/C$ for $h>0$ small enough.
\end{prop}
\begin{proof} We have $Qe_{\mathrm{wkb}}=r,$ $\Vert r\Vert={\cal
  O}(e^{-1/Ch})$ and since $Q$ is self adjoint we deduce that
$t_0^2$ is exponentially small. If $e_0$ denotes the corresponding
normalized eigenfunction, we see that $(P-z)e_0=:v$ with $\Vert v\Vert$
exponentially small. Considering this ODE on $]x_-(z)-2\pi ,x_-(z)[$,
we get 
$$e_0(x)=Ch^{-\frac{1}{4}}a(h)e^{\frac{i}{h}\phi _+(x)}+Fv(x),$$
$$
Fv(x)=\frac{i}{h}\int_{x_+}^xe^{\frac{i}{h}(\phi _+(x)-\phi _+(y))}v(y)dy,
$$
where $\phi _+(x)=\int_{x_+}^x(z-g(y))dy$. We observe that $\Im (\phi
_+(x)-\phi _+(y))\ge 0$ on the domain of integration. With some
more work Hager showed that $\Vert F\Vert_{{\cal L}(L^2,L^2)} ={\cal O}(h^{-1/2})$. Hence
for our particular $v$, we see that $Fv$ is exponentially decaying in
$L^2$. Recalling the form of $e_{\mathrm{wkb}}(x)$ we conclude that
$\Vert e_0-e_{\mathrm{wkb}}\Vert$ is exponentially small.

To show that $t_1^2-t_0^2\ge h/C$, it suffices to show that $(Qu|u)\ge
\frac{h}{C}\Vert u\Vert ^2$ when $u\perp e_0$ or in other words, that 
\ekv{4.1ny}
{
\Vert u\Vert \le \sqrt{\frac{C}{h}}\Vert (P-z)u\Vert .
} 
If $v:= (P-z)u$, we again have 
$$u=Ch^{-\frac{1}{4}}a(h)e^{\frac{i}{h}\phi _+(x)}+Fv$$
for some constant $C$
and the orthogonality requirement on $u$ implies that 
$$
0=(1+{\cal O}(h^\infty ))C+(Fv|e_0),
$$
where $(Fv|e_0)={\cal O}(h^{-\frac{1}{2}})\Vert v\Vert$, so $C={\cal
  O}(h^{-1/2})\Vert v\Vert$ and we get the desired estimate on 
$\Vert u\Vert$.\end{proof}
\subsection{Grushin (Shur, Feschbach, bifurcation)
    approach}
Let $f_0$ be the normalized eigenfunction such that
$\widetilde{Q}f_0=t_0^2f_0$. As observed prior to Proposition \ref{41}, we get 
$$
(P-z)e_0=\alpha _0f_0,\ (P-z)^*f_0=\beta _0e_0,\quad \alpha _0\beta
_0=t_0^2,
$$
and combining this with $((P-z)e_0|f_0)=(e_0|(P-z)^*f_0)$, we see that $\alpha _0=\overline{\beta }_0$.
Define $R_+:L^2(S^1)\to {\bf C}$, $R_-={\bf C}\to L^2(S^1)$ by 
$$R_+u=(u|e_0),\quad R_-u_-=u_-f_0.$$
Then 
$$
{\cal P}(z):=\left(\begin{array}{ccc}P-z &R_-\\R_+
    &0 \end{array}\right):
H^1\times {\bf C}\to L^2\times {\bf C}
$$
is bijective with the bounded inverse
$${\cal E}(z)=\left(\begin{array}{ccc}E &E_+\\E_-
    &E_{-+} \end{array}\right).$$

Here $E={\cal O}(h^{-1/2})$ in $L^2\to L^2$ is basically the inverse of $P-z$ from $(f_0)^{\perp}$ to $(e_0)^\perp$,
$E_-v=(v|f_0)$, $E_+v_+=v_+e_0$,
$E_{-+}={\cal O}(e^{-1/(Ch)})$.
It is a general feature of such auxiliary (Grushin) operators that 
$$
z\in \sigma (P) \Leftrightarrow E_{-+}(z)=0.
$$ 
\subsection{d-bar equation for $E_{-+}$}
\begin{prop}\label{42}
We have 
\ekv{4.2ny}
{
\partial _{\overline{z}}E_{-+}(z)+f(z)E_{-+}(z)=0,
}
where 
\ekv{4.3ny}{f(z)=f_+(z)+f_-(z),\quad f_+(z)=(\partial
  _{\overline{z}}R_+)E_+,\ f_-(z)=E_-\partial _{\overline{z}}R_-.}
Thus,
\ekv{4.4ny}
{
\partial _{\overline{z}}(e^{F(z)}E_{-+}(z))=0\hbox{ if }\partial _{\overline{z}}F(z)=f(z).
}
Moreover,
\ekv{4.5ny}{
\Re \Delta F(z)=\Re 4\partial
_zf=\frac{2}{h}(\frac{1}{\frac{1}{i}\{p,\overline{p}\}(\rho _+)}-
\frac{1}{\frac{1}{i}\{p,\overline{p}\}(\rho _-)})+{\cal O}(1).
}
\end{prop}

\begin{proof}
(\ref{4.2ny}), (\ref{4.3ny}) follow from the general formula for the
differentiation of the inverse of an operator, here:
$$
\partial _{\overline{z}}{\cal E}+{\cal E}(\partial _{\overline{z}}{\cal
  P}){\cal E}=0.
$$
Let $\Pi (z)$ be the spectral projection of $Q$: $L^2\to {\bf
  C}e_0$. It is easy to see that the various $z$ and $\overline{z}$
derivatives of $e_{\mathrm{wkb}}$ and $\Pi (z)$ have at most temperate
growth in $1/h$ and since $e_0$ is the normalization of $\Pi
(z)e_{\mathrm{wkb}}$ we get the same fact for $e_0$ and hence for
$e_0-e_{\mathrm{wkb}}$. This quantity is also exponentially small in
$L^2$ and by elementary interpolation estimates for the successive
derivatives in $z,\overline{z}$ we get the same conclusion for the
higher $z$-gradients of $e_0-e_{\mathrm{wkb}}$.\\
It follows that 
$$
f_+(z)=(e_0(z)|\partial
_ze_0(z))=(e_{\mathrm{wkb}}(z)|\partial _ze_{\mathrm{wkb}}(z))+{\cal O}(e^{-\frac{1}{Ch}}),
$$
and the various $z,\overline{z}$-derivatives of the remainder are 
also exponentially decaying.

Using that $e_{\mathrm{wkb}}$ behaves like a Gaussian, peaked at the
point $x_+(\zeta )$, we can apply a variant of the method of
stationary phase to get 
\ekv{4.7}
{(e_{\mathrm{wkb}}| \partial
_ze_{\mathrm{wkb}})=-\frac{i}{h}\overline{(\partial _z\phi
  _+)(x_+(z),z)}+{\cal O}(1),}
where the remainder remains bounded after taking $z,\overline{z}$
derivatives.\\
Using that $\phi _+(x_+(z),z)=0$, $(\phi_+)'_x(x_+(z),z)=\xi _+(z)$, we
get after applying $\partial _z$ to the first of these relations, that 
$$
(\partial _z\phi _+)(x_+(z),z)=-\xi _+(z)\partial _zx_+(z).
$$
On the other hand, if we apply $\partial _z$ and $\partial
_{\overline{z}}$ to the equation, $p(x_+(z),\xi _+(z))=z$ and use that
$x_+$ and $\xi _+(z)$ are real valued we can show that 
$$
\partial _{\overline{z}}x_+=\frac{p_\xi '}{\{ p,\overline{p}\}}(\rho
_+), \quad \partial _{\overline{z}}\xi _+=\frac{-p_x '}{\{ p,\overline{p}\}}(\rho _+) .
$$
Plugging this into (\ref{4.7}), applying $\partial _z$ and taking real parts, we get the
second (non-trivial) identity in (\ref{4.5ny}) for the contribution from
$f_+$. The one from $f_-$ can be treated similarly.\end{proof}

Using the expressions for the $z$-derivatives of $x_+, \xi _+$ and the
analogous ones for $x_-,\xi _-$, we have the following easy result
relating (\ref{4.5ny}) to the symplectic volume:
\begin{prop}\label{43}
Writing $z=x+iy$, we have:
$$
d\xi _+(z)\wedge dx_+(z)=\frac{2}{\frac{1}{i}\{p,\overline{p}\}(\rho
  _+)}dy\wedge dx,
$$
$$
-d\xi _-(z)\wedge dx_-(z)=-\frac{2}{\frac{1}{i}\{p,\overline{p}\}(\rho
  _-)}dy\wedge dx,
$$
so by (\ref{4.5ny}),
\ekv{4.11}{
\Re \Delta F(z)dy\wedge dx=\frac{1}{h}(d\xi _+\wedge dx_+-
d\xi _-\wedge dx_-)+{\cal O}(1).
}
\end{prop}

\subsection{Adding the random perturbation}

Let $X\sim {\cal N}_{{\bf C}}(0,\sigma ^2)$ be a complex Gaussian random
variable, meaning that $X$ has the probability distribution
\ekv{5.1}
{
X_*(P(d\omega))=\frac{1}{\pi \sigma ^2}e^{-\frac{|X|^2}{\sigma
    ^2}}d(\Re X)d(\Im X). 
}
Here $\sigma >0$. For $t<1/\sigma ^2$, we have the expectation value
\ekv{5.2}
{
E(e^{t|X|^2})=\frac{1}{1-\sigma ^2t}.
}
Bordeaux Montrieux \cite{Bo} observed that we have the following possibly classical result (improving a similar statement in \cite{HaSj08}).
\begin{prop}\label{51}
There exists $C_0>0$ such that the following holds: Let $X_j\sim {\cal
  N}_{\bf C}(0,\sigma _j^2)$, $1\le j\le N<\infty $ be independent complex Gaussian
random variables. Put $s_1=\max \sigma _j^2$. Then for every $x>0$, we
have 
$$
P(\sum_1^N|X_j|^2\ge x)\le \exp(\frac{C_0}{2s_1}\sum_1^N \sigma _j^2-\frac{x}{2s_1}).
$$
\end{prop}

\begin{proof}
For $t\le 1/(2s_1)$, we have 
\begin{eqnarray*}
P(\sum |X_j|^2\ge x)\le E(e^{t(\sum |X_j|^2-x)})=e^{-tx}\prod_1^NE(e^{t
|X_j|^2})\\
=\exp (\sum_1^N \ln \frac{1}{1-\sigma _j^2t}-tx)\le \exp t(C_0\sum
\sigma _j^2t -x).
\end{eqnarray*}
It then suffices to take $t=(2s_1)^{-1}$.\end{proof}

\par Recall that 
\ekv{5.3}
{
Q_\omega u(x)=\sum_{|k|,|j|\le C_1/h}\alpha _{j,k}(\omega
)(u|e^k)e^j(x),\ e^k(x)=\frac{1}{\sqrt{2\pi }}e^{ikx}.
}
Since the Hilbert-Schmidt norm of $Q_\omega $ is given by 
$\Vert Q_\omega \Vert_{\mathrm{HS}}^2=\sum |\alpha _{j,k}(\omega )|^2$,
we get from the preceding proposition:

\begin{prop}\label{52}
If $C>0$ is large enough, then 
\ekv{5.4}
{
\Vert Q_\omega \Vert_{\mathrm{HS}}\le \frac{C}{h}\hbox{ with
  probability }\ge 1-e^{-\frac{1}{Ch^2}}.
}
\end{prop}
Now, we work under the assumption that $\Vert Q_\omega
\Vert_{\mathrm{HS}}\le C/h$ and recall that $\Vert Q_\omega \Vert \le
\Vert Q_\omega \Vert_{\mathrm{HS}}$. Assume that 
\ekv{5.5}
{
\delta \ll h^{3/2},
}
so that $\Vert \delta Q_\omega \Vert \ll h^{1/2}$. 
Then, by simple
perturbation theory we see that 
$$
{\cal P}_\delta (z)=\left(\begin{array}{ccc}P_\delta -z &R_-\\ R_+
    &0 \end{array}\right): H^1\times {\bf C}\to L^2\times {\bf C}
$$
is bijective with the bounded inverse 
$$
{\cal E}_\delta =\left(\begin{array}{ccc}E^\delta &E_+^\delta \\
E_-^{\delta } &E_{-+}^\delta  \end{array}\right)
$$

\begin{eqnarray}\label{5.6}
&&E^\delta =E+{\cal O}(\frac{\delta }{h})={\cal O}(h^{-1/2}) \hbox{ in
} {\cal L}(L^2,L^2)\\
&&E_+^\delta =E_++{\cal O}(\frac{\delta }{h^{3/2}})={\cal O}(1) \hbox{ in
} {\cal L}({\bf C},L^2)\nonumber\\
&&E_-^\delta =E_-+{\cal O}(\frac{\delta }{h^{3/2}})={\cal O}(1) \hbox{ in
} {\cal L}(L^2,{\bf C})\nonumber\\
&&E_{-+}^\delta =E_{-+}-\delta E_-QE_++{\cal O}(\frac{\delta ^2}{h^{5/2}}).\nonumber
\end{eqnarray}
As before the eigenvalues of $P_\delta $ are the zeros of
$E_{-+}^\delta $ and we have the d-bar equation
$$
\partial _{\overline{z}}E_{-+}^\delta +f^\delta (z)E_{-+}^\delta =0,$$
$$
f^\delta (z)=\partial _{\overline{z}}R_+ E_+^\delta
+E_-^\delta \partial _{\overline{z}}R_-=f(z)+{\cal
  O}(\frac{1}{h}\frac{\delta }{h^{3/2}}). 
$$

We can solve $\partial
_{\overline{z}}F^\delta =f^\delta $ (making $e^{F^\delta
}E_{-+}^\delta $ holomorphic) with
\ekv{5.7}
{
F^\delta =F+{\cal O}(\frac{\delta }{h^{5/2}})=
F+{\cal O}(\frac{\delta }{h^{3/2}})\frac{1}{h}.
}
\begin{prop}\label{53}
Assume that $0<t\ll 1$, $\delta \ll h^{3/2}$,
\ekv{5.8}
{
\delta t\gg e^{-\frac{1}{C_0h}},\quad t\gg \frac{\delta }{h^{5/2}},
}
where $C_0\gg 1$ is fixed. Then with probability $\ge
1-e^{-\frac{1}{Ch^2}}$, we have
\ekv{5.9}
{
|E_{-+}^\delta (z)|\le e^{-\frac{1}{Ch}}+\frac{C\delta }{h},\ \forall
z\in \Omega .
}
For every $z\in \Omega $, we have with probability $\ge 1-{\cal
  O}(t^2)-e^{-\frac{1}{Ch^2}}$, that
\ekv{5.10}
{
|E_{-+}^\delta (z)|\ge \frac{t\delta }{C},
}
\end{prop}
 
We only give the main idea of the proof which is to notice that
$E_-Q_\omega E_+$ can be written as a sum of independent Gaussian
random variables and is therefore itself a Gaussian random
variable. Applying the standard formula for the variance of such a sum
we get for the variance:
\ekv{5.12}
{\sigma ^2=\sum_{|k|,|j|\le \frac{C_1}{h}}|\widehat{e}_0(j)|^2|\widehat{f}_0(k)|^2,}
where $\widehat{e}_0(j)$, $\widehat{f}_0(j)$ are the Fourier
coefficients of $e_0$, $f_0$. Now we can show that the Fourier
coefficients are ${\cal O}( (h/|j|)^N)$ for every $N\ge 0$, when
$h|j|$ is sufficiently large, so if we take $C_1$ (in the definition
of $Q_\omega $) large enough, we conclude that $\sigma ^2=1+{\cal
  O}(h^\infty )$.\\
The remainder of the proof then consists in showing that $|E_-Q_\omega
E_+|$
is $\ge t$ with probability $\ge 1-{\cal O}(t^2)$ and observing that
when this happens, then the second term in the expression for
$E_{-+}^\delta $ in (\ref{5.6}) is dominant.\hfill{$\Box$}

\begin{prop}\label{54}
Let $\kappa >5/2$ and fix $\epsilon _0\in ]0,1[$ sufficiently
small. Let $\delta =\delta (h)$ satisfy $e^{-\epsilon _0/h}\ll \delta
\ll h^\kappa $, and put $\epsilon =\epsilon (h)=h\ln \frac{1}{\delta
}$. Then with probability $\ge 1-e^{-1/(Ch^2)}$ we have
$|E_{-+}^\delta |\le 1$ for all $z\in \Omega $. \\
For any $z\in \Omega $, we have $|E_{-+}^\delta |\ge e^{-C\epsilon
  /h}$ with probability $\ge 1-{\cal O}(\delta ^2/h^5)$. 
\end{prop}
This follows from Proposition \ref{53} by choosing $t$ such that
$$
\max (\frac{1}{\delta }e^{-\frac{1}{C_0h}},\frac{\delta }{h^{5/2}}, C\delta ^{C-1})\ll t \le {\cal O}(\frac{\delta }{h^{5/2}}), 
$$
which is possible to do since 
$$\frac{1}{\delta }e^{-\frac{1}{C_0h}},\ C\delta ^{C-1}\ll \frac{\delta }{h^{5/2}}.$$

Under the same assumptions, we also have 
$$
|F_\delta -F|\le{\cal O}(\frac{\delta 
}{h^{3/2}})\frac{1}{h}\le {\cal O}(\epsilon )\frac{1}{h}.
$$
Thus for the holomorphic function $u(z)=e^{F_\delta (z)}E_{-+}^\delta
(z)$ we have \begin{itemize}
\item With probability $\ge 1-e^{-1/(Ch^2)}$ we have $|u(z)|\le \exp
  (\Re F(z)+C\epsilon /h)$ for all $z\in \Omega $.
\item For every $z\in \Omega $, we have $|u(z)|\ge \exp (\Re
  F(z)-C\epsilon /h)$ with probability $\ge 1-{\cal O}(\delta
  ^2/h^5)$. 
\end{itemize}

Theorem \ref{31} on the Weyl asymptotics of small random perturbations of
the operator $P=hD+g(x)$ is now a consequence of the following result of 
M.~Hager, that we apply with $\phi =h\Re F$
\begin{prop}\label{56}
Let $\Gamma \Subset {\bf C}$ have smooth boundary and let $\phi $ be a
real valued $C^2$-function defined in a fixed neighborhood of
$\overline{\Gamma }$. Let $z\mapsto u(z;h)$ be a family of holomorphic
functions defined in a fixed neighborhood of $\overline{\Gamma }$, and
let $0<\epsilon =\epsilon (h)\ll 1$. Assume
\begin{itemize}
\item $|u(z;h)|\le \exp (\frac{1}{h}(\phi (z)+\epsilon ))$ for all $z$
  in a fixed neighborhood of $\partial \Gamma $.
\item There exist $z_1,...,z_N$ depending on $h$, with $N=N(h)\asymp
  \epsilon ^{-1/2}$ such that $\partial \Gamma \subset
  \cup_1^ND(z_k,\sqrt{\epsilon })$ such that $|u(z_k;h)|\ge \exp
  (\frac{1}{h}(\phi (z_k)-\epsilon ))$, $1\le k\le N(h)$.
\end{itemize}
Then, the number of zeros of $u$ in $\Gamma $ satisfies
$$
|\#(u^{-1}(0)\cap \Gamma )-\frac{1}{2\pi h}\int _\Gamma \Delta \phi
(z)dxdy| \le C\frac{\sqrt{\epsilon }}{h}.
$$ 
\end{prop}
This is essentially a special case of Theorem \ref{d3}, but we outline the simple and direct proof of Hager in the next subsection.
\subsection{Proof of Proposition \ref{56}, an outline}

Define $\phi _j(z)$ by $i\phi _j(z)=\phi (z_j)+2\partial _z\phi
(z_j)(z-z_j)$. Then
\begin{eqnarray*}
\phi (z)&=&\Re (i\phi _j(z))+R_j(z),\ R_j(z)={\cal O}((z-z_j)^2)\\
\phi '_j(z)&=& \frac{2}{i}\partial _z\phi (z)+{\cal O}((z-z_j)).
\end{eqnarray*}
Consider the holomorphic function 
$$
v_j(z;h)=u(z;h)e^{-\frac{i}{h}\phi _j(z)}.
$$
Then $|v_j(z;h)|\le e^{\frac{1}{h}(\phi (z)-\Re i\phi
  _j(z))}=e^{\frac{1}{h}R_j}\le e^{\frac{C\epsilon }{h}}$, when
$z-z_j={\cal O}(\sqrt{\epsilon })$, while 
$$
|v_j(z_j;h)|\ge e^{-\frac{C\epsilon }{h}}.
$$
In a $\sqrt{\epsilon }$-neighborhood of $z_j$ we put $v=v_j$ and make
the change of variables $w=(z-z_j)/\sqrt{\epsilon }$,
$\widetilde{v}(w)=v(z)$, so that {
$$
|\widetilde{v}(w)|\le e^{C\epsilon /h} \hbox{ on }D(0,2),\
|\widetilde{v}(0)|\ge e^{-C\epsilon /h}.
$$}

Using Jensen's formula we see that the number of zeros $w_1,...,w_N$  of
$\widetilde{v}$ in $D(0,3/2)$ (repeated with their multiplicity) is ${\cal O}(\epsilon /h)$. Factorize:
$$
\widetilde{v}(w)=e^{g(w)}\prod _1^N(w-w_k).
$$
Using the maximum principle and a suitably chosen disc of radius
between 4/3 and 3/2, and then also Harnack's inequality we can follow
a standard procedure to show that 
$$
\Re g(w),\, g'(w)\, = {\cal O}(\epsilon /h) \hbox{ in }D(0,6/5).
$$

\par Using finally that $\partial \Gamma $ is covered by the discs
$D(z_j,\sqrt{\epsilon })$ and using the above representation of $u$ in
each disc, we can show that the number of zeros of $u(\cdot ;h)$ in
$\Gamma $ is equal to 
\begin{eqnarray*}
\Re \frac{1}{2\pi i}\int _{\partial \Gamma }\frac{u'(z)}{u(z)}dz&=&
\Re \frac{1}{2\pi h}\int_{\partial \Gamma }\frac{2}{i}\partial _z\phi
(z)dz+{\cal O}(\frac{\sqrt{\epsilon }}{h})\\
&=& \frac{1}{2\pi h}\int \Delta \phi (x)dxdy+{\cal
  O}(\frac{\sqrt{\epsilon }}{h}).\hskip 2cm \Box
\end{eqnarray*}

\section{The multi-dimensional semi-classical case}\label{mult}
\setcounter{equation}{0}
\subsection{Introduction}\label{int}

In this section we consider general semi-classical operators with multiplicative random perturbations. We follow \cite{Sj08a, Sj08b} which make use of the work \cite{HaSj08}. The use of Theorem \ref{ze2} rather than the corresponding weaker result in \cite{HaSj08} led us to improved remainder estimates in comparison with \cite{Sj08b}. 

Let $X$ be a compact smooth manifold on which we choose a positive density of integration so that the scalar product on $L^2(X)$ is well-defined.
On $X$ we consider an $h$-differential operator $P$ which in local
coordinates takes the form,
\ekv{int.1}
{
P=\sum_{|\alpha |\le m}a_\alpha (x;h)(hD)^\alpha ,
} 
where we use standard multiindex notation and let
$D=D_x=\frac{1}{i}\frac{\partial }{\partial x}$. We assume that the 
coefficients $a_\alpha  $ are uniformly bounded in $C^\infty $ for
$h\in ]0,h_0]$, $0<h_0\ll 1$. (We will also discuss the case when we
only have some Sobolev space control of $a_0(x)$.) Assume
\eekv{int.2}
{
&&a_\alpha (x;h)=a_\alpha ^0(x)+{\cal O}(h) \mbox{ in }C^\infty ,}
{&&a_\alpha (x;h)=a_\alpha (x)\hbox{ is independent of }h\hbox{ for
  }|\alpha |=m.
}
Notice that this assumption is invariant under changes of local
coordinates. 

Also assume that $P$ is elliptic in the classical sense, uniformly
with respect to $h$:
\ekv{int.3}
{
|p_m(x,\xi )|\ge \frac{1}{C}|\xi |^m,
}
for some positive constant $C$, where
\ekv{int.4}
{
p_m(x,\xi )=\sum_{|\alpha |=m}a_\alpha (x)\xi ^\alpha 
}
is invariantly defined as a function on $T^*X$.
It follows that $p_m(T^*X)$ is a closed cone in ${\bf C}$ and we
assume that 
\ekv{int.5}
{
p_m(T^*X)\ne {\bf C}.
}
If $z_0\in {\bf C}\setminus p_m(T^*X)$, we see that $\lambda
z_0\not\in \Sigma (p)$ if $\lambda \ge 1$ is sufficiently large and
fixed, where $\Sigma (p):=p(T^*X)$ and $p$ is the semiclassical
principal symbol
\ekv{int.6}
{
p(x,\xi )=\sum_{|\alpha |\le m}a_\alpha ^0(x)\xi ^\alpha .
}
Actually, (\ref{int.5}) can be replaced by the weaker condition that
$\Sigma (p)\ne {\bf C}$.

\par Standard elliptic theory and analytic Fredholm theory now show
that if we consider $P$ as an unbounded operator: $L^2(X)\to L^2(X)$
with domain ${\cal D}(P)=H^m(X)$ (the Sobolev space of order $m$),
then $P$ has purely discrete spectrum and each eigenvalue has finite algebraic multiplicity.

\par We will need the symmetry assumption
\ekv{int'.7}{P^*=\Gamma P\Gamma ,}
where $P^*$ denotes the formal complex adjoint of $P$ in $L^2(X,dx)$,
and $dx$ is the fixed smooth positive density of integration
and $\Gamma $ is the antilinear operator of complex conjugation;
$\Gamma u=\overline{u}$. Notice that this assumption implies that 
\ekv{int'.8}
{
p(x,-\xi )=p(x,\xi ),
}
and conversely, if $p$ fulfills (\ref{int'.8}), then we get
(\ref{int'.7}) if we replace $P$ by $\frac{1}{2}(P+\Gamma P^*\Gamma $),
which has the same semi-classical principal symbol $p$. Actually, (\ref{int'.7}) can be formulated more simply by saying that $P$ is symmetric for the bilinear form $\int_Xu(x)v(x)dx$.

\par Let $V_z(t):=\mathrm{vol\,}(\{ \rho \in T^*X;\,|p(\rho
)-z|^2 \le t\} )$. For $\kappa \in ]0,1]$, $z\in {\bf C}$, we consider
the property that
\ekv{int.6.2}{V_z(t)={\cal O}(t^ \kappa ),\ 0\le t \ll 1.} 
Since $r\mapsto p(x,r\xi )$ is a polynomial of degree $m$ in $r$ with
non-vanishing leading coefficient, we see that (\ref{int.6.2}) holds
with $\kappa =1/(2m)$.

The random potential will be of the form 
\ekv{int.6.3}
{q_\omega (x)=\sum_{0<\mu _k\le L}\alpha _k(\omega )\epsilon
_k(x),\ |\alpha |_{{\bf C}^D}\le R,}
where $\epsilon _k$ is the orthonormal basis of eigenfunctions of
$h^2\widetilde{R}$, where $\widetilde{R}$ is an $h$-independent
positive elliptic 2nd order operator on $X$ with smooth 
coefficients. Moreover, $h^2\widetilde{R}\epsilon _k=\mu
_k^2 \epsilon _k$, $\mu _k>0$ and we may assume for simplicity that the $\mu _k$ form a (non-strictly) increasing sequence. 
We choose $L=L(h)$, $R=R(h)$ in the interval
\eekv{int.6.4}
{h^{\frac{\kappa -3n}{s-\frac{n}{2}-\epsilon }}\ll L\le Ch^{-M},&&
M\ge \frac{3n-\kappa }{s-\frac{n}{2}-\epsilon },}
{\frac{1}{C}h^{-(\frac{n}{2}+\epsilon )M+\kappa -\frac{3n}{2}}\le R\le
 C h^{-\widetilde{M}},&& \widetilde{M}\ge \frac{3n}{2}-\kappa
  +(\frac{n}{2}+\epsilon )M,}
for some $\epsilon \in ]0,s-\frac{n}{2}[$, $s>\frac{n}{2}$,
so by Weyl's law for the large eigenvalues of elliptic
self-adjoint operators, the dimension $D$ is of the order of magnitude
$(L/h)^n$. We introduce the  small parameter 
$\delta =\tau _0 h^{N_1+n}$, $0<\tau _0\le \sqrt{h}$, where 
\ekv{int.6.4.3}
{
N_1:=\widetilde{M}+sM+\frac{n}{2}.
} 
The randomly perturbed operator is
\ekv{int.6.4.5}
{
P_\delta =P+\delta h^{N_1}q_\omega =:P+\delta Q_\omega .
}

\par The random variables $\alpha _j(\omega )$ will have a joint probability distribution \ekv{int.6.5}{P(d\alpha )=C(h)e^{\Phi (\alpha ;h)}L(d\alpha ),} where for some $N_4>0$, \ekv{int.6.6}{ |\nabla _\alpha \Phi |={\cal O}(h^{-N_4}),} $L(d\alpha )$ is the Lebesgue measure and we use the standard $\ell^2$ norm on ${\bf C}^D$. ($C(h)$ is the normalizing constant, assuring that the probability of $B_{{\bf C}^D}(0,R)$ is equal to 1.)

\par We also need the parameter 
\ekv{int.6.7.5}{\epsilon _0(h)=(h^{\kappa }+h^n\ln 
\frac{1}{h})(\ln \frac{1}{\tau _0}+(\ln \frac{1}{h})^2)} and assume
that $\tau _0=\tau _0(h)$ is not too small, so that $\epsilon _0(h)$ is
small. Let $\Omega \Subset {\bf C}$ be open, simply connected, not
entirely contained in $\Sigma (p)$. The main result of this section is:
\begin{theo}\label{int1} Under the assumptions above, let 
$\Gamma \Subset \Omega $ have smooth boundary, let $\kappa \in
]0,1]$ be the parameter in \no{int.6.3}, \no{int.6.4}, \no{int.6.7.5} and assume that 
\no{int.6.2} holds uniformly for $z$ in a
neighborhood of $\partial \Gamma $.  Then there
exists a constant $C>0$ such that for
$C^{-1}\ge r>0$,
$\widetilde{\epsilon }\ge C \epsilon _0(h)$ 
we have with probability 
\ekv{int.6.8}{
\ge 1-\frac{C\epsilon _0(h)}
{rh^{n+\max (n(M+1), N_4+\widetilde{M})}}
e^{-\frac{\widetilde{\epsilon }}{C\epsilon _0(h)}} }
that:
\eekv{int.7}
{
&&|
\#(\sigma (P_\delta )\cap \Gamma )-\frac{1}{(2\pi h)^n
}\mathrm{vol\,}(p^{-1}(\Gamma ))
|\le
}
{&&
\frac{C}{h^n}\left( \frac{\widetilde{\epsilon }}{r}
+\mathrm{vol\,}(p^{-1}(\partial
\Gamma +D(0,r)))
 \right).}
Here $\#(\sigma (P_\delta )\cap \Gamma )$ denotes the number of
eigenvalues of $P_\delta $ in $\Gamma $, counted with their algebraic multiplicity.
\end{theo}

Actually, the theorem holds for the slightly more general
operators, obtained by replacing $P$ by $P_0=P+\delta
_0(h^{\frac{n}{2}}q_1^0+q_2^0)$, where $\Vert q_1^0\Vert_{H^s_h}\le
1$, $\Vert q_2\Vert_{H^s}\le 1$, $0\le \delta _0\le h$. Here, $H^s$ is
the standard Sobolev space and $H_h^s$ is the same space with the
natural semiclassical $h$-dependent norm. See Subsection \ref{hs}. This allows us in principle to consider more general random perturbations and will be used in Section \ref{alm}.

We also have a result valid simultaneously for a
family ${\cal C}$ of domains $\Gamma \subset \Omega $ satisfying the
assumptions of Theorem \ref{int1} uniformly in the natural sense:
With a probability 
\ekv{int.8}{
\ge 1-\frac{{\cal O}(1)\epsilon _0(h)}{r^2h^{n+\max (n(M+1), N_4+\widetilde{M})}}e^{-\frac{\widetilde{\epsilon }}{C\epsilon _0(h)}}, } the
estimate \no{int.7} holds simultaneously for all $\Gamma \in {\cal C}$.
\begin{remark}\label{int1.5}
{\rm If $\kappa >1/2$, then $\mathrm{vol\,}p^{-1}(\partial \Gamma +D(0,r))={\cal O}(r^{2\kappa -1})$, where the exponent $2\kappa -1$ is $>0$. More generally, if 
$$
\mathrm{vol\,}(p^{-1}(\partial \Gamma +D(0,r)))={\cal O}(r^{\alpha }),
$$
for some $\alpha \in ]0,1]$, then we can choose $r=\widetilde{\epsilon }^{\frac{1}{\alpha +1}}$ and obtain that that the right hand side in (\ref{int.7}) is ${\cal O}(1)h^{-n}\widetilde{\epsilon }^{\frac{\alpha }{\alpha +1}}$ showing that we have Weyl asymptotics. Notice here that if $z$ is not a critical value of $p$, in the sense that $d\Re p(\rho )$ and $\Im p(\rho )$ are independent whenever $p(\rho )=z$, then (\ref{int.6.2}) holds with $\kappa =1$.

\par In the proof we replace the zero counting proposition from \cite{HaSj08} by the stronger Theorem \ref{ze2} leading to an improved remainder estimate. It may be possible (though we have not yet checked the details) to replace the right hand side in (\ref{int.7}) by 
$$
\frac{C}{h^n}\mathrm{vol\,}(p^{-1}(\partial \Gamma +D(0,h^{\frac{1}{2}-\epsilon }))),
$$
for any fixed $\epsilon >0$, and also to let $\Gamma $ be $h$-dependent of a suitable Lipschitz class as in section \ref{ze}.}

\end{remark}
\begin{remark}\label{int2}
{\rm When $\widetilde{R}$ has real coefficients, we may assume that the
eigenfunctions $\epsilon _j$ are real. Then (cf Remark 8.3 in \cite{Sj08a}) we may
restrict $\alpha $ in (\ref{int.6.3}) to be in ${\bf R}^D$ so that
$q_\omega $ is real, still with $|\alpha |\le R$, and change
$C(h)$ in (\ref{int.6.5}) so that $P$ becomes a probability measure on 
$B_{{\bf R}^D}(0,R)$. Then Theorem \ref{int1} remains valid.}
\end{remark}
\begin{remark}\label{int3}
{\rm The assumption (\ref{int'.7}) cannot be
completely eliminated. Indeed, let $P=hD_x+g(x)$ on ${\bf T}={\bf
  R}/(2\pi {\bf Z})$ where $g$ is smooth and complex valued. Then (cf
Hager \cite{Ha06a}) the spectrum of $P$ is contained in the line 
$\Im z = \int_0^{2\pi }\Im g(x)dx/(2\pi )$. This line will vary only very
little under small multiplicative perturbations of $P$ so 
Theorem \ref{int1} cannot hold in this case. On the other hand, for other classes of perturbations, like the ones in Section \ref{one} or in \cite{HaSj08}, the symmetry assumption can be dropped.}
\end{remark}
In the remainder of this section, we shall outline the proof of Theorem \ref{int1} following \cite{Sj08b, Sj08a}.

\subsection{Semiclassical Sobolev spaces and multiplication}
\label{al}
We let $H_h^s({\bf R}^n)\subset {\cal S}'({\bf R}^n)$, $s\in {\bf R}$, 
denote the semiclassical Sobolev space of order
$s$ equipped with the norm $\Vert \langle hD\rangle^s u\Vert$ where
the norms are the ones in $L^2$, $\ell^2$ or the corresponding
operator norms if nothing else
is indicated. Here $\langle hD\rangle= (1+(hD)^2)^{1/2}$. In
\cite{Sj08a} we recalled the following result:
\begin{prop}\label{al1}
Let $s>n/2$. Then there exists a constant $C=C(s)$ such that for all
$u,v\in H_h^s({\bf R}^n)$, we have $u\in L^\infty ({\bf R}^n) $, 
$uv\in H_h^s({\bf R}^n)$ and 
\ekv{al.1}
{
\Vert u\Vert_{L^\infty }\le Ch^{-n/2}\Vert u\Vert_{H_h^s},
}
\ekv{al.2}
{
\Vert uv\Vert_{H_h^s} \le Ch^{-n/2} \Vert u\Vert_{H_h^s} \Vert v\Vert_{H_h^s}.
}
\end{prop}

We cover $X$ by
finitely many coordinate neighborhoods $X_1,...,X_p$ and for
each $X_j$, we let $x_1,...,x_n$ denote the corresponding local 
coordinates on $X_j$. Let $0\le \chi _j\in C_0^\infty (X_j)$ have the
property that $\sum_1^p\chi _j >0$ on $X$. Define $H_h^s(X)$ to be the
space of all $u\in {\cal D}'(X)$ such that 
\ekv{al.4}
{
\Vert u\Vert_{H_h^s}^2:=\sum_1^p \Vert \chi _j\langle hD\rangle^s \chi
_j u\Vert ^2 <\infty .
}
It is standard to show that this definition does not depend on the
choice of the coordinate neighborhoods or on $\chi _j$. With different
choices of these quantities we get norms in \no{al.4} which are
uniformly equivalent when $h\to 0$. In fact, this follows from the
$h$-pseudodifferential calculus on manifolds with symbols in the
H\"ormander space $S^m_{1,0}$, that we quickly reviewed in the
appendix in \cite{Sj08a}. See also \cite{Sj08b}, Section 4.
An equivalent definition of $H_h^s(X)$ is the following: Let 
\ekv{al.5}
{
h^2\widetilde{R}=\sum (hD_{x_j})^*r_{j,k}(x)hD_{x_k}
}
be a non-negative elliptic operator with smooth coefficients on $X$,
where the star indicates that we take the adjoint with respect to the
fixed positive smooth density on $X$. Then $h^2\widetilde{R}$ is
essentially self-adjoint with domain $H^2(X)$, so
$(1+h^2\widetilde{R})^{s/2}:L^2\to L^2$ is a closed densely defined
operator for $s\in {\bf R}$, which is bounded precisely when $s\le
0$. Standard methods allow to show that $(1+h^2\widetilde{R})^{s/2}$
is an $h$-pseudodifferential operator with symbol in $S^s_{1,0}$ and
semiclassical principal symbol given by $(1+r(x,\xi ))^{s/2}$, where
$r(x,\xi )=\sum_{j,k}r_{j,k}(x)\xi _j\xi _k$ is the semiclassical
principal symbol of $h^2\widetilde{R}$.  See the appendix in
\cite{Sj08a}.
The
$h$-pseudodifferential calculus gives for every $s\in {\bf R}$:
\begin{prop}\label{al2}
  $H_h^s(X)$ is the space of all $u\in {\cal D}'(X)$ such that 
$(1+h^2\widetilde{R})^{s/2}u\in L^2$ and the norm $\Vert u\Vert_{H_h^s}$ is
equivalent to $\Vert (1+h^2\widetilde{R})^{s/2}u\Vert$, uniformly when $h\to 0$.
\end{prop}
\begin{remark}\label{al3}
\rm From the first definition we see that Proposition \ref{al1} remains
valid if we replace ${\bf R}^n$ by a compact $n$-dimensional 
manifold $X$.
\end{remark}

\par Of course, $H_h^s(X)$ coincides with the standard Sobolev space
$H^s(X)$ and the norms are equivalent for each fixed value of $h$, but
not uniformly with respect to $h$. The following variant of
Proposition \ref{al1} will be useful when studying the high
energy limit in Section \ref{alm}.
\begin{prop}\label{al4}
Let $s>n/2$. Then there exists a constant $C=C_s>0$ such that 
\ekv{al.6}
{
\Vert uv\Vert_{H_h^s}\le C\Vert u\Vert_{H^s}\Vert v\Vert_{H_h^s},\
\forall u\in H^s({\bf R}^n),\, v\in H_h^s({\bf R}^n).
}
The result remains valid if we replace ${\bf R}^n$ by $X$.
\end{prop}
The proof is straight forward. We work in local coordinates and make a Fourier transform. Then we have to estimate convolutions in certain weighted $L^2$ spaces. See \cite{Sj08b} for the details.

\subsection{$H^s$-perturbations and eigenfunctions}\label{hs}

\par Let $S^m(T^*X)=S^m_{1,0}(T^*X)$,
$S^m(U\times {\bf R}^n)=S^m_{1,0}(U\times {\bf R}^n)$ denote the
classical H\"ormander symbol spaces, where $U\subset {\bf R}^n$ is
open. The condition (\ref{int.5}) implies that the closure of the image of $p$ is not equal to the whole complex plane and (as in
\cite{Ha06b, HaSj08} we can find $\widetilde{p}\in S^m(T^*X)$
which is equal to $p$ outside any given fixed neighborhood of
$p^{-1}(\overline{\Omega })$ such that $\widetilde{p}-z$ is
non-vanishing, for any $z\in \overline{\Omega }$. Let
$\widetilde{P}=P+\mathrm{Op}_h(\widetilde{p}-p)$, where
$\mathrm{Op}_h(\widetilde{p}-p)$ denotes any reasonable quantization
of $(\widetilde{p}-p)(x,h\xi )$. (See for instance the appendix in 
\cite{Sj08a}.) Then $\widetilde{P}-z:H^m_h(X)\to H^0_h(X)$ has a
uniformly bounded inverse for $z\in \overline{\Omega }$ and $h>0$
small enough. Now (see for instance \cite{HaSj08, Sj08a}) the eigenvalues of
$P$ in $\Omega $, counted with their algebraic multiplicity, coincide
with the zeros of the function $z\mapsto \det
((\widetilde{P}-z)^{-1}(P-z))=\det
(1-(\widetilde{P}-z)^{-1}(\widetilde{P}-P))$. Notice here that $(\widetilde{P}-z)^{-1}(\widetilde{P}-P)$ is of trace class so the determiant is well-defined (\cite{GoKr}). 

\par Fix $s>n/2$ and consider the perturbed operator
\ekv{hs.1}{
P_\delta =P+\delta (h^{\frac{n}{2}}q_1+q_2)=P+\delta
(Q_1+Q_2)=P+\delta Q,
}
where $q_j\in H^s(X)$, 
\ekv{hs.2}{
\Vert q_1\Vert_{H^s_h}\le 1,\ \Vert q_2\Vert_{H^s}\le 1,\ 0\le \delta
\ll 1.
}
According to Propositions \ref{al1}, \ref{al4}, $Q={\cal
  O}(1):H_h^s \to H_h^s$ and hence by duality and interpolation,
\ekv{hs.3}
{
Q={\cal O}(1):H_h^\sigma \to H_h^\sigma ,\ -s\le \sigma \le s.
} 

\par Again, the spectrum of $P_\delta $ in $\Omega $
 is discrete and coincides with the set of zeros of 
\ekv{hs.4}
{
\det ((\widetilde{P}_\delta -z)^{-1}(P_\delta -z))=\det (1-
(\widetilde{P}_\delta -z)^{-1}(\widetilde{P}-P)),}
where $\widetilde{P}_\delta :=\widetilde{P}+\delta Q$.
Here $(\widetilde{P}-z)^{-1}={\cal O}(1):H_h^\sigma \to H_h^\sigma $
for $\sigma $ in the same range and by an easy perturbation argument, we get the same
conclusion for $(\widetilde{P}_\delta -z)^{-1}$.  

\par
Put 
\ekv{hs.5}
{
P_{\delta ,z}:=(\widetilde{P}_\delta -z)^{-1}(P_\delta -z)=
1-(\widetilde{P}_\delta -z)^{-1}(\widetilde{P}-P)=:1-K_{\delta ,z},
}
\ekv{hs.6}
{
S_{\delta ,z}:=P_{\delta ,z}^*P_{\delta ,z}=1-(K_{\delta ,z}+K_{\delta ,z}^*-K_{\delta ,z}^*K_{\delta ,z})=:1-L_{\delta ,z}.
}
Clearly, 
\ekv{hs.7}
{K_{\delta ,z},
L_{\delta ,z}={\cal O}(1):H_h^{-s}\to H_h^s.
}
For $0\le \alpha \le 1/2$, let $\pi _\alpha =1_{[0,\alpha ]}(S_{\delta
,z})$. Then using some simple functional calculus we showed in \cite{Sj08a}, that
\ekv{hs.8}
{
\pi _\alpha ={\cal O}(1): H_h^{-s}\to H_h^s.
}

\par We have the corresponding result for $P_\delta -z$. Let 
\ekv{hs.9}
{
S_\delta =(P_\delta -z)^*(P_\delta -z)
}
be defined as the Friedrichs extension from $C^\infty (X)$ with
quadratic form domain $H_h^m(X)$. For $0\le \alpha \le {\cal O}(1)$,
we now put $\pi _\alpha =1_{[0,\alpha ]}(S_\delta )$. Then as in
\cite{Sj08a}, we see that this new spectral projection also fulfils 
(\ref{hs.8}), for $0\le \alpha \ll 1$.

\subsection{Some functional and pseudodifferential cal\-culus}\label{fu}

\par Let $P$ be of the form (\ref{int.1}) and let $p$ in (\ref{int.6})
be the corresponding semi-classical principal symbol. Assume classical
ellipticity as in (\ref{int.3}) and let $z\in {\bf C}$ be fixed
throughout this subsection. 
Let 
\ekv{fu.1}
{
S=(P-z)^*(P-z),
}
viewed as the self-adjoint Friedrichs 
extension from $C^\infty $. Later on we will also consider a different choice of $S$, 
namely 
\ekv{fu.1.5}{
S=P_z^*P_z,\hbox{ where }P_z=(\widetilde{P}-z)^{-1}(P-z)} and
$\widetilde{P}$ is defined prior to (\ref{hs.1}). The main goal is to
make a trace class
study of $\chi (\frac{1}{\alpha }S)$ when $0<h\le \alpha \ll 1$, $\chi
\in C_0^\infty ({\bf R})$. With the second choice of $S$, we shall
also study $\ln \det (S+\alpha \chi (\frac{1}{\alpha }S))$, when $\chi
\ge 0$, $\chi (0)>0$. The main step will be to get enough information
about the resolvent $(w-\frac{1}{\alpha }S)^{-1}$ for $w={\cal O}(1)$,
$\Im w\ne 0$ and then apply the Cauchy-Riemann-Green-Stokes formula
\ekv{fu.2}
{
\chi (\frac{1}{\alpha }S)=-\frac{1}{\pi }\int \frac{\partial
  \widetilde{\chi }(w)}{\partial \overline{w}}(w-\frac{1}{\alpha
}S)^{-1}L(dw ),
}
where $\widetilde{\chi }\in C_0^\infty ({\bf C})$ is an almost
holomorphic extension of $\chi $, so that 
\ekv{fu.3}
{
\frac{\partial \widetilde{\chi }}{\partial \overline{w}}={\cal O}(
|\Im w|^\infty ).
}
Thanks to (\ref{fu.3}) we can work in symbol classes with some
temparate but otherwise unspecified growth in $1/|\Im w|$.

\par Let 
\ekv{fu.4}
{
s=|p-z|^2
}
be the semiclassical principal symbol of $S$ in (\ref{fu.1}). A basic
weight function in our calculus will be 
\ekv{fu.5}
{
\Lambda :=\left(\frac{\alpha +s}{1+s} \right)^{\frac{1}{2}},
}
satisfying $\sqrt{\alpha }\le \Lambda \le 1$.

\par As a preparation and motivation for the calculus, we first
consider symbol properties of $1+\frac{s}{\alpha }$ and its powers.
\begin{prop}\label{fu1}
For every choice of local coordinates $x$ on $X$, let $(x,\xi )$
denote the corresponding canonical coordinates on $T^*X$. Then for all 
$\ell\in {\bf R}$, $\widetilde{\alpha },\beta \in {\bf N}^n$, we have 
uniformly in $\xi $ and locally uniformly in $x$:
\ekv{fu.6}
{
\partial _x^{\widetilde{\alpha }}\partial _\xi^\beta
(1+\frac{s}{\alpha })^\ell ={\cal O}(1) (1+\frac{s}{\alpha })^\ell
\Lambda ^{-|\widetilde{\alpha }|-|\beta |}\langle \xi \rangle
^{-|\beta |}.
}
\end{prop} 
The proof (\cite{Sj08b}) is straight forward and the same can be said about the proof of

\begin{prop}\label{fu2} (\cite{Sj08b}) Let $w$ vary in some bounded subset of ${\bf C}$.
For all $\ell\in {\bf R}$, $\widetilde{\alpha },\beta \in {\bf N}^n$,
there exists $J\in {\bf N}$, such that 
\ekv{fu.9}
{
\partial _x^{\widetilde{\alpha }}\partial _\xi ^\beta
(w-\frac{s}{\alpha })^\ell ={\cal O}(1)(1+\frac{s}{\alpha })^\ell 
\Lambda ^{-|\widetilde{\alpha }|-|\beta |}\langle \xi \rangle^{-|\beta
  |}|\Im w|^{-J},
}
uniformly in $\xi $ and locally uniformly in $x$.
\end{prop}

\par We now define our new symbol spaces.
\begin{dref}\label{fu3}
Let $\widetilde{m}(x,\xi )$ be a weight function of the form 
$\widetilde{m}(x,\xi )=\langle \xi \rangle ^k\Lambda^\ell $. We say
that the family $a=a_w\in C^\infty (T^*X)$, $w\in D(0,C)$, belongs to $S_\Lambda (\widetilde{m})$ if
for all $\widetilde{\alpha },\beta \in {\bf N}^n$ there exists $J\in
{\bf N}$ such that 
\ekv{fu.10}
{
\partial _x^{\widetilde{\alpha }}\partial _\xi ^\beta a=
{\cal O}(1)\widetilde{m}(x,\xi )\Lambda ^{-|\widetilde{\alpha
  }|-|\beta |}\langle \xi \rangle^{-|\beta |}|\Im w|^{-J}.
} 
\end{dref}

Here, as in Proposition \ref{fu2}, it is understood that the
estimate is expressed in canonical coordinates and is locally uniform
in $x$ and uniform in $\xi $. Notice that the set of estimates
(\ref{fu.10}) is invariant under changes of local coordinates in $X$.

\par Let $U\subset X$ be a coordinate neighborhood that we shall view
as a subset of ${\bf R}^n$ in the natural way. Let $a\in S_\Lambda
(T^*U,\widetilde{m})$ be a symbol as in Definition \ref{fu3} so that
(\ref{fu.10}) holds uniformly in $\xi $ and locally uniformly in
$x$. For fixed values of $\alpha $, $w$ the symbol $a$ belongs
to $S^k_{1,0}(T^*U)$, so the classical $h$-quantization
\ekv{fu.11}
{
Au=\mathrm{Op}_h(a)u(x)=\frac{1}{(2\pi h)^n}\iint
e^{\frac{i}{h}(x-y)\cdot \eta }a(x,\eta ;h)u(y)dyd\eta 
}
is a well-defined operator $C_0^\infty (U)\to C^\infty (U)$, ${\cal
  E}'(U)\to {\cal D}'(U)$. In order to develop our rudimentary
calculus on $X$ we need a pseudolocal property for the
distribution kernel $K_A(x,y)$, whose proof is also routine (see \cite{Sj08b}).
\begin{prop}\label{fu4}
For all $\widetilde{\alpha },\beta \in {\bf N}^n$, $N\in {\bf N}$,
there exists $M\in {\bf N}$ such that 
\ekv{fu.12}
{
\partial _x^{\widetilde{\alpha }}\partial _y^\beta K_A(x,y)={\cal
  O}(h^N|\Im w|^{-M}),
}
locally uniformly on $U\times U\setminus \mathrm{diag}(U\times U)$.
\end{prop}

This means that if $\phi ,\psi \in C_0^\infty (U)$ have disjoint
supports, then for every $N\in {\bf N}$, there exists $M\in {\bf N}$
such that $\phi A\psi :H^{-N}({\bf R}^n)\to H^N({\bf R}^n)$ with norm
${\cal O}(h^N|\Im w|^{-M})$, and this leads to a simple way of
introducing pseudo\-differential operators on $X$: Let $U_1,...,U_s$ be
coordinate neighborhoods that cover $X$. Let $\chi _j\in C_0^\infty
(U_j)$ form a partition of unity and let $\widetilde{\chi }_j\in
C_0^\infty (U_j)$ satisfy $\chi _j\prec \widetilde{\chi }_j$ in the
sense that $\widetilde{\chi }_j$ is equal
to 1 near $\mathrm{supp\,}(\chi _j)$. Let $a=(a_1,...,a_s)$, where
$a_j\in S_\Lambda (\widetilde{m})$. Then we quantize $a$ by the
formula:
\ekv{fu.13}
{
A=\sum_1^s \widetilde{\chi }_j\circ \mathrm{Op}_h(a_j)\circ \chi _j.
}
This is not an invariant quantization procedure but it
will suffice for our purposes. 

Using integration by parts and stationary phase we can 
study the composition to the left with non-exotic
pseudodifferential operators and we obtain the following result for a coordinate neighborhood:
\begin{prop}\label{fu5} (\cite{Sj08b}).
Let $A=\mathrm{Op}_h(a)$, $a\in S_{1,0}(m_1)$, $B=\mathrm{Op}_h(b)$,
$b\in S_\Lambda (m_2)$ and assume that $b$ has uniformly compact
support in $x$. Then $A\circ B=\mathrm{Op}_h(c)$, where $c$ belongs to
$S_\Lambda
(m_1m_2)$ and has the asymptotic expansion
$$
c\sim \sum \frac{h^{|\beta |}}{\beta !}\partial _\xi ^\beta a(x,\xi )D_x^\beta
b(x,\xi ),
$$
in the sense that for every $N\in {\bf N}$,
$$
c= \sum_{|\beta |<N} \frac{h^{|\beta |}}{\beta !}\partial _\xi ^\beta a(x,\xi )D_x^\beta
b(x,\xi ) + r_N(x,\xi ;h),
$$
where 
$r_N\in S_\Lambda (\frac{m_1m_2}{(\Lambda \langle \xi \rangle)^N}h^N)$.
\end{prop}

\par We have a parametrix construction for $w-\frac{1}{\alpha
}S$, still with $S$ as in (\ref{fu.1}). Let us first work in a coordinate neighborhood $U$, viewed as an open set in ${\bf R}^n$. Then for every $N\in {\bf N}$ we can construct a symbol
\ekv{fu.22}
{
E_N\equiv \frac{1}{w-\frac{s}{\alpha }}\ \mathrm{mod}\ S_\Lambda
(\frac{\alpha }{\Lambda ^2\langle \xi \rangle ^{2m}}\frac{h}{\Lambda
  ^2\langle \xi \rangle}),
}
such that on the symbol level
\ekv{fu.23}
{
(w-\frac{1}{\alpha }S)\# E_N=1+r_N,\ r_N\in S_\Lambda ((\frac{h}{\Lambda ^2\langle \xi 
\rangle})^{N+1}),
}
\ekv{fu.24}
{
E_N\hbox{ is a holomorphic function of }w,\hbox{ for }|\xi |\ge C,
}
where $C$ is independent of $N$.

\par Now we return to the manifold situation and denote by
$E_N^{(j)}$, $r_N^{(j)}$ the corresponding symbols on $T^*U_j$,
constructed above. Denote the operators by the same symbols, and put
on the operator level:
\ekv{fu.25}
{
E_N=\sum_{j=1}^s \widetilde{\chi }_jE_N^{(j)}\chi _j,
}
with $\chi _j$, $\widetilde{\chi _j}$ as in (\ref{fu.13}). Then
\eeekv{fu.26}
{
(w-\frac{1}{\alpha }S)E_{N-1}&=&1-\sum_{j=1}^s \frac{1}{\alpha
}[S,\widetilde{\chi }_j]E_{N-1}^{(j)}\chi _j+\sum_{j=1}^s\widetilde{\chi
}_j
r_N^{(j)}\chi _j
}{&=:&1+R_N^{(1)}+R_N^{(2)}
}
{&=:&1+R_N.}
Proposition \ref{fu4} implies that for every $\widetilde{N}$, there
exists an $\widetilde{M}$ such that the trace class norm of
$R_N^{(1)}$ satisfies
\ekv{fu.27}
{
\Vert R_N^{(1)}\Vert_{\mathrm{tr}}\le {\cal O}(h^{\widetilde{N}}|\Im
w| ^{-\widetilde{M}}).
}

\par As for the trace class norm of $R_N^{(2)}$, we can combine standard
facts about such norms for pseudodifferential operators and scaling to get
\ekv{fu.29}
{
\Vert R_N\Vert_{\mathrm{tr}}\le Ch^{-n}|\Im w|^{-M(N)}\iint
\left(\frac{h}{\Lambda ^2\langle \xi \rangle} \right)^N dxd\xi .
}
The contribution to this expression from the region where $\Lambda \ge
1/C$ is ${\cal O}(h^{N-n})|\Im w|^{-M(N)}$. 

\par The volume growth assumption (\ref{int.6.2}), that we now assume
for our fixed $z$, says that
\ekv{fu.30}
{
V(t):=\mathrm{vol\,}(\{ \rho \in T^*X;\, s\le t \})={\cal O}(t^\kappa
),\ 0\le t\ll 1,
}
for $0<\kappa \le 1$. Using this and (\ref{fu.29}) one can show that

\ekv{fu.31}
{
\Vert R_N\Vert_{\mathrm{tr}}\le {\cal O}(1) h^{-n}\alpha ^{\kappa
}\left(\frac{h}{\alpha } \right)^N|\Im w|^{-M(N)}.
}

\par From (\ref{fu.26}), we get 
$$
(w-\frac{1}{\alpha }S)^{-1}=E_{N-1}-(w-\frac{1}{\alpha }S)^{-1}R_N.
$$
Write 
$$
E_{N-1}=\frac{1}{w-\frac{s}{\alpha }}+F_{N-1},\quad F_{N-1}\in S_\Lambda
(\frac{\alpha h}{\Lambda ^4\langle \xi \rangle^{2m+1}}).
$$
More precisely we do this for each $E_{N-1}^{(j)}$ in (\ref{fu.25}). Then
quantize and plug this into (\ref{fu.2}):
\eekv{fu.32}
{
\chi (\frac{1}{\alpha }S)&=&-\frac{1}{\pi} \int \frac{\partial
    \widetilde{\chi }}{\partial
    \overline{w}}\mathrm{Op}_h(\frac{1}{w-\frac{s}{\alpha }})L(dw)
-\frac{1}{\pi }\int \frac{\partial
    \widetilde{\chi }}{\partial
    \overline{w}}F_{N-1} L(dw)
}
{&&-\frac{1}{\pi} \int \frac{\partial
    \widetilde{\chi }}{\partial
    \overline{w}}(w-\frac{1}{\alpha }S)^{-1}R_N L(dw)=:\mathrm{I}+\mathrm{II}+\mathrm{III}.
}
Here by definition,
$$ \mathrm{Op}_h\left( \frac{1}{w-\frac{s}{\alpha
  }}\right)=\sum_{j=1}^s\widetilde{\chi }_j\mathrm{Op}_h\left(\frac{1}{w-\frac{s}{\alpha
  }} \right)\chi _j
$$
with the coordinate dependent quantization appearing to the right. 

\par After some further estimates we get
\ekv{fu.33}{
\mathrm{tr\,}(\mathrm{I})=\frac{1}{(2\pi h)^n}\iint \chi (\frac{s(x,\xi
  )}{\alpha })dxd\xi .
}
As at the last estimate in the proof of Proposition 4.4 in 
\cite{HaSj08} we see that this quantity is ${\cal O}(\alpha
^\kappa h^{-n})$ and more generally, 
$$
\Vert \mathrm{I}\Vert_{\mathrm{tr}}={\cal O}(\alpha ^\kappa h^{-n}).
$$

\par For II, one can show, using the fact that the symbol is holomorphic in
$w$ for large $\xi $, that
$$ \Vert \mathrm{II}\Vert_{\mathrm{tr}}={\cal O}(1)\frac{\alpha ^\kappa }{h^n}\frac{h}{\alpha }. $$

\par It is also clear that 
$$\Vert \mathrm{III}\Vert_{tr}={\cal O}(1)\frac{\alpha ^\kappa
}{h^n}\left(\frac{h}{\alpha } \right)^N.$$

Summing up our estimates, we get the following result:
\begin{prop}\label{fu6}
Let $\chi \in C_0^\infty ({\bf R})$. For $0<h\le \alpha <1$, we have
\ekv{fu.34}
{
\Vert \chi (\frac{1}{\alpha }S)\Vert_{\mathrm{tr}}={\cal
  O}(1)\frac{\alpha ^\kappa }{h^n},
}
\ekv{fu.35}
{
\mathrm{tr\,}\chi (\frac{1}{\alpha }S)=\frac{1}{(2\pi h)^n}
\iint \chi (\frac{s(x,\xi )}{\alpha })dxd\xi +{\cal O}(\frac{\alpha
  ^\kappa }{h^n}\frac{h}{\alpha }).
}
\end{prop}

\begin{remark}\label{fu7}
Using simple $h$-pseudodifferential calculus (for instance as in
the appendix of \cite{Sj08a}, we see that if we redefine $S$ as in
(\ref{fu.1.5}), then in each local coordinate chart,
$S=\mathrm{Op}_h(S)$, where $S\equiv s\,\mathrm{mod\,}S_{1,0}(h\langle
\xi \rangle^{-1})$ and $s$ is now redefined as
\ekv{fu.36}
{
s(x,\xi )=\left(\frac{|p(x,\xi )-z|}{|\widetilde{p}(x,\xi )-z|} 
\right)^2.
} 
The discussion goes through without any changes (now with $m=0$) and
we still have Proposition \ref{fu6} with the new choice of $S$, $s$.

In both cases it follows from (\ref{fu.34}) that the number $N(\alpha )$ of eigenvalues of $S$ in the interval $[0,\alpha ]$ satisfies 
\ekv{grny.4}
{{\cal O}(\alpha ^\kappa /h^n).}
\end{remark}

In the remainder of this subsection, we choose $S$, $s$ as in
(\ref{fu.1.5}), (\ref{fu.36}). In this case we notice that $S$ is a
trace class perturbation of the identity, whose symbol is $1+{\cal
  O}(h^\infty /\langle \xi \rangle^\infty )$ and similarly for all its
derivatives, in a region $|\xi |\ge \mathrm{Const}$.

\par Let $0\le \chi \in C_0^\infty ([0,\infty [)$ with $\chi (0)>0$ and
let $\alpha _0>0$ be small and fixed. Using standard
pseudodifferential calculus in the spirit of \cite{MeSj02}, we get
\ekv{fu.37}
{
\ln\det (S+\alpha _0\chi (\frac{1}{\alpha _0}S))=
\frac{1}{(2\pi h)^n}(\iint \ln (s+\alpha _0\chi (\frac{1}{\alpha
  _0}s))dxd\xi +{\cal O}(h)).
}

\par Extend $\chi $ to be an element of $C_0^\infty ({\bf R};{\bf C})$
in such a way that $t+\chi (t)\ne 0$ for all $t\in {\bf R}$. As in
\cite{HaSj08}, we use that 
\ekv{fu.39}
{
\frac{d}{dt}\ln (E+t\chi (\frac{E}{t}))=\frac{1}{t}\psi (\frac{E}{t}),
}
where
\ekv{fu.39.5}
{
\psi (E)=\frac{\chi (E)-E\chi '(E)}{E+\chi (E)},
}
so that $\psi \in C_0^\infty ({\bf R})$. By standard functional
calculus for self-adjoint operators, we have 
\ekv{fu.40}
{
\frac{d}{dt}\ln \det (S+t\chi (\frac{S}{t}))=\mathrm{tr\,}
\frac{1}{t}\psi (\frac{S}{t}).
}
Using (\ref{fu.35}), we then get for $t\ge \alpha \ge h>0$:
$$
\frac{d}{dt}\ln\det (S+t\chi (\frac{1}{t}S))
=\frac{1}{(2\pi h)^n}(\iint \frac{1}{t}\psi (\frac{s}{t})dxd\xi +{\cal
  O}(ht^{\kappa -2})).
$$

Integrating this from $t=\alpha _0$ to $t=\alpha $ and using
(\ref{fu.37}), (\ref{fu.39}), leads to 
\begin{prop}\label{fu8}
If $0\le \chi \in C_0^\infty ([0,\infty [)$, $\chi (0)>0$, we have
uniformly for $0<h\le \alpha \ll 1$
\ekv{fu.42}
{
\ln\det (S+\alpha \chi (\frac{1}{\alpha }S))=\frac{1}{(2\pi h)^n}
(\iint \ln s(x,\xi )dxd\xi +{\cal O}(\alpha ^\kappa \ln \alpha )).
}
Here the remainder term can be replaced by ${\cal O}(\alpha ^\kappa )$
when $\kappa <1$ and by ${\cal O}(\alpha +h\ln \alpha )$ when $\kappa =1$.
\end{prop}
Notice that (\ref{fu.42}) implies the upper bound,
\ekv{grny.5}
{
\ln \det P_z^*P_z\le \frac{1}{(2\pi h)^n}(\iint \ln (s)dxd\xi +
{\cal O}(\alpha ^\kappa \ln \frac{1}{\alpha })).
}

We next consider $P_{\delta ,z}=(\widetilde{P}_\delta
-z)^{-1}(P_\delta -z)=1-K_{\delta ,z}$ with $P_\delta =P+\delta Q $,
$\widetilde{P}_\delta =\widetilde{P}+\delta Q $ as in Subsection \ref{hs} and under the
assumptions \no{hs.4}, \no{hs.6}. Put 
$$
S_{\delta,z} =P^*_{\delta ,z}P_{\delta ,z}=1-K_{\delta ,z}-K_{\delta ,z}^*+K_{\delta ,z}^*K_{\delta ,z},
$$
where $K_{\delta ,z}$ is given by \no{hs.6}, so that
$$
\Vert K_{\delta ,z}\Vert \le {\cal O}(1),\ \Vert K_{\delta
  ,z}\Vert_{\mathrm{tr}}
\le \Vert (\widetilde{P}_\delta -z)^{-1}\Vert \Vert
\widetilde{P}-P\Vert_{\mathrm{tr}}\le {\cal O}(h^{-n}).
$$
Here $\Vert \cdot \Vert_{\mathrm{tr}}$ denotes the trace class norm,
and we refer for instance to \cite{DiSj99} for the standard estimate on
the trace class norm of an $h$-pseudodifferential operator with
compactly supported symbol, that we used for the last estimate.

\par Write $\dot{K}_{\delta ,z}=\frac{\partial }{\partial \delta }K_{\delta
,z}$. Then 
$$
\dot {K}_{\delta ,z}=-(z-\widetilde{P}_\delta
)^{-1}Q(z-\widetilde{P}_\delta )^{-1}(\widetilde{P}-P),
$$
so 
$$
\Vert \dot{K}_{\delta ,z}\Vert \le {\cal O}(\Vert Q\Vert ),\quad \Vert \dot{K}_{\delta ,z}\Vert_{\mathrm{tr}} \le {\cal O}(\Vert Q\Vert h^{-n}).
$$
It follows that 
$$
\Vert \dot{S}_{\delta, z}\Vert \le {\cal O}(\Vert Q\Vert ),\quad \Vert
\dot{S}_{\delta, z}\Vert_{\mathrm{tr}} \le {\cal O}(\Vert Q\Vert h^{-n}).
$$

\par Let $N=N(\alpha ,\delta )$ denote the number of singular values
of $P_{\delta ,z}$ in the interval $[0,\sqrt{\alpha} [$ for $h\ll \alpha \ll
1$. Assume 
\ekv{grny.5.5}
{
\delta \le {\cal O}(h).
}
Then $\Vert S_{\delta ,z}-S_{0,z} \Vert \le {\cal O}(h)$ and from
\no{grny.4} we get 
\ekv{grny.5.6}
{
N(\alpha ,\delta )={\cal O}(\alpha ^\kappa h^{-n}).
}

\par Let $1_\alpha (t)=\max (\alpha ,t)$.
For $0<\epsilon \ll 1 $, let $C^\infty (\overline{{\bf R}}_+)\ni
1_{\alpha ,\epsilon }\ge 1_\alpha $ be equal to $t$ outside a small
neighborhood of $t=0$ and converge to $1_\alpha $ uniformly when
$\epsilon \to 0$. For any fixed $\epsilon >0 $, we put $f(t)=1_{\alpha
,\epsilon }(t)$ for $t\ge 0$ and extend $f$ to ${\bf R}$ in
such a way that $f(t)=t+g(t)$, $g\in C_0^\infty ({\bf R})$.  Let
$\widetilde{f}(t)=t+\widetilde{g}(t)$ be an almost holomorphic
extension of $f$ with $\widetilde{g}\in C_0^\infty ({\bf C})$.  Then
we have:
$$
f(S_{\delta,z} )=S_\delta -\frac{1}{\pi }\int (w-S_{\delta,z}
)^{-1}\overline{\partial }\widetilde{g}(w)L(dw).
$$  
Differentiating with respect to $\delta $ one can show the
identity
$$
\frac{\partial }{\partial \delta }\ln \det f(S_{\delta,z} )=\mathrm{tr\,}(
f(S_{\delta,z} )^{-1}f'(S_{\delta,z} )\dot {S}_{\delta ,z} ).
$$
Now we can choose $f=1_{\alpha,\epsilon }$ such that $|f'(t)|\le 1$
for $t\ge 0$. Then we get the estimate
\begin{eqnarray*}
\frac{\partial }{\partial \delta }\ln \det (1_{\alpha ,\epsilon }(S_{\delta,z}
))&=&\mathrm{tr\,}(1_{\alpha ,\epsilon }(S_{\delta,z} )^{-1}1'_{\alpha
  ,\epsilon }(S_{\delta,z} )\dot {S}_{\delta ,z} )\\
&=& {\cal O}(\frac{\Vert \dot{S}_{\delta ,z} \Vert _{\mathrm{tr}}}{\alpha
})\\
&=& {\cal O}(1)\frac{\Vert Q\Vert}{\alpha h^n}.
\end{eqnarray*}

\par Since $\ln \det 1_\alpha (S_{\delta,z} )=\lim_{\epsilon \to 0}\ln\det
1_{\alpha ,\epsilon }(S_{\delta,z} )$, we can integrate the above
estimate, pass to the limit and obtain
$$
\ln \det 1_\alpha (S_{\delta ,z})=\ln\det 1_{\alpha }(S_{0,z} )+{\cal
  O}(\frac{\delta \Vert Q\Vert}{\alpha h^n}).
$$
With some more work, we also get
\ekv{fu.43}
{
\ln \det 1_\alpha (S_{\delta ,z})=\frac{1}{(2\pi h)^n}
(\iint \ln s(x,\xi ) dxd\xi +{\cal O}(\alpha ^\kappa \ln \alpha )+{\cal O}(\frac{\delta \Vert Q\Vert}{\alpha })).
}

\subsection{Grushin problems}\label{gr}

Let $P:{\cal H}\to {\cal H}$ be a bounded operator, where ${\cal H}$
is a complex separable Hilbert space. Following the standard
definitions (see \cite{GoKr}) we define the singular values of $P$ to
be the decreasing sequence $s_1(P)\ge s_2(P)\ge ...$ of eigenvalues of
the self-adjoint operator $(P^*P)^{1/2}$ as long as these eigenvalues
lie above the supremum of the essential spectrum. If there are only
finitely many such eigenvalues, $s_1(P),...,s_k(P)$ then we define
$s_{k+1}(P)=s_{k+2}(P)=...$ to be the supremum of the essential
spectrum of $(P^*P)^{1/2}$. When $\mathrm{dim\,}{\cal H}=M<\infty $
our sequence is finite (by definition); $s_1\ge s_2\ge ...\ge s_M$,
otherwise it is infinite. Using that if $P^*Pu=s_j^2u$, then
$PP^*(Pu)=s_j^2Pu$ and similarly with $P$ and $P^*$ permuted, we see
that $s_j(P^*)=s_j(P)$. Strictly speaking, $P^*P:\,{\cal N}(P)^\perp
\to {\cal N}(P)^\perp$ and $PP^*:\,{\cal N}(P^*)^\perp
\to {\cal N}(P^*)^\perp$ are unitarily equivalent via the map
$P(P^*P)^{-1/2}:\, {\cal N}(P)^\perp \to {\cal N}(P^*)^\perp$ and its
inverse
$P^*(PP^*)^{-1/2}:\, {\cal N}(P^*)^\perp \to {\cal N}(P)^\perp$. (To
check this, notice that the relation $P(P^*P)=(PP^*)P$ on ${\cal
  N}(P)^\perp$ implies $P(P^*P)^\alpha =(PP^*)^\alpha P$ on the same
space for every $\alpha \in {\bf R}$.)

\par In the case when $P$ is a Fredholm operator of index $0$, it will
be convenient to introduce the increasing sequence $0\le t_1(P)\le
t_2(P)\le ...$ consisting first of all eigenvalues of $(P^*P)^{1/2}$
below the infimum of the essential spectrum and then, if there are
only finitely many such eigenvalues, we repeat indefinitely that
infimum. (The length of the resulting sequence is the
dimension of ${\cal H}$.) When $\mathrm{dim\,}{\cal H}=M<\infty $, we
have $t_j(P)=s_{M+1-j}(P)$. Again, we have $t_j(P^*)=t_j(P)$ (as
reviewed in \cite{HaSj08}). Moreover, in the case when $P$ has a 
bounded inverse, we see that 
\ekv{gr.00}{s_j(P^{-1})=\frac{1}{t_j(P)}.}

\par Let $P$ be a Fredholm operator of index $0$. Let $1\le N<\infty $
and let $R_+:{\cal H}\to {\bf C}^N$, $R_-:{\bf C}^N\to {\cal H}$ be
bounded operators. Assume that 
\ekv{gr.01}{
{\cal P}=\left(\begin{array}{ccc}P &R_-\\ R_+ &0 \end{array}\right):
{\cal H}\times {\bf C}^N\to {\cal H}\times {\bf C}^N
}
is bijective with a bounded inverse
\ekv{gr.02}
{
{\cal E}=\left(\begin{array}{ccc}E &E_+\\E_- &E_{-+} \end{array}\right)
}

\par
Recall (for instance from \cite{SjZw2}) that $P$ has a bounded inverse precisely when 
$E_{-+}$ has, and when this happens we have the relations,
\ekv{s.8}{
P^{-1}=E-E_+E_{-+}^{-1}E_-,\quad E_{-+}^{-1}=-R_+P^{-1}R_-.
}
Recall (\cite{GoKr}) that if $A,B$ are bounded operators, then we have
the general estimates of Ky Fan,
\ekv{s.9}
{
s_{n+k-1}(A+B)\le s_n(A)+s_k(B),
}
\ekv{s.10}
{
s_{n+k-1}(AB)\le s_n(A)s_k(B),
}
in particular for $k=1$, we get
$$
s_n(AB)\le \Vert A\Vert s_n(B),\ s_n(AB)\le s_n(A)\Vert B\Vert,\
s_n(A+B)\le s_n(A)+\Vert B\Vert . 
$$ 
Applying this to the second part of \no{s.8}, we get 
$$
s_k(E_{-+}^{-1})\le \Vert R_-\Vert \Vert R_+\Vert s_k(P^{-1}),\ 1\le
k\le N
$$
implying
\ekv{s.11}
{
t_k(P)\le \Vert R_-\Vert \Vert R_+\Vert t_k(E_{-+}),\ 1\le k\le N.
}
By a perturbation argument, we see that this holds also in the
case when $P$, $E_{-+}$ are non-invertible. 

\par Similarly from the first part of \no{s.8}, we get 
$$
s_k(P^{-1})\le \Vert E\Vert +\Vert E_+\Vert \Vert E_-\Vert s_k(E_{-+}^{-1}),
$$
leading to
\ekv{s.12}
{
t_k(P)\ge \frac{t_k(E_{-+})}{\| E\| t_k(E_{-+})+\Vert E_+\Vert\Vert E_-\Vert}.
}
Again this can be extended to the non-necessarily invertible case by
means of small perturbations.

\par Next, we recall from \cite{HaSj08} a natural construction of an
associated Grushin problem to a given operator. Let $P_0:{\cal H}\to
{\cal H}$ be a Fredholm operator of index $0$ as above. Assume that
the first $N$ singular values $t_1(P_0)\le t_2(P_0)\le ...\le t_N(P_0)$
correspond to discrete eigenvalues of $P_0^*P_0$ and assume that $t_{N+1}(P_0)$ is
strictly positive. In the following we sometimes write $t_j$ instead
of $t_j(P_0)$ for short.

\par Recall that $t_j^2$ are the first
eigenvalues both for $P_0^*P_0$ and $P_0P_0^*$.
%
Let $e_1,...,e_N$ and
$f_1,...,f_N$ be corresponding orthonormal systems of eigenvectors
of $P_0^*P_0$ and $P_0P_0^*$ respectively. They can be chosen so that
\ekv{gr.1} { P_0e_j=t _jf_j,\ P_0^*f_j=t_je_j.
} Define $R_+:L^2\to {\bf C}^N$ and $R_-:{\bf C}^N\to L^2$ by
\ekv{gr.2}{R_+u(j)=(u|e_j),\ R_-u_-=\sum_1^Nu_-(j)f_j.}  
As in \cite{HaSj08}, the Grushin problem \ekv{gr.3} {
\left\{
\begin{array}{ll}P_0u+R_-u_-=v,\\ R_+u=v_+,
 \end{array} \right.  } has a unique solution $(u,u_-)\in L^2\times
{\bf C}^N$ for every $(v,v_+)\in L^2\times {\bf C} ^N$, given by \ekv{gr.4} {
\left\{\begin{array}{ll} u=E^0v+E_+^0v_+,\\ u_-=E_-^0v+E_{-+}^0v_+,
\end{array}\right.}  where
\begin{eqnarray}\label{gr.4.5} E^0_+v_+=\sum_1^Nv_+(j)e_j,& E^0_-v(j)=(v|f_j),\\
E^0_{-+}=-{\rm diag\,}(t _j),& \Vert E^0\Vert \le
{1\over t_{N+1}}.\nonumber
\end{eqnarray}
$E^0$ can be viewed
as the inverse of $P_0$ as an operator from the orthogonal space
$(e_1,e_2,...,e_N)^\perp$ to $(f_1,f_2,...,f_N)^\perp$.
\par
We notice that in this case, the norms of $R_+$ and $R_-$ are equal to
1, so \no{s.11} tells us that $t_k(P_0)\le t_k(E^0_{-+})$ for $1\le k\le
N$, but of course the expression for $E^0_{-+}$ in \no{gr.4.5} implies
equality. 

\par Let $Q\in {\cal L}({\cal H}, {\cal H}) $ and put $P_\delta
=P_0-\delta Q$ (where we sometimes put a minus sign in front of the
perturbation for notational convenience). We are particularly
interested in the case when $Q=Q_\omega u=q_\omega u$ is the operator of
multiplication with a random function $q_\omega $. Here $\delta >0$
is a small parameter. Choose $R_\pm$ as in \no{gr.2}. Then if $\delta
<t_{N+1}$ and $\Vert Q\Vert\le 1$, 
the perturbed Grushin problem 
\ekv{gr.5}{ \left\{
\begin{array}{ll}P_\delta u+R_-u_-=v,\\ R_+u=v_+,
 \end{array} \right.  } is well posed and has the solution 
\ekv{gr.6}
{ \left\{\begin{array}{ll} u=E^\delta v+E_+^\delta v_+,\\
u_-=E_-^\delta +E_{-+}^\delta v_+,
\end{array} \right.}  where
\ekv{gr.6.5}
{
{\cal E}^\delta =\left(\begin{array}{ccc}E^\delta &E_+^\delta \\
E_-^\delta &E_{-+}^\delta
 \end{array}\right)
}
is obtained from ${\cal E}^0$ by
\ekv{gr.6.6}{
{\cal E}^\delta ={\cal E}^0\left( 1-\delta
\left(\begin{array}{ccc}Q E^0 & Q E_+^0 \\
0&0 \end{array}\right)\right) ^{-1}.
}
Using the Neumann series, we get
\ekv{gr.7} {
E_{-+}^\delta =E_{-+}^0+\delta E_-^0 Q E_+^0+ \delta^2
E_-^0 Q E^0 Q E_+^0+ \delta^3 E_-^0 Q
(E^0 Q )^2 E_+^0+...  }
We also get 
\ekv{gr.7.1}
{
E^\delta =E^0+\sum_1^{\infty }\delta ^kE^0(QE^0)^k
}
\ekv{gr.7.2}
{
E_+^\delta =E_+^0+\sum_1^{\infty }\delta ^k(E^0Q)^kE_+^0
}

\ekv{gr.7.3}
{
E_-^\delta =E_-^0+\sum_1^{\infty }\delta ^kE_-^0(QE^0)^k.
}
\par The leading perturbation in $E_{-+}^\delta $ is
$\delta M $, where $M =E_-^0 Q E_+^0: {\bf C}^N\to
{\bf C}^N$ has the matrix \ekv{gr.8}{ M(\omega )_{j,k}=(Q e_k|f_j),
} which in the multiplicative case reduces to \ekv{gr.9} { M(\omega
)_{j,k}=\int q (x)e_k(x)\overline{f_j(x)}dx.  }

\par Put $\tau _0=t_{N+1}(P_0)$ and recall the 
assumption
\ekv{s.13}
{
\Vert Q\Vert \le 1.
}
Then, if $\delta \le \tau _0/2$, the new Grushin problem is well posed
with an inverse ${\cal E}^{\delta }$ given in
\no{gr.6.5}--\no{gr.7.3}. We get 
\ekv{s.14}
{
\Vert E^\delta \Vert \le \frac{1}{1-\frac{\delta }{\tau _0}}
\Vert E^0\Vert \le \frac{2}{\tau _0},\quad
\| E_\pm ^\delta \| \le \frac{1}{1-\frac{\delta }{\tau _0}}\le 2,  
}
\ekv{s.15}
{
\Vert E_{-+}^\delta -(E_{-+}^0+\delta E_-^0QE_+^0)\Vert \le
\frac{\delta ^2}{\tau _0}\frac{1}{1-\frac{\delta }{\tau_0}}\le 
2\frac{\delta ^2}{\tau _0}.
}
Using this in \no{s.11}, \no{s.12} together with the fact that $t_k(E_{-+}^\delta )
\le 2\tau_0$, we get
\ekv{s.16}
{
\frac{t_k(E^\delta _{-+})}{8}\le t_k(P_\delta )\le t_k(E^\delta _{-+}).
}

\begin{remark}\label{gr0} \rm under suitable assumptions, the 
preceding discussion can be extended to the case of unbounded
operators. This turns out to be the case for our elliptic operator 
$P_\delta $.\end{remark}

\par We next collect some facts from \cite{HaSj08}. The first result
follows from Section 2 in that paper.
\begin{prop}\label{gr1}
Let $P:{\cal H}\to {\cal H}$ be bounded and assume that $P-1$ is of
trace class, so that $P$ is Fredholm of index $0$. Let $R_+,R_-,{\cal
  P}, {\cal E}={\cal P}^{-1}$ be as in \no{gr.01}, \no{gr.02}. Then
${\cal P}$ is also a trace class perturbation of the identity operator
and 
\ekv{grny.1}{\det P=\det {\cal P}\det E_{-+}.
} 
\end{prop}

\par Now consider the operator $P_z=P_{0,z}$ in \no{hs.5} 
for $z\in \Omega $, and keep the assumption (\ref{fu.30}).

 Define 
$$
{\cal P}_\delta =\left(\begin{array}{ccc}P_{\delta ,z} &R_{-,\delta
    }\\ R_{+,\delta } &0
 \end{array}\right)
$$
as in \no{gr.1}--\no{gr.3}, so that ${\cal P}={\cal P}_0$. As in (5.10)
in \cite{HaSj08} we have 
\ekv{grny.6}
{
|\det {\cal P}_\delta |^2=\alpha ^{-N}\det 1_\alpha (S_{\delta,z} ),\quad 2\ln |\det
{\cal P}_\delta |=\ln \det 1_\alpha (S_{\delta,z} )+N\ln
\frac{1}{\alpha }, 
}
where $1_\alpha (t)=\max (\alpha ,t)$, $t\ge 0$, 
which with (\ref{fu.43}) and the bound $N={\cal O}(\alpha ^\kappa h^{-n})$ gives
\ekv{grny.9}{
\ln |\det {\cal P}_\delta |=\frac{1}{(2\pi h)^n}(\iint \ln |p_z| dxd\xi 
+{\cal O}(\alpha ^\kappa \ln \frac{1}{\alpha }+\frac{\delta }{\alpha
}\Vert Q\Vert )). 
}

\subsection{Singular values and determinants for certain
matrices associated to $\delta $ potentials}\label{inv} 

\begin{lemma}
\label{inv4} (\cite{Sj08a, Sj08b}) Let $e_1,...,e_N\in C^0(X)$ and put 
$$
\overrightarrow{e}(x)=\left(\begin{array}{ccc}e_1(x)\\e_2(x)\\ ..\\e_N(x) \end{array}\right),\ x\in X.
$$
Let $L\subset {\bf C}^N$ be a linear subspace of dimension $M-1$,
for some $1\le M\le N$. Then there exists $x\in X $ such that 
\ekv{inv.3} { \mathrm{dist\,}(\overrightarrow{e}(x), L)^2 \ge
\frac{1}{\mathrm{vol\,}(X )}\mathrm{tr\,}((1-\pi _L){\cal
  E}_X ) ,}
where ${\cal E}_X =((e_j|e_k)_{L^2(X )})_{1\le j,k\le N}$
and $\pi _L$ is the orthogonal projection from ${\bf C}^N$ onto $L$.
\end{lemma}
\begin{proof} 
Let $\nu _1,...,\nu _N$ be an orthonormal basis in ${\bf C}^N$ such that
$L$ is spanned by $\nu _1,...,\nu _{M-1}$ (and equal to $0$ when
$M=1$). Then by direct calculations,
$$
\int_X \mathrm{dist\,}(\overrightarrow{e}(x),L)^2dx=\sum_{\ell =M}^N ({\cal
  E}_X \nu _\ell|\nu _\ell )= \mathrm{tr\,}((1-\pi _L){\cal
  E}_X ). 
$$
It then suffices to estimate the integral from above by
$$\mathrm{vol\,}(X ) \sup_{x\in X }\mathrm{dist\,}
(\overrightarrow{e}(x),L)^2,$$ 
and we can
find an $x\in X $ satisfying (\ref{inv.3}).
\end{proof}

\par If we make the assumption that
 \ekv{inv.1} { e_1,...,e_N \mbox{ is an orthonormal family in }
L^2(X ), }
then ${\cal E}_X =1$ and \no{inv.3} simplifies to 
\ekv{inv.3.5}
{
\max_{x\in X }\mathrm{dist\,}(\overrightarrow{e}(x),L)^2\ge
\frac{N-M+1}{\mathrm{vol\,}(X )}.
}

In the general case, let $0\le\varepsilon _1\le \varepsilon _2\le ...\le
\varepsilon _N$ denote the eigenvalues of ${\cal E}_X $. Then, using the mini-max principle, one can show that 
\ekv{inv.1.6}
{
\inf_{\mathrm{dim\,}L=M-1}\mathrm{tr\,}((1-\pi _L){\cal E}_X )=
\varepsilon _1+\varepsilon _2+...+\varepsilon _{N-M+1}=:E_M.
}

Now, we can use the lemma to choose successively $a_1,...,a_N\in X
$ such that
\begin{eqnarray*} \| \overrightarrow{e}(a_1)\Vert^2&\ge& {E_1\over {\rm vol\,}(X
)},\\ {\rm dist\,}(\overrightarrow{e}(a_2),{\bf C} \overrightarrow{e}(a_1))^2&\ge& {E_2\over {\rm
vol\,}(X )},\\ &...&\\ {\rm dist\,}(\overrightarrow{e}(a_M),{\bf C} \overrightarrow{e}(a_1)\oplus
...\oplus{\bf C} \overrightarrow{e}(a_{M-1}))^2 &\ge& {E_M\over {\rm vol\,}(X )},\\
&...&
\end{eqnarray*}
\par Let $\nu _1,\nu _2,...,\nu _N$ be the Gram-Schmidt
orthonormalization of the basis $\overrightarrow{e}(a_1), \overrightarrow{e}(a_2),..., \overrightarrow{e}(a_N)$, so that
\ekv{inv.4} {\overrightarrow{e}(a_M)\equiv c_M \nu _M {\rm mod\,}(\nu
_1,...,\nu_{M-1}), \mbox{ where } | c_M | \ge \left(\frac{E_M}{{\rm
vol\,}(X )}\right)^{1\over 2}.}

\par Consider the $N\times N$ matrix $E=(\overrightarrow{e}(a_1)\, \overrightarrow{e}(a_2)\, ...\,
\overrightarrow{e}(a_N))$ where $\overrightarrow{e}(a_j)$ are viewed
as columns. Expressing these vectors
in the basis $\nu _1,...,\nu _N$ will not change the absolute value of
the determinant and $E$ now becomes an upper triangular matrix with
diagonal entries $c_1,...,c_N$. Hence \ekv{inv.6} {| \det E| = |
c_1\cdot ...\cdot c_N| , } and \no{inv.4} implies that 
\ekv{inv.7} {
|\det E| \ge {(E_1E_2...E_N)^{1/2}\over ({\rm vol\,}(X ))^{N/2}}.  }

Let $f_1,f_2,...,f_N$ be a second family of continuous functions on $X$.
Define $M={\bf C}^N\to {\bf C}^N$ by \ekv{inv.0} {
Mu=\sum_1^N (u|\overrightarrow{f}(a_\nu ))\overrightarrow{e}(a_\nu ),\
u\in{\bf C}^N.  }
Then
\ekv{inv.8} {
M=E\circ F^*, } where 
\ekv{inv.9} { F=(\overrightarrow{f}(a_1)...\overrightarrow{f}(a_N)).  }
Now, we assume
\ekv{inv.2} { f_j=\overline{e}_j,\ \forall j.  }
Then $F^*=\trans{E}$, so
\ekv{inv.10} { M=E\circ \trans{E}.  }  
We get from \no{inv.7}, \no{inv.10}, that 
\ekv{inv.10.5}
{
|
\det M | \ge {E_1E_2...E_N\over {\rm vol\,}(X )^N}.
}
Under the assumption \no{inv.1}, this simplifies to
\ekv{inv.11} { |
\det M | \ge {N!\over {\rm vol\,}(X )^N}.  }

\par Using that 
\ekv{inv.11.1}
{
|\det M|=\prod_1^N s_j\le s_1^{k-1}s_k^{N)k+1}\le s_1^N,\hbox{ where } s_j=s_j(M),
}
we get
\begin{prop}\label{inv5} Under the above assumptions,
\ekv{inv.11.2.5}{
s_1\ge \frac{(E_1...E_N)^{\frac{1}{N}}}{\mathrm{vol\,}(X )},
}
\ekv{inv.11.3}{s_k\ge s_1\left(\prod_1^N\left(\frac{E_j}{s_1\mathrm{vol\,}(X )}\right)
\right)^{\frac{1}{N-k+1}}.}
\end{prop}

\subsection{Singular values of matrices associated to suitable 
admissible potentials}\label{cl}
We let $P,\widetilde{P},p,\widetilde{p}$ be as in the
introduction to this section. 
We also choose $\epsilon _k$, $\mu _k$, $D=D(h)$,
$L=L(h)$ as in and around (\ref{int.6.3}), (\ref{int.6.4}).
\begin{dref}\label{spe01}
An admissible potential is a potential of the form
\ekv{cl.1}
{
q(x)=\sum_{0<\mu _k\le L}\alpha _k\epsilon _k(x),\ \alpha
\in {\bf C}^D.
}
\end{dref}
We shall approximate $\delta $-potentials in $H_h^{-s}$ with admissible
ones and then apply the results of the preceding subsection. As in the introduction we let $s>n/2$,
$0<\epsilon <s-n/2$.
\begin{prop}\label{spe02}
Let $a\in X$. Then $\exists \alpha
\in {\bf C}^D$, $r\in H_h^{-s}$ such that 
\ekv{cl.7}
{
\delta _a(x)=\sum_{\mu _k\le L}\alpha _k\epsilon_k+r(x),
}
where
\ekv{cl.9}
{
\Vert r\Vert_{H_h^{-s}}\le C_{s,\epsilon }
L^{-(s-\frac{n}{2}-\epsilon )}h^{-\frac{n}{2}},
}
\ekv{cl.12}
{
(\sum \vert \alpha _k\vert^2)^{\frac{1}{2}}
\le \langle L\rangle^{\frac{n}{2}+\epsilon} (\sum_{\mu _k\le L}\langle
\mu _k\rangle^{-2(\frac{n}{2}+\epsilon ) }
|\alpha _k|^2)^{\frac{1}{2}}
\le CL^{\frac{n}{2}+\epsilon }  h^{-\frac{n}{2}}.
}
\end{prop}

The proof uses that $\delta _a\in H^{-(\frac{n}{2}+\epsilon )}$ with norm ${\cal O}(h^{-n/2})$ and the fact that for any $\widetilde{s}\in {\bf R}$, 
$$\|\sum_1^\infty \alpha _k\epsilon _k\|_{H_h^{\widetilde{s}}}^2\asymp \sum \langle \mu _k\rangle^{2\widetilde{s}}|\alpha _k|^2.$$ It then suffices to truncate the expansion of $\delta _a$ in  the basis $\epsilon _1,\epsilon _2,...$.

\par Let $P_\delta $ be as in \no{hs.3} and assume \no{hs.4}, (\ref{hs.6}).  Let 
${\cal R}(\pi _\alpha )={\bf C}e_1\oplus ...\oplus {\bf C}e_N$ be as
in one of the two cases of subsection \ref{hs}. By the mini-max
principle and standard spectral asymptotics (see \cite{DiSj99}), we know
that $N={\cal O}(h^{-n})$ and if we want to use the assumption
(\ref{int.6.2}) we even have $N={\cal O}((\max (\alpha ,h))^\kappa
h^{-n})$ by(\ref{grny.5.6}). For the moment we shall only use
that $N$ is bounded by a negative power of $h$. Let 
$a=(a_1,...,a_N)\in X^N$ and put 
\ekv{spe.01}
{
q_a(x)=\sum_1^N \delta (x-a_j),
}
\ekv{spe.02}
{
M_{q_a;j,k}=\int q_a(x)e_k(x)e_j(x)dx,\ 1\le j,k\le N.
}
Now \no{hs.8} implies that $\Vert \sum \lambda _ke_k\Vert_{H^s_h}\le {\cal O}(\Vert \lambda \Vert_{\ell^2})$ so \no{al.2} and the fact that $\Vert
q_a\Vert_{H_h^{-s}}={\cal O}(1)Nh^{-n/2}$, imply that for all $\lambda ,\mu \in
{\bf C}^n$,
\begin{eqnarray*}
\langle M_{q_a}\lambda ,\mu \rangle&=&\int q_a(x)(\sum \lambda
_ke_k)(\sum \mu _je_j)dx\\
&=& {\cal O}(1)Nh^{-n}\Vert \lambda \Vert \Vert \mu \Vert
\end{eqnarray*}
and hence
\ekv{spe.03}{s_1(M_{q_a})=\Vert M_{q_a}\Vert_{{\cal L}({\bf C}^N,{\bf
      C}^N)}={\cal O}(1)Nh^{-n}.}

\par We now choose $a$ so that (\ref{inv.11.2.5}), \no{inv.11.3} hold. 
\par The $e_j$ form and orthonormal system, so ${\cal E}_X=1$ and 
\ekv{spe.04}
{
E_j=N-j+1.
}
Then, choosing the $a_j$ as in Subsection \ref{inv}, \no{inv.11.2.5} gives the lower bound
\ekv{spe.05}
{
s_1\ge \frac{(N!)^{\frac{1}{N}}}{\mathrm{vol\,}(X )}
= 
(1+{\cal O}(\frac{\ln N}{N}))\frac{N}{e\,\mathrm{vol\,}(X )}
,
}
where the last identity follows from Stirling's formula.

\par Rewriting \no{inv.11.3} as
$$
s_k\ge s_1^{-\frac{k-1}{N-k+1}}\left( \prod_1^N
  \frac{E_j}{\mathrm{vol\,}(X )}\right)^{\frac{1}{N-k+1}},
$$
and using \no{spe.03}, we get
\ekv{spe.06}
{
s_k\ge \frac{1}{C^{\frac{k-1}{N-k+1}}(\mathrm{vol\,}(X ))^{\frac{N}{N-k+1}}}
\left( \frac{h^n}{N} \right)^{\frac{k-1}{N-k+1}}(N!)^{\frac{1}{N-k+1}}.
}

\par Summing up, we get
\begin{prop}\label{spe1} 
We can find $a_1,...,a_N\in X $ such that if 
$q_a=\sum_1^N \delta (x-a_j)$ and $M_{q_a;j,k}=\int q_a(x)
e_k(x)e_j(x)dx$, then the singular values $s_1\ge s_2\ge ...\ge s_N$ of $M_{q_a}$,
satisfy \no{spe.03}, \no{spe.05} and \no{spe.06}.
\end{prop}

\par We next approximate $q_a$ with an admissible potential by applying
Proposition \ref{spe02} to each $\delta $-function in $q_a$:
\ekv{spe.4}
{
q_a=q+r,\ q=\sum_{\mu _k\le L}\alpha _k\epsilon _k,
}
where 
\ekv{spe.4.5}
{
\Vert q\Vert_{H_h^{-s}}\le Ch^{-\frac{n}{2}}N,
}
\ekv{spe.5}
{
\Vert r\Vert_{H_h^{-s}}\le C_\epsilon L^{-(s-\frac{n}{2}-\epsilon )}h^{-\frac{n}{2}}N,
}
\ekv{spe.6}
{
(\sum \vert \alpha _k\vert^2)^{\frac{1}{2}}\le
CL^{\frac{n}{2}+\epsilon }h^{-\frac{n}{2}}N.
}
Below, we shall have $N={\cal O}(h^{\kappa -n})$ so if we choose $L$
as in (\ref{int.6.4}), we get 
$$
|\alpha |_{{\bf C}^D}\le C h^{-(\frac{n}{2}+\epsilon )M+\kappa -\frac{3n}{2}}
$$
and $q$ satisfies
(\ref{int.6.3}), (\ref{int.6.4}).
We get
\ekv{spe.8}
{
\Vert M_r\Vert\le C_\epsilon L^{-(s-\frac{n}{2}-\epsilon )}h^{-n}N.
}

\par For the admissible potential $q$ in \no{spe.4}, we thus obtain
from \no{spe.06}, \no{spe.8}:
\ekv{spe.9}
{
s_k (M_q)\ge  \frac{1}{C^{\frac{k-1}{N-k+1}}(\mathrm{vol\,}(X ))^{\frac{N}{N-k+1}}}
\left( \frac{h^n}{N} \right)^{\frac{k-1}{N-k+1}}(N!)^{\frac{1}{N-k+1}}-C_\epsilon L^{-(s-\frac{n}{2}-\epsilon )}h^{-n}N.
}

\par Similarly, from (\ref{spe.03}), (\ref{spe.8}) we get for $L\ge 1$:
\ekv{spe.10}
{
\Vert M_q\Vert \le CNh^{-n}.
}

\par Using Proposition \ref{al2}, we get
for every $\epsilon >0$,
\ekv{spe.11}
{
\Vert q\Vert_{H_h^s}\le {\cal O}(1)NL^{s+\frac{n}{2}+\epsilon }
h^{-\frac{n}{2}},\ \forall \epsilon >0.
}
Summing up, we have obtained
\begin{prop}\label{spe2}
Fix $s>n/2$ and $P_\delta $ as in \no{hs.3}, \no{hs.4}, (\ref{hs.6}) and let $\pi
_\alpha $, $e_1,...,e_N$  be as in one of the two cases in Subsection \ref{hs}. Choose the $h$-dependent parameter
$L$ with $1\ll L\le {\cal O}(h^{-N_0})$ for some fixed $N_0>0$. Then
we can find an admissible potential $q$ as in \no{spe.4} (different
from the one in \no{hs.3}, \no{hs.4}) such that the matrix $M_q$, defined by
$$
M_{q;j,k}=\int q e_ke_j dx,
$$
satisfies \no{spe.9}, \no{spe.10}. Moreover the $H_h^s$-norm of $q$ 
satisfies \no{spe.11}.
\end{prop}

Notice also that if we choose $\widetilde{R}$ with real coefficients,
then we can choose $q$ real-valued.

\subsection{Lower bounds on the small singular values for 
suitable perturbations}\label{sv}

\par In this subsection, we fix a $z\in \Omega $. We will use Proposition
\ref{spe2} iteratively to construct a special admissible perturbation
$P_\delta $ for which we have nice lower bounds on the small singular
values of $P_\delta -z$, that will lead to similar bounds for the ones
of $P_{\delta ,z}$ and to a lower bound on $|\det P_{\delta ,z}|$. 

\par We will need the symmetry assumption (\ref{int'.7}):
\ekv{sv.2}
{
P^*=\Gamma \circ P\circ\Gamma ,
}
This property remains unchanged if we add a multiplication
operator to $P$.

\par As in the introduction, we let
\ekv{sv.3}
{
V_z(t)=\mathrm{vol\,}(\{ \rho \in T^*X;\, |p(\rho )-z|^2\le
t\}),
}
and assume (for our fixed value of $z$) that (\ref{int.6.2}) holds:
\ekv{sv.4}
{
V_z(t)={\cal O}(t^\kappa
),\ 0\le t\ll 1,
}
for some $\kappa \in ]0,1]$.
Proposition \ref{fu6} gives:
\begin{prop}\label{sv1}
Assume \no{sv.4} and recall Remark
\ref{gr0}. For $0<h\ll \alpha \ll 1$, the number $N(\alpha )$ of
eigenvalues of $(P-z)^*(P-z)$ in $[0,\alpha ]$ satisfies
\ekv{sv.5}
{
N(\alpha )={\cal O}(\alpha ^\kappa h^{-n}).
}
\end{prop}

\par Let $\epsilon >0$, $s>\frac{n}{2}+\epsilon $ be fixed as in the
introduction and consider
\ekv{sv.5.5}
{
P_0=P+\delta _0(h^{\frac{n}{2}}q_1^0+q_2^0),\hbox{ with }0\le \delta _0\ll h,\ \Vert
q_1^0\Vert _{H_h^s}, \Vert
q_2^0\Vert _{H^s}\le 1.}
 From the mini-max principle, we see that Proposition
\ref{sv1} still applies after replacing $P$ by $P_0$.

\par Choose $\tau_0\in
]0,(Ch)^\frac{1}{2}]$ and let $N={\cal O}(h^{\kappa -n})$ be the number
of singular values of $P_0-z$; $0\le t_1(P_0-z)\le ...\le t_N(P_0-z)<
\tau_0$ in the interval $[0,\tau_0[$. 
As in the introduction we put 
\ekv{sv.14b}
{
N_1=\widetilde{M}+sM+\frac{n}{2},
}
where $M,\widetilde{M}$ are the parameters 
in (\ref{int.6.4}).
Fix $\theta \in ]0,\frac{1}{4}[$ and recall that $N$ is determined by
the property $t_N(P_0-z)<\tau_0\le t_{N+1}(P_0-z)$. Fix $\epsilon
_0>0$.
\begin{prop}\label{sp1} a) If $q$ is an admissible potential as in
  (\ref{int.6.3}), (\ref{int.6.4}), we
  have 
\ekv{sv.14a}
{
\Vert q\Vert_\infty \le Ch^{-n/2}\Vert q\Vert_{H_h^s}\le \widetilde{C}h^{-N_1}.
}
\par\noindent b) If $N$ is sufficiently large, 
there exists such an admissible potential $q$, such
that if 
$$P_\delta =P_0+\frac{\delta h^{N_1}}{\widetilde{C}}q=:P_0+\delta Q,\
\delta =\frac{\tau_0}{C}h^{N_1+n}$$
(so that $\Vert Q\Vert \le 1$) then 
\ekv{sp.15}
{
t_\nu (P_\delta -z )\ge t_\nu (P_0-z)-\frac{\tau_0h^{N_1+n}}{C}\ge
(1-\frac{h^{N_1+n}}{C})t_\nu (P_0-z),\ \nu >N ,
}
\ekv{sp.16}
{
t_\nu (P_{\delta }-z)\ge \tau_0 h^{N_2},\ [N-\theta N] +1 \le \nu \le N.
}
Here, we put 
\ekv{sp.11}
{
N_2=2(N_1+n)+\epsilon _0,
}
and we let $[a]=\max ({\bf Z}\cap ]-\infty ,a])$ denote the integer
part of the real number $a$.
When $N={\cal O}(1)$, we have the same result provided that we replace
(\ref{sp.16}) by 
\ekv{sp.16a}
{
t_N (P_{\delta })\ge \tau_0 h^{N_2}.}
\end{prop}
\begin{proof}
The part a) follows from Subsection \ref{al}, the definition of
admissible potentials in the introduction and from the definition of
$N_1$ in (\ref{sv.14b}). (See also (\ref{spe.11}).) We shall therefore concentrate on the proof
of b).

Let $e_1,...,e_N$ be an
orthonormal family of eigenfunctions corresponding to $t_\nu (P_0-z)$, 
so that 
\ekv{sv.6}
{
(P_0-z)^*(P_0-z)e_j=(t_j(P_0-z))^2e_j.
}
Using \no{int.6} $\Leftrightarrow$ \no{sv.2}, we see that a
corresponding family of eigenfunctions of $(P-z)(P-z)^*$ is given by 
\ekv{sv.7}
{
\widetilde{f}_j=\Gamma e_j.
}
$\widetilde{f}_1,...,\widetilde{f}_N$ and $f_1,...,f_N$ are 
orthonormal families  that span the same space $F_N$. Let $E_N$ be the 
span of $e_1,...,e_N$. We then know that 
\ekv{sv.8}
{
(P_0-z):E_N\to F_N \hbox{ and }(P_0-z)^*:F_N\to E_N
}
have the same singular values $0\le t_1\le t_2\le ...\le t_N$. 

\par Define $R_+:L^2\to {\bf C}^N$, $R_-:{\bf C}^N\to L^2$ by
\ekv{sv.9}
{
R_+u(j)=(u|e_j),\quad R_-u_-=\sum_1^N u_-(j)\widetilde{f}_j.
}
Then 
\ekv{sv.9.5}
{
{\cal P}=\left(\begin{array}{ccc}P_0-z &R_-\\ R_+ &0 \end{array}\right)
:{\cal D}(P_0)\times {\bf C}^N \to L^2\times {\bf C}^N
}
has a bounded inverse 
$$
{\cal E}=\left(\begin{array}{ccc}E &E_+\\ E_-
    &E_{-+} \end{array}\right) .
$$
The singular values of $E_{-+}$
are given by
$
t_j(E_{-+})=t_j(P_0-z),\ 1\le j\le N,
$ 
or equivalently by $s_j(E_{-+})=t_{N+1-j}(P_0-z)$, for $1\le j\le N$.
   
\par We will apply Subsection \ref{gr}, and recall that $N$ is assumed to
be sufficiently large and that $\theta $ has been fixed in $]0,1/4[$.
(The case of bounded $N$ will be treated later.) 
Let $N_2$ be given in (\ref{sp.11}). Since $z$ is fixed it will also be
notationally convenient to assume that $z=0$.

\medskip\par\noindent 
\emph{Case 1.} $s_j(E_{-+})\ge \tau_0h^{N_2}$, 
for $1\le j \le N-[(1-\theta )N]$. Then we get the desired conclusion with
$q =0$, $P_\delta =P_0$.

\medskip
\par\noindent \emph{Case 2.}
\ekv{sp.2}
{
s_j(E_{-+})<\tau_0h^{N_2} \hbox{ for some }j\hbox{ such that }1\le
j \le N-[(1-\theta )N].
}

Recall that for the special 
admissible potential $q$ in \no{spe.4}, we have
\no{spe.9}. For $k\le N/2$, we have $N-k+1>N/2$, so 
$$
\frac{k-1}{N-k+1}\le 1,
$$
and \no{spe.9} gives
$$
s_k(M_q)\ge 
\frac{h^n}{CN}(N!)^\frac{1}{N}-C_\epsilon
L^{-(s-\frac{n}{2}-\epsilon )}\frac{N}{h^n}.
$$
By Stirling's formula, we have $(N!)^\frac{1}{N}\ge N/\mathrm{Const}$,
so for $1\le k\le N/2$, we obtain with a new constant $C>0$:
$$
s_k(M_q)\ge \frac{h^n}{C}-C_\epsilon
L^{-(s-\frac{n}{2}-\epsilon )}\frac{N}{h^n}.
$$
Here, we recall from Proposition \ref{sv1} (which also applies to
$P_0$) that $N={\cal O}(h^{\kappa
  -n})$
and choose $L$ so
that 
$$
L^{-(s-\frac{n}{2}-\epsilon )}h^{\kappa -2n}\ll h^n,
$$
i.e. so that (in agreement with \no{int.6.4}) 
\ekv{sv.10}
{
L\gg h^\frac{\kappa -3n}{s-\frac{n}{2}-\epsilon }.
}
We then get 
\ekv{sv.11}
{
s_k(M_q)\ge \frac{h^n}{C},\ 1\le k\le \frac{N}{2},
}
for a new constant $C>0$.

\par From (\ref{spe.10}) and the fact that $N={\cal O}(h^{\kappa -n})$
we get
\ekv{sv.12}
{
s_1(M_q)\le \Vert M_q\Vert\le CNh^{-n}\le \widetilde{C}h^{\kappa -2n}.
}

\par In addition to the lower bound \no{sv.10} we assume as in 
\no{int.6.4} (in all
cases) that 
\ekv{sv.13}
{
L\le C h^{-M},\mbox{ for some } M\ge \frac{3n-\kappa }{s-\frac{n}{2}-\epsilon }.
}
As we saw after (\ref{spe.6}), $q$ is indeed an admissible potential
as in (\ref{int.6.3}), (\ref{int.6.4}),
so that by (\ref{sv.14a})
\ekv{sv.14}
{
\Vert q\Vert_\infty \le Ch^{-\frac{n}{2}}\Vert q\Vert_{H_h^s}
\le \widetilde{C}h^{-N_1}.}

Put 
\ekv{sp.3}{P_\delta =P_0+\frac{\delta h^{N_1}}{\widetilde{C}}q=P_0+\delta Q,\
Q=\frac{h^{N_1}}{\widetilde{C}}q,\ \Vert Q\Vert \le 1.}
Then, if $\delta \le \tau_0/2$, we can replace $P_0$ by $P_\delta $
in \no{sv.9.5} and we still have a well-posed problem as in Subsection \ref{gr} 
with $Q_\omega =Q$ as
above. Here $E_-^0QE_+^0=h^{N_1}M_q/\widetilde{C}$ 
so according to \no{sv.11},
we have with a new constant $C$
\ekv{sp.4}
{
s_k(\delta E_-^0QE_+^0)\ge \frac{\delta h^{N_1+n}}{C},\ 1\le
k\le \frac{N}{2}.
}

\par Playing with the general estimate \no{s.9}, we get
$$s_\nu (A+B)\ge s_{\nu +k-1}(A)-s_k(B)$$ and for a sum of three operators
$$s_\nu (A+B+C)\ge s_{\nu +k+\ell -2}(A)-s_k(B)-s_\ell (C).$$ We apply this to 
$E_{-+}^\delta $ in \no{s.15} and get 
\ekv{sp.6}
{
s_\nu (E_{-+}^\delta )\ge s_{\nu +k-1}(\delta
E_-^0QE_+^0)-s_k(E_{-+}^0)-2\frac{\delta ^2}{\tau_0}.
} 
Here we use \no{sp.2} with $j=k=N-[(1-\theta )N]$ as well as 
\no{sp.4}, to
get for $\nu  \le N-[(1-\theta )N]$ 
\ekv{sp.7}
{
s_\nu (E_{-+}^\delta )\ge 
\frac{\delta h^{N_1+n}}{C}
-\tau
_0h^{N_2}-2\frac{\delta ^2}{\tau_0}.
}
Recall that $\theta <\frac{1}{4}$.

\par Choose 
\ekv{sp.9}
{
\delta =\frac{1}{C}\tau_0h^{N_1+n},
}
where (the new constant) $C>0$ is sufficiently large.

\par 
Then, with a new constant $C>0$,
we get (for $h>0$ small enough)
\ekv{sp.10}
{
s_\nu (E_{-+}^\delta )\ge \frac{\delta }{C}h^{N_1+n},\ 1\le \nu \le N-
[(1- \theta) N],
}
implying
\ekv{sp.12}
{
s_\nu (E_{-+}^\delta )\ge 8\tau_0h^{N_2},\ 1\le \nu  \le N-
[(1- \theta) N].} For the corresponding
operator $P_\delta $, we have for $\nu >N$:
$$
t_\nu (P_\delta )\ge t_\nu (P_0)-\delta =t_\nu (P_0)-\frac{\tau _0h^{N_1+n}}{C}.
$$ Since $t_\nu (P)\ge \tau_0$ in this case, we get (\ref{sp.15}).

\par From \no{sp.12} and \no{s.16}, we
get (\ref{sp.16}).

\par When $N={\cal O}(1)$, we still get (\ref{sp.7}) with $\nu =1$ and
this leads to (\ref{sp.16a}).

\end{proof}

\par The construction can now be iterated. Assume that $N\gg 1$ and 
replace $(P_0,N,\tau_0)$
by $(P_\delta ,[(1-\theta )N], \tau_0h^{N_2})=:
(P^{(1)},N^{(1)},\tau_0^{(1)})$ and keep on, using the same values
for the exponents $N_1,N_2$. Then we get a sequence $(P^{(k)},N^{(k)},\tau
_0^{(k)})$, $k=0,1,...,k(N)$, where the last value $k(N)$ is
determined by the fact that $N^{(k(N))}$ is of the order of magnitude of
a large constant. Moreover,
\ekv{sp.17}
{
t_\nu (P^{(k)})\ge \tau_0^{(k)},\ N^{(k)}<\nu \le N^{(k-1)},
} 
\ekv{sp.18}
{
t_\nu (P^{(k+1)})\ge t_\nu (P^{(k)})-
\frac{\tau_0^{(k)}h^{N_1+\nu }}{C},\ \nu >N^{(k)},
}
\ekv{sp.18.1}
{
\tau_0^{(k+1)}=\tau_0^{(k)}h^{N_2},
}
\ekv{sp.18.2}
{
N^{(k+1)}=[(1-\theta )N^{(k)}],
}
$$
P^{(0)}=P,\ N^{(0)}=N,\ \tau_0^{(0)}=\tau_0.
$$
Here,
\begin{eqnarray*}
&P^{(k+1)}=P^{(k)}+\delta ^{(k+1)}Q^{(k+1)}=P^{(k)}+\frac{\delta
  ^{(k+1)}h^{N_1}}{\widetilde{C}}q^{(k+1)},&\\ &\Vert Q^{(k+1)}\Vert\le 1,\
\delta ^{(k+1)}=\frac{1}{C}\tau_0^{(k)}h^{N_1+n}.&
\end{eqnarray*}
Notice that $N^{(k)}$ decays exponentially fast with $k$:
\ekv{sp.18.5}
{
N^{(k)}\le (1-\theta )^kN,
}
so we get the condition on $k$ that $(1-\theta )^kN\ge C\gg 1$ which
gives,
\ekv{sp.19}{k\le \frac{\ln \frac{N}{C}}{\ln \frac{1}{1-\theta }}.}
We also have 
\ekv{sp.20}
{
\tau_0^{(k)}=\tau_0\left( h^{N_2} \right)^k .
}

\par For $\nu >N$, we iterate \no{sp.18}, to get
\eekv{sp.23}
{t_\nu (P^{(k)})&\ge& t_\nu (P)- \tau_0 \frac{h^{N_1+n}}{C}\left( 1+
h^{N_2}+h^{2N_2}+...\right)}
{&\ge& t_\nu (P)-\tau_0 {\cal O}(\frac{h^{N_1+n}}{C}).}

\par For $1\ll \nu \le N$, let $\ell=\ell (N)$ be the unique value for
which $N^{(\ell )}<\nu \le N^{(\ell -1)}$, so that 
\ekv{sp.24}
{
t_\nu (P^{(\ell )})\ge \tau_0^{(\ell )},
}
by \no{sp.17}. If $k>\ell $, we get 
\ekv{sp.24.5}
{
t_\nu (P^{(k)})\ge t_\nu (P^{(\ell )})-
\tau_0^{(\ell )}{\cal O}(\frac{h^{N_1+n}}{C}) .
}

\par The iteration above works until we reach a value $k=k_0={\cal O}
(\frac{\ln \frac{N}{C}}{\ln \frac{1}{1-\theta }})$ for which
$N^{(k_0)}={\cal O}(1)$. After that, we continue the iteration further 
by decreasing
$N^{(k)}$ by one unit at each step. 

\medskip

\par Summing up the discussion so far, we have obtained
\begin{prop}\label{sp2}
Let $(P,z)$ satisfy the assumptions in the beginning of this
subsection and choose $P_0$ as in (\ref{sv.5.5}).
 Let
$s>\frac{n}{2}$, $0<\epsilon <s-\frac{n}{2}$, $M\ge \frac{3n-\kappa
}{s-\frac{n}{2}-\epsilon }$, $N_1=\widetilde{M}+sM+\frac{n}{2}$, 
$N_2=2(N_1+n)+\epsilon _0$, where $\epsilon _0>0$. Let $L$ be an
$h$-dependent parameter satisfying
\ekv{sp.24.8}
{
h^{\frac{\kappa -3n}{s-\frac{n}{2}-\epsilon }}\ll L\le C h^{-M}. 
}
Let $0<\tau_0\le \sqrt{h}$ and let $N^{(0)}={\cal O}(h^{\kappa -n})$ 
be the number of singular
values of $P_0-z$ in $[0,\tau_0[$. Let $0<\theta <\frac{1}{4}$ and let
$N(\theta )\gg 1$ be sufficiently large. Define $N^{(k)}$, $1\le k\le
k_1$ iteratively in the following way. As long as $N^{(k)}\ge
N(\theta )$, we put $N^{(k+1)}=[(1-\theta )N^{(k)}]$. Let $k_0\ge 0$
be the last $k$ value we get in this way. For $k>k_0$ put 
$N^{(k+1)}=N^{(k)}-1$, until
we reach the value $k_1$ for which $N^{(k_1)}=1$.

\par Put $\tau_0^{(k)}=\tau_0h^{kN_2}$, $1\le k\le k_1+1$. Then there 
exists an admissible potential $q=q_h(x)$ as in (\ref{int.6.3}), (\ref{int.6.4}), satisfying \no{spe.6}, \no{spe.11}, so that,
$$
\Vert q\Vert_{H_h^s}\le {\cal O}(1)h^{-N_1+\frac{n}{2}}, \
\Vert q\Vert_{L^\infty }\le {\cal O}(1)h^{-N_1}, 
$$
such that if $P_{\delta}=P_0+\frac{1}{C}\tau
_0h^{2N_1+n}q=P_0+\delta Q$, $\delta =\frac{1}{C}h^{N_1+n}\tau_0$, $Q=h^{N_1}q$,
we have the following estimates on the singular values of 
$P_{\delta }-z$:
\begin{itemize}
\item If $\nu >N^{(0)}$, we have 
$t_\nu (P_{\delta }-z)\ge (1-\frac{h^{N_1+n}}{C})t_\nu (P_0-z)$.
\item If $N^{(k)}<\nu \le N^{(k-1)},$ $1\le k\le k_1$, then $
t_\nu (P_\delta -z)\ge (1-{\cal O}(h^{N_1+n}))\tau_0^{(k)}$.

\item Finally, for $\nu =N^{(k_1)}=1$, we have  $
t_1(P_\delta -z)\ge (1-{\cal O}(h^{N_1+n}))\tau_0^{(k_1+1)}$.
\end{itemize}
\end{prop}

Now it is possible to pass from the Grushin problem for $P_\delta -z$ to a suitable one for $P_{\delta ,z}$ and follow up with estimates on the singular values (cf (\ref{s.12})) and obtain:
\begin{prop}\label{sp3}
Proposition \ref{sp2} remains valid if we replace $P_\delta -z$ there
with $P_{\delta ,z}$.
\end{prop}

Taking a suitable Grushin problem for $P_{\delta ,z}$ and using Proposition \ref{gr1} we can show
when $\tau_0 =\sqrt{h}$: 
\begin{prop}\label{sp4}
For the special admissible perturbation $P_\delta $ in the
propositions \ref{sp2}, \ref{sp3}, we have 
\eekv{sp.30}
{
&&\ln |\det P_{\delta ,z}|\ge }
{&&\hskip -6mm\frac{1}{(2\pi h)^n}
\left( \iint \ln |p_z|dxd\xi -{\cal O}\left(h^{N_1+n-\frac{1}{2}}
+(h^{\kappa }+h^n\ln \frac{1}{h})(\ln \frac{1}{\tau_0}+
(\ln \frac{1}{h})^2)\right)
\right) .
}
\end{prop}

\par We also have the upper bound 
$$
|\det E_{-+}| \le \Vert E_{-+}\Vert ^{N^{(0)}}\le 
\exp (CN^{(0)}),
$$
which leads to
\ekv{sp.31}
{
\ln |\det P_{\delta ,z}|\le \frac{1}{(2\pi h)^n}
\left( \iint \ln |p_z|dxd\xi +{\cal O}\left( h^{N_1+n-\frac{1}{2}}
+h^{\kappa }\ln \frac{1}{h}\right)
\right) .
}
Notice that this bound is more general, it only depends on the
fact that the perturbation of $P$ is of the form $\delta Q$ with $\delta
=\tau_0h^{N_1+n}/C$ and with $\Vert Q\Vert ={\cal O}(1)$.

\par When $\tau_0\le \sqrt{h}$
 we keep the same Grushin problem as before and notice that the
 singular values of $E_{-+}$ that are $\le \tau_0$, obey the estimates
 in Proposition \ref{sp2}. Their contribution to $\ln |\det E_{-+}|$
 can still be estimated from below. The contribution
 from the singular values of $E_{-+}$ that are $>\tau_0$ can be estimated from below by $-{\cal O}(h^{\kappa -n}\ln
 (1/\tau_0 ))$ and this leads to the conclusion that {\it Proposition \ref{sp4} remains valid when $0<\tau_0\le
 \sqrt{h}$. The same holds for the upper bound \no{sp.31}.}

\subsection{Estimating the probability that $\det E_{-+}^\delta $ is
small}
\label{pr} 
In this subsection we keep the assumptions on $(P,z)$ of the beginning of
Subsection \ref{sv} and choose $P_0$ as in (\ref{sv.5.5}). We consider
general $P_\delta $ of the form 
\ekv{pr.1}
{P_\delta =P_0+\delta Q,\ \delta Q=\delta h^{N_1}q(x),\ \delta =\frac{1}{C}h^{N_1+n}\tau_0,
}
where $q$ is an admissible potential as in (\ref{int.6.3}), (\ref{int.6.4}).
Notice that  
$D:=\# \{k;\, \mu _k\le L\}$ satisfies:
\ekv{pr.3}{D \le {\cal O}(L^n h^{-n})\le {\cal O}(h^{-N_3}),\ 
N_3:=n(M+1).
}
With $R$ as in (\ref{int.6.3}), we allow $\alpha $ to vary in the ball
\ekv{pr.4}
{
| \alpha |_{{\bf C}^D}\le 2R={\cal O}(h^{-\widetilde{M}}).
}
(Our probability measure will be supported in $B_{{\bf C}^D}(0,R)$ but
we will need to work in a larger ball.)

\par We consider the holomorphic function 
\ekv{pr.5}
{
F(\alpha )=(\det P_{\delta ,z})\exp (-\frac{1}{(2\pi h)^n}\iint \ln
|p_z| dxd\xi ).
}
 Then by \no{sp.31}, we have
\ekv{pr.6}
{
\ln |F(\alpha )|\le \epsilon _0(h)h^{-n}, \ |\alpha |<2R,
}
and for one particular value $\alpha =\alpha ^0$ with $|\alpha ^0|\le 
\frac{1}{2}R$, corresponding to the special potential in Proposition \ref{sp2}:
\ekv{pr.7}
{
\ln |F(\alpha ^0 )|\ge -\epsilon _0(h)h^{-n},
}
where $\epsilon _0(h)$ is given in (\ref{int.6.7.5}).

\par Let $\alpha ^1\in {\bf C}^D$ with $|\alpha ^1|=R$ and 
consider the holomorphic function of one complex variable
\ekv{pr.9}
{
f(w)=F(\alpha ^0+w\alpha ^1).
}
We will
mainly consider this function for $w$ in the disc 
determined by the condition $|\alpha ^0+w\alpha ^1|<R$:
\ekv{pr.10}
{
D_{\alpha ^0,\alpha ^1}:\left | w+\left( \frac{\alpha ^0}{R} |
\frac{\alpha ^1}{R}\right) \right|^2<1-\left| \frac{\alpha
^0}{R}\right|^2+\left|\left(\frac{\alpha^0}{R}|\frac{\alpha^1}{R}\right)\right|
^2=:r_0^2,}
whose radius is between $\frac{\sqrt{3}}{2}$ and $1$. 

\par From \no{pr.6}, \no{pr.7} we get 
\ekv{pr.11}
{
\ln |f(0)|\ge -\epsilon _0(h) h^{-n},\ 
\ln |f(w)|\le \epsilon _0(h)h^{-n}.
} 
By \no{pr.6}, we may
assume that the last estimate holds in a larger disc, say 
$D(-(\frac{\alpha ^0}{R}|\frac{\alpha ^1}{R}),2r_0)$. Let
$w_1,...,w_M$ be the zeros of $f$ in  
$D(-(\frac{\alpha ^0}{R}|\frac{\alpha ^1}{R}),3r_0/2)$. Then it is
standard to get the factorization 
\ekv{pr.12}
{
f(w)=e^{g(w)}\prod_1^M (w-w_j),\  w\in D(-(\frac{\alpha ^0}{R}|\frac{\alpha ^1}{R}),4r_0/3),
}
together with the bounds
\ekv{pr.13}{|\Re g(w)|\le {\cal O}(\epsilon _0(h)h^{-n}),\ 
M={\cal O}(\epsilon _0(h)h^{-n}).}
See for instance Section 5 in \cite{Sj} where further references are also given.

\par For $0<\epsilon \ll 1$, put 
\ekv{pr.14}
{
\Omega (\epsilon )=\{ r\in [0,r_0[;\, \exists w\in D_{\alpha ^0,\alpha
  ^1}
\hbox{ such that }|w|=r\hbox{ and }|f(w)|<\epsilon \} .
}
If $r\in \Omega (\epsilon )$ and $w$ is a corresponding point in
$D_{\alpha ^0,\alpha ^1}$, we have with $r_j=|w_j|$,
\ekv{pr.14.5}
{\prod_1^M |r-r_j| \le \prod _1^M|w-w_j|\le \epsilon \exp ({\cal
  O}(\epsilon _0(h)h^{-n})).}

Then at least one of the factors $|r-r_j|$ is bounded by 
$ (\epsilon e^{{\cal O}(\epsilon _0(h)h^{-n})})^{1/M}  $. 
In particular, the Lebesgue measure $\lambda (\Omega
(\epsilon ))$ of $\Omega (\epsilon )$ is bounded by 
$2M(\epsilon e^{{\cal O}(\epsilon _0(h)h^{-n})})^{1/M}$. 
Noticing that the last bound increases with $M$ when the last member
of (\ref{pr.14.5}) is $\le 1$,
we get
\begin{prop}\label{pr1} Let $\alpha ^1\in {\bf C}^D$ with 
$|\alpha ^1|=R$ and assume that $\epsilon >0$ is small enough so that the last member of \no{pr.14.5} is
  $\le 1$.
Then 
\eekv{pr.15}
{
\lambda (\{ r\in [0,r_0];\ |\alpha ^0+r\alpha ^1|<R,\ |F(\alpha
^0+r\alpha ^1)|<\epsilon \}) 
\le}
{\frac{\epsilon _0(h)}{h^n}\exp ({\cal O}(1)+\frac{h^n}{{\cal O}(1)
  \epsilon _0(h)}\ln \epsilon ).
}
Here and in the following, the symbol ${\cal O}(1)$ in a denominator
indicates a bounded positive quantity.
\end{prop}

\par Typically, we can choose $\epsilon =\exp -\frac{\epsilon
  _0(h)}{h^{n+\alpha }}$ for some small $\alpha >0$ and then the upper
bound in \no{pr.15} becomes 
$$
\frac{\epsilon _0(h)}{h^n}\exp ({\cal O}(1)-\frac{1}{{\cal O}(1)h^{\alpha }}).
$$

\par
Now we equip $B_{{\bf C}^D}(0,R)$ with a probability measure of the
form
\ekv{pr.16}
{
P(d\alpha )=C(h)e^{\Phi (\alpha )}L(d\alpha ),
} 
where $L(d\alpha )$ is the Lebesgue measure, $\Phi $ is a $C^1$
function which depends on $h$ and satisfies
\ekv{pr.17}
{
\vert \nabla \Phi \vert ={\cal O}(h^{-N_4}),
}
and $C(h)$ is the appropriate normalization constant.

\par Writing $\alpha =\alpha ^0+Rr\alpha ^1$, $0\le r<r_0(\alpha ^1)$,
$\alpha ^1\in S^{2D-1}$, $\frac{\sqrt{3}}{2}\le r_0\le 1$, we get 
\ekv{pr.18}
{
P(d\alpha )=\widetilde{C}(h)e^{\phi (r)}r^{2D-1}dr S(d\alpha ^1),
}
where $\phi (r)=\phi _{\alpha ^0,\alpha ^1}(r)=\Phi (\alpha
^0+rR\alpha ^1)$ so that $\phi '(r)={\cal O}(h^{-N_5})$,
$N_5=N_4+\widetilde{M}$. 
Here
$S(d\alpha ^1)$ denotes the Lebesgue measure on $S^{2D-1}$.

\par For a fixed $\alpha ^1$, we consider the normalized measure 
\ekv{pr.19}{
\mu (dr)=\widehat{C }(h)e^{\phi (r)}r^{2D-1}dr
}
on $[0,r_0(\alpha ^1)]$
and we want to show an estimate similar to \no{pr.15} for $\mu $
instead of $\lambda $. Write 
$e^{\phi (r)}r^{2D-1}=\exp (\phi (r)+(2D-1)\ln r)$ and consider
the derivative of the exponent,
$$
\phi '(r)+\frac{2D-1}{r}.
$$
This derivative is $\ge 0$ for $r\le
\frac{h^{[N_5-N_3]_+}}{C}=:2\widetilde{r}_0$, where we may assume that
$2\widetilde{r}_0\le r_0$. Introduce the measure
$\widetilde{\mu }\ge \mu $ by
\ekv{pr.20}
{
\widetilde{\mu }(dr)=\widehat{C}(h)e^{\phi (r_{\rm max} )}r_{\rm max}^{2D-1}dr,\
r_{\rm max} :=\max (r,\widetilde{r}_0).
}
Since $\widetilde{\mu }([0,\widetilde{r}_0])\le \mu ([\widetilde{r}_0,2\widetilde{r}_0])$, we get
\ekv{pr.21}
{
\widetilde{\mu }([0,r(\alpha ^1)])\le {\cal
  O}(1).
}
We can write 
\ekv{pr.22}
{
\widetilde{\mu }(dr)=\widehat{C}(h)e^{\psi (r)}dr,
}
where 
\eekv{pr.23}
{&
\psi '(r)={\cal O}(1)(h^{-N_5}+h^{-N_3+[N_5-N_3]_+})={\cal O}(h^{-N_6}),&
}
{&N_6=\max (N_3,N_5).&}
Cf (\ref{pr.3}).

\par  We now decompose $[0,r_0(\alpha ^1)]$ into $\asymp h^{-N_6}$
intervals of length $\asymp h^{N_6}$. If $I$ is such an interval, we see that 
\ekv{pr.24}
{
\frac{\lambda (dr)}{C\lambda (I)}\le \frac{\widetilde{\mu
  }(dr)}{\widetilde{\mu }(I)}\le C\frac{\lambda (dr)}{\lambda
  (I)}\hbox{ on }I.
}

From \no{pr.15}, \no{pr.24} we get when the right hand side of 
\no{pr.14.5} is $\le 1$,
\begin{eqnarray*}
\widetilde{\mu }(\{ r\in I;\, |F(\alpha ^0+rR\alpha ^1)|<\epsilon
\})/\widetilde{\mu }(I)&\le& \frac{{\cal O}(1)}{\lambda (I)}
\frac{\epsilon _0(h)}{h^n}\exp (\frac{h^n}{{\cal O}(1)\epsilon _0(h)}\ln
\epsilon )\\
&=& {\cal O}(1) h^{-N_6}
\frac{\epsilon _0(h)}{h^n}\exp (\frac{h^n}{{\cal O}(1)\epsilon _0(h)}\ln
\epsilon ).
\end{eqnarray*}
Multiplying with $\widetilde{\mu }(I)$ and summing the estimates over $I$ we get 
\ekv{pr.25}
{
\widetilde{\mu }(\{ r\in [0,r(\alpha ^1)];\, |F(\alpha ^0+rR\alpha
^1)|<\epsilon \})\le {\cal O}(1)h^{-N_6}\frac{\epsilon _0(h)}{h^n}
\exp (\frac{h^n}{{\cal O}(1)\epsilon _0(h)}\ln \epsilon )
.
}
Since $\mu \le \widetilde{\mu }$, we get the same estimate with
$\widetilde{\mu }$ replaced by $\mu $. Then from \no{pr.18} we get
\begin{prop}\label{pr2}
Let $\epsilon >0$ be small enough for 
the right hand side of \no{pr.14.5} to be $\le 1$. Then
\ekv{pr.26}
{
P(|F(\alpha )|<\epsilon )\le {\cal O}(1) h^{-N_6}
\frac{\epsilon _0(h)}{h^n}\exp (\frac{h^n}{{\cal O}(1)\epsilon
  _0(h)}\ln \epsilon ).
}
\end{prop}

\begin{remark}\label{pr3}
{\rm In the case when $\widetilde{R}$ has real coefficients, we may assume
that the eigenfunctions $\epsilon _j$ are real, and from the observation after
Proposition \ref{spe2} we see that we can choose $\alpha _0$ above to
be real. The discussion above can then be restricted to the case of
real $\alpha ^1$ and hence to real $\alpha $. We can then introduce
the probability measure $P$ as in \no{pr.16} on the real ball $B_{{\bf
    R}^D}(0,R)$. The subsequent discussion goes through without any
changes, and we still have the conclusion of Proposition \ref{pr2}.}
\end{remark}

\subsection{End of the proof of the main result}\label{en}

We now work under the assumptions of Theorem \ref{int1}. For $z$ in a
fixed neighborhood of $\Gamma $,
we rephrase \no{sp.31} as
\ekv{en.1}
{
|\det P_{\delta ,z}| \le \exp \frac{1}{h^n}(\frac{1}{(2\pi )^n}
\iint \ln |p_z| dxd\xi +\epsilon _0(h)),
}
where $\epsilon _0(h)$ is given in \no{int.6.7.5}. 
Moreover, Proposition \ref{pr2} shows that with probability 
\ekv{en.3}
{
\ge 1-{\cal O}(1)h^{-N_6-n}\epsilon _0(h)e^{-\frac{
h^n}{{\cal O}(1)\epsilon _0(h)}\ln \frac{1}{\epsilon }},
}
we have
\ekv{en.4}{
|\det P_{\delta ,z}|\ge \epsilon \exp (\frac{1}{h^n}(\frac{1}{(2\pi
  )^n})
\iint \ln |p_z| dxd\xi ),
}
provided that $\epsilon >0$ is small enough so that
\ekv{en.5}
{
\hbox{The right hand side of \no{pr.14.5} is }\le 1,\,\forall \alpha
^1\in S^{2D-1}.
}

\par Write $\epsilon =e^{-\widetilde{\epsilon }/h^n}$,
$\widetilde{\epsilon }=h^n\ln \frac{1}{\epsilon }$. Then \no{en.5}
holds if 
\ekv{en.9}
{
\widetilde{\epsilon }\ge C\epsilon _0(h),
}
for some large constant $C$. \no{en.3}, \no{en.4} can be
rephrased by saying that with probability
\ekv{en.10}
{
\ge 1-{\cal O}(1)h^{-N_6-n}\epsilon _0(h)e^{-\frac{1}{C}\frac{\widetilde{\epsilon }}{\epsilon _0(h)}},
}
we have
\ekv{en.11}
{
|\det P_{\delta ,z}|\ge \exp \frac{1}{h^n}(\frac{1}{(2\pi )^n}\iint
\ln |p_z|dxd\xi -\widetilde{\epsilon }).
}
This is of interest for $\widetilde{\epsilon }$ in the range
\ekv{en.12}
{
\epsilon _0(h)\ll \widetilde{\epsilon }\ll 1.
}

\par Now, let $\Gamma \Subset \Omega $ be connected with smooth
boundary. Recall that $0<\kappa \le 1$ and that 
\ekv{en.13}
{\hbox{(\ref{int.6.2}) holds uniformly for all }z\hbox{ in some neighborhood of
  }\partial \Gamma .}
Then the function
\ekv{en.14}
{
\phi (z) =\frac{1}{(2\pi )^n}\iint \ln |p_z|dxd\xi 
}
is continuous and subharmonic in a neighborhood of
$\partial \Gamma $. We shall apply Theorem \ref{ze2}, with $0<r\ll 1$ constant, to the holomorphic function 
$$
u(z)=\det P_{\delta ,z}.
$$
Then, according to \no{en.10}, \no{en.11} we know that with
probability
\ekv{en.15}
{
\ge 1-\frac{{\cal O}(1)\epsilon _0(h)}{rh^{N_6+n}}e^{-
\frac{\widetilde{\epsilon }}{{\cal O}(1)\epsilon _0(h)}}
}
we have 
\ekv{en.16}
{
h^n\ln |u(\widetilde{z}_j)|\ge \phi (\widetilde{z}_j)-\widetilde{\epsilon },\ j=1,...,N,\quad N\asymp \frac{1}{r}.
}

\par In a full neighborhood of $\partial \Gamma $ we also have
\ekv{en.17}
{
h^n\ln |u(z)|\le \phi (z)+C\widetilde{\epsilon }.
}
We conclude from Theorem \ref{ze2} that with probability
bounded from below as in \no{en.15}
 we have for every $\widehat{M}>0$:
\eekv{en.18}
{
&&|
\#(u^{-1}(0)\cap \Gamma )-\frac{1}{h^n2\pi }\int_\Gamma \Delta \phi L(dz)
|\le
}
{&&
\frac{{\cal O}(1)}{h^n}\left( \frac{\widetilde{\epsilon }}{r}
+\mu (\partial \Gamma +D(0,r))
 \right),
}
where $\mu $ denotes the measure $\Delta \phi L(dz)$.

According to Section 10 in \cite{HaSj08}, 
the measure $\frac{1}{2\pi
}\Delta \phi L(dz)$ is the push forward under $p$ of $(2\pi )^{-n}$
times the symplectic volume element, and we
can replace $\frac{1}{2\pi
}\Delta \phi L(dz)$ by this push forward in \no{en.18}. Moreover $u^{-1}(0)$ is the set
of eigenvalues of $P_\delta $ so we can rephrase \no{en.18} as 
\eekv{en.19}
{
&&|
\#(\sigma (P_\delta )\cap \Gamma )-\frac{1}{(2\pi h)^n
}\mathrm{vol\,}(p^{-1}(\Gamma ))
|\le
}
{&&
\frac{{\cal O}(1)}{h^n}\left( \frac{\widetilde{\epsilon }}{r}
+\mathrm{vol\,}(p^{-1}(\partial
\Gamma +D(0,r)))
 \right).
}
This concludes the proof of Theorem \ref{int1}, with $P$ replaced by
the slightly more general operator $P_0$.

\section{Almost sure Weyl asymptotics of large eigenvalues}\label{alm}
\setcounter{equation}{0}

\subsection{Introduction}\label{intalm}

W.~Bordeaux Montrieux \cite{Bo} has studied
elliptic systems of differential operators on $S^1$
with random perturbations of the coefficients,
and under some additional assumptions, he showed that the large eigenvalues obey the
Weyl law \emph{almost surely}. His analysis was based on a reduction to
the semi-classical case (using essentially the Borel-Cantelli lemma), 
where he could use and extend the methods of Hager \cite{Ha06b}. 

The purpose of this section is to describe the work of Bordeaux Montrieux and the author \cite{BoSj09} on the almost sure Weyl asymptotics of the large eigenvalues of elliptic operators on compact manifolds. For simplicity, we treat only the scalar case
and the random perturbation is a potential.

Let $X$ be a smooth compact manifold of dimension $n$. Let $P^0$ be an
elliptic differential operator on $X$ of order $m\ge 2$ with smooth
coefficients and with
principal symbol $p(x,\xi )$. In local coordinates we get, using
standard multi-index notation,
\ekv{in.1}
{
P^0=\sum_{|\alpha |\le m}a_\alpha ^0(x)D^\alpha ,\quad 
p(x,\xi )=\sum_{|\alpha |= m}a_\alpha ^0(x)\xi ^\alpha.
}
Recall that the ellipticity of $P^0$ means that $p(x,\xi )\ne 0$ for
$\xi \ne 0$. We assume that
\ekv{in.2}
{
p(T^*X)\ne {\bf C}.
}
Fix a strictly positive smooth density of integration $dx$ on $X$, so
that the $L^2$ norm $\Vert \cdot \Vert$ and inner product $(\cdot
|\cdot \cdot )$ are unambiguously defined. Let $\Gamma :L^2(X)\to
L^2(X)$ be the antilinear operator of complex conjugation, given by
$\Gamma u=\overline{u}$. We need the symmetry assumption
\ekv{in.3}
{
P^*=\Gamma P\Gamma ,\ (\hbox{or equivalently, } P^\mathrm{t}=P)
}
where $P^*$ is the formal complex adjoint of $P$. As in \cite{Sj08b} we
observe that the property (\ref{in.3}) implies that
\ekv{in.4}
{
p(x,-\xi )=p(x,\xi ),
}
and conversely, if (\ref{in.4}) holds, then the operator
$\frac{1}{2}(P+\Gamma P\Gamma )$ has the same principal symbol $p$ and
satisfies (\ref{in.3}).

\par Let $\widetilde{R}$ be an elliptic differential operator on $X$
with smooth coefficients, which is self-adjoint and strictly
positive. Let $\epsilon _0,\epsilon _1,...$ be an orthonormal basis of
eigenfunctions of $\widetilde{R}$ so that 
\ekv{in.5}
{
\widetilde{R}\epsilon _j=(\mu _j^0)^2\epsilon _j,\quad 0<\mu _0^0<\mu
_1^0\le \mu _2^0\le ...
}
Our randomly perturbed operator is 
\ekv{in.6}
{
P_\omega ^0=P+q_\omega ^0(x),
}
where $\omega $ is the random parameter and 
\ekv{in.7}
{
q_\omega ^0(x)=\sum_{0}^\infty \alpha _j^0(\omega )\epsilon _j.
}
Here we assume that $\alpha _j^0(\omega )$ are independent complex
Gaussian random variables of variance $\sigma _j^2$ and mean value 0:
\ekv{in.8}
{
\alpha _j^0\sim {\cal N}(0,\sigma _j^2),
}
where 
\ekv{in.8.5}
{
(\mu _j^0)^{-\rho }e^{-(\mu _j^0)^{\frac{\beta }{M+1}}}\lesssim 
\sigma _j\lesssim (\mu _j^0)^{-\rho },
}
\ekv{in.9}
{M=\frac{3n-\frac{1}{2}}{s-\frac{n}{2}-\epsilon },\ 0\le \beta
  <\frac{1}{2},\ 
\rho >n,
}
where $s$, $\rho $, $\epsilon $ are fixed constants such that
$$
\frac{n}{2}<s<\rho -\frac{n}{2},\ 0<\epsilon <s-\frac{n}{2}.
$$

\par Let $H^s(X)$ be the standard Sobolev space of order $s$. As will
follow from considerations below, we have $q_\omega^0 \in H^s(X)$ almost
surely since $s<\rho -\frac{n}{2}$. Hence $q_\omega^0\in L^\infty $
almost surely, implying that $P_\omega ^0$ has purely discrete
spectrum.

\par Consider the function $F(\omega )=\mathrm{arg\,}p(\omega )$ on
$S^*X$. For given $\theta _0\in S^1\simeq {\bf R}/(2\pi {\bf Z})$,
$N_0\in \dot{{\bf N}}:={\bf N}\setminus \{ 0\}$, we introduce the property:
\ekv{in.10}
{P(\theta _0,N_0):\quad
\sum_1^{N_0}|\nabla ^kF(\omega )|\ne 0\hbox{ on }\{ \omega \in S^*X;\,
F(\omega )=\theta _0\}.
}
Notice that if $P(\theta _0,N_0)$ holds, then $P(\theta ,N_0)$ holds
for all $\theta $ in some neighborhood of $\theta _0$.

\par We can now state our main result.

\begin{theo}\label{in1} (\cite{BoSj09})
Assume that $m\ge 2$. Let $0\le \theta _1\le \theta _2\le 2\pi $ and
assume that $P(\theta _1,N_0)$ and $P(\theta _2,N_0)$ hold for some
$N_0\in\dot{{\bf N}}$. Let $g\in C^\infty ([\theta _1,\theta
_2];]0,\infty [)$ and put 
$$
\Gamma ^g_{\theta _1,\theta _2;0,\lambda }=\{ re^{i\theta } ; \theta
  _1\le \theta \le \theta _2,\ 0\le r\le \lambda g(\theta )\}.
$$
Then for every $\delta \in ]0,\frac{1}{2}-\beta [$ there exists $C>0$ such
that almost surely: $\exists C(\omega )<\infty $ such that for all
$\lambda \in [1,\infty [$:
\eekv{in.11}
{
|\#(\sigma (P_\omega ^0)\cap \Gamma _{\theta _1,\theta _2;0,\lambda }^g)
-\frac{1}{(2\pi )^n}\mathrm{vol\,}p^{-1}(\Gamma ^g_{\theta _1,\theta
  _2;0,\lambda })
|
}
{\le C(\omega )+C\lambda ^{\frac{n}{m}-\frac{1}{m}(\frac{1}{2}-\beta -\delta
    )\frac{1}{N_0+1}}.}
Here $\sigma (P_\omega ^0)$ denotes the spectrum and $\# (A)$ denotes
the number of elements in the set $A$. In (\ref{in.11}) the
eigenvalues are counted with their algebraic multiplicity.
\end{theo}

The proof actually allows to have almost surely a simultaneous
conclusion for a whole family of $\theta _1,\theta _2,g$:

\begin{theo}\label{in2}
Assume that $m\ge 2$. Let $\Theta $ be a compact subset of $[0,2\pi ]$. Let
$N_0\in {\bf N}$ and assume that $P(\theta ,N_0)$ holds uniformly for
$\theta \in \Theta $. Let ${\cal G}$ be a subset of $\{(g,\theta
_1,\theta _2);\ \theta _j\in \Theta, \theta _1\le \theta _2,\ g\in 
C^\infty ([\theta _1,\theta
_2];]0,\infty [)
\}$ with the property that $g$ and $1/g$ are uniformly bounded in $C^\infty ([\theta _1,\theta
_2];]0,\infty [)$ when $(g,\theta _1,\theta _2)$ varies in ${\cal
  G}$. Then for every $\delta \in ]0,\frac{1}{2}-\beta [$ 
there exists $C>0$ such
that almost surely: $\exists C(\omega )<\infty $ such that for all
$\lambda \in [1,\infty [$ and all $(g,\theta _1,\theta _2)\in {\cal
  G}$, we have the estimate (\ref{in.11}).
\end{theo}

The condition (\ref{in.8.5}) allows us to choose $\sigma _j$
 decaying faster than any negative power of
$\mu _j^0$. Then from the
discussion below, it will follow that $q_\omega (x)$ is almost
surely a smooth function. A rough and somewhat intuitive
interpretation of Theorem \ref{in2} is then that for 
almost every elliptic
operator of order $\ge 2$ with smooth coefficients on a compact manifold which satisfies the 
conditions (\ref{in.2}), (\ref{in.3}), the large eigenvalues
distribute according to Weyl's law in sectors with limiting directions
that satisfy a weak non-degeneracy condition.

\subsection{Volume considerations}\label{vo}

In the next subsection we shall perform a reduction to a semi-classical
situation and work with $h^mP_0$ which has the semi-classical
principal symbol $p$ in (\ref{in.1}). Again,
\ekv{vo.1}
{
V_z(t)=\mathrm{vol\,}\{ \rho \in T^*X;\, |p(\rho )-z|^2\le t\},\ t\ge 0.
}
\begin{prop}\label{vo1}
For any compact set $K\subset \dot{{\bf C}}={\bf C}\setminus \{ 0\}$, we have
\ekv{vo.2}
{
V_z(t)={\cal O}(t^\kappa ),\hbox{ uniformly for }z\in K,\ 0\le t\ll 1,
}
with $\kappa =1/2$.
\end{prop}

This follows from:
\begin{prop}\label{vo2}
Let $\gamma $ be the curve $\{ re^{i\theta }\in {\bf C};\, r=g(\theta
),\ \theta \in S^1 \}$, where $0<g\in C^1(S^1)$. Then 
$$
\mathrm{vol\,}(p^{-1}(\gamma +D(0,t)))={\cal O}(t),\ t\to 0.
$$
\end{prop}

This follows from the fact that the radial derivative of $p$ is $\ne
0$.

\par Using (\ref{in.10}), we can prove:
\begin{prop}\label{vo3}
Let $\theta _0\in S^1$, $N_0\in\dot{{\bf N}}$ and assume that
$P(\theta _0,N_0)$ holds. Then if $0<r_1<r_2$ and $\gamma $ 
is the radial segment $[r_1,r_2]e^{i\theta _0}$, we have 
$$
\mathrm{vol\,}(p^{-1}(\gamma +D(0,t)))={\cal O}(t^{1/N_0}),\ t\to 0.
$$
\end{prop}

\par Now, let $0\le \theta _1<\theta _2\le 2\pi $, $g\in C^\infty
([\theta _1,\theta _2];]0,\infty [)$ and put 
\ekv{vo.4}
{
\Gamma ^g_{\theta _1,\theta _2;r_1,r_2}=\{ re^{i\theta };\, \theta
_1\le \theta \le \theta _2,\ r_1g(\theta )\le r\le r_2g(\theta )\} ,
}
for $0\le r_1\le r_2 <\infty $. If $0<r_1<r_2<+\infty $ and
$P(\theta _j,N_0)$ hold for $j=1,2$, then the last two propositions
imply that 
\ekv{vo.5}
{
\mathrm{vol\,}p^{-1}(\partial \Gamma ^g_{\theta _1,\theta
  _2;r_1,r_2}+D(0,t))={\cal O}(t^{1/N_0}),\ t\to 0.
}

\subsection{Semiclassical reduction}\label{sc}

We are interested in the distribution of large eigenvalues $\zeta $ of
$P_\omega ^0$, so we make a standard reduction to a semi-classical
problem by letting $0<h\ll 1$ satisfy
\ekv{sc.1}
{
\zeta =\frac{z}{h^m},\ |z|\asymp 1,\ h\asymp |\zeta |^{-1/m},
}
and write
\ekv{sc.2}
{
h^m(P_\omega ^0-\zeta )=h^mP_\omega ^0-z=:P+h^mq^0_\omega -z,
}
where
\ekv{sc.3}
{
P=h^mP^0=\sum_{|\alpha |\le m} a_\alpha (x;h)(hD)^\alpha .
}
Here 
\eekv{sc.4}
{a_\alpha (x;h)&=&{\cal O}(h^{m-|\alpha |})\hbox{ in }C^\infty ,}
{a_\alpha (x;h)&=&a_\alpha ^0(x)\hbox{ when }|\alpha |=m.}
So $P$ is a standard semi-classical differential operator with
semi-classical principal symbol $p(x,\xi )$.

Our strategy will be to decompose the random perturbation
$$
h^mq_\omega ^0=\delta Q_\omega +k_\omega (x),
$$
where the two terms are independent, 
and with probability very close to 1, $\delta Q_\omega $ will be a
semi-classical random perturbation as in Section \ref{mult} while
\ekv{sc.5}
{
\Vert k_\omega \Vert_{H^s}\le h,
}
and 
\ekv{sc.5.5}{
s\in ]\frac{n}{2},\rho -\frac{n}{2}[} is fixed. Then $h^mP_\omega
^0$ will be viewed as a random perturbation of $h^mP^0+k_\omega $.
In order to achieve this without
extra assumptions on the order $m$, we will also have to represent
some of our eigenvalues $\alpha _j^0(\omega )$ as sums of two
independent Gaussian random variables.

\par We start by examining when
\ekv{sc.6}
{
\Vert h^mq_\omega ^0\Vert_{H^s}\le h.
}
\begin{prop}\label{sc1}
There is a constant $C>0$ such that (\ref{sc.6}) holds with
probability
$$
\ge 1-\exp (C-\frac{1}{2Ch^{2(m-1)}}).
$$
\end{prop}

Here is a brief outline of the proof.
We have
\ekv{sc.7}
{
h^mq_\omega ^0=\sum_0^\infty \alpha _j(\omega )\epsilon _j,\quad
\alpha _j=h^m\alpha _j^0\sim {\cal N}(0,(h^m\sigma _j)^2),
}
and the $\alpha _j$ are independent. Now, by the functional
characterization of $H^s$ in Subsection \ref{al}, we get
\ekv{sc.8}
{
\Vert h^mq_\omega ^0\Vert_{H^s}^2\asymp \sum _0^\infty  |(\mu _j^0)^s\alpha _j(\omega )|^2,
}
where $(\mu _j^0)^s\alpha _j\sim {\cal N}(0,(\widetilde{\sigma }_j)^2)$ are independent random variables and $\widetilde{\sigma }_j=(\mu _j^0)^sh^m\sigma _j$.

Combining this with Proposition \ref{51} and standard Weyl asymptotics for $\widetilde{R}$ leads to the result.

\par Write 
\ekv{sc.19}{q_\omega ^0=q_\omega ^1+q_\omega ^2,}
\ekv{sc.20}
{
q_\omega ^1=\sum_{0<h\mu _j^0\le L}\alpha _j^0(\omega )\epsilon _j,\
q_\omega ^2=\sum_{h\mu _j^0 > L}\alpha _j^0(\omega )\epsilon _j.
}
From Proposition \ref{sc1} and its proof, we know that
\ekv{sc.21}
{
\| h^mq_\omega ^2\|_{H^s}\le h\hbox{ with probability }\ge 1-\exp (C_0-\frac{1}{2Ch^{2(m-1)}}).
}
We write
$$
P+h^mq_\omega ^0=(P+h^mq_\omega ^2)+h^mq_\omega ^1,
$$
Theorem \ref{int1} can be applied with $P$
replaced by the perturbation $P+h^mq_\omega ^2$, provided that $\|
h^mq_\omega ^2\|_{H^s}\le h$.

The next question is then wether $h^mq_\omega ^1$ can be written as
$\tau_0h^{2N_1+n}q_\omega $ where $q_\omega =\sum_{0<h\mu _j^0\le
  L}\alpha _j\epsilon _j$ and $\vert \alpha \vert_{{\bf C}^D}\le R$
with probability close to 1. This turns out to be impossible without extra assumptions.

\par In order to avoid such an extra assumption, we shall now
represent $\alpha _j^0$ for $h\mu _j^0\le L$ as the sum of two
independent Gaussian random variables. Let $j_0=j_0(h)$ be the largest $j$
for which $h\mu _j^0\le L$. Put 
\ekv{sc.24}
{
\sigma '=\frac{1}{C}h^Ke^{-Ch^{-\beta }},\hbox{ where }K\ge \rho
(M+1),\ C\gg 1
}
so that $\sigma '\le \frac{1}{2}\sigma _j$ for $1\le j\le j_0(h)$. The
factor $h^K$ is needed only when $\beta =0$.

\par For $j\le j_0$, we may assume that $\alpha _j^0(\omega )=\alpha
_j'(\omega )+\alpha _j''(\omega ),$ where $\alpha _j'\sim {\cal
  N}(0,(\sigma ')^2)$, $\alpha _j''\sim {\cal N}(0,(\sigma _j'')^2)$ are independent
random variables and 
$$
\sigma _j^2=(\sigma ')^2+(\sigma _j'')^2,
$$
so that 
$$
\sigma _j''=\sqrt{\sigma _j^2-(\sigma ')^2}\asymp \sigma _j .
$$

\par Put $q_\omega ^1 =q_\omega '+q_\omega ''$, where
$$
q_\omega '=\sum_{h\mu _j^0\le L}\alpha _j'(\omega )\epsilon _j,\
q_\omega ''=\sum_{h\mu _j^0\le L}\alpha _j''(\omega )\epsilon _j.
$$
Now (cf (\ref{sc.19})) we write
$$
P+h^mq_\omega ^0=(P+h^m(q_\omega ''+q_\omega^2))+h^mq_\omega '. 
$$
Theorem \ref{int1} is valid for random perturbations of
$$
P_0:=P+h^m(q_\omega ''+q_\omega ^2),
$$
provided that $\Vert h^m(q_\omega ''+q_\omega ^2)\Vert_{H^s}\le h$,
which again holds with a probability as in (\ref{sc.21}). The new random
perturbation is now $h^mq_\omega '$ which we write as $\tau
_0h^{2N_1+n}\widetilde{q}_\omega $, where $\widetilde{q}_\omega $
takes the form
\ekv{sc.25}{
\widetilde{q}_\omega (x)=\sum_{0<h\mu _j^0\le L}\beta _j(\omega
)\epsilon _j,
}
with new independent random variables 
\ekv{sc.26}
{
\beta _j=\frac{1}{\tau _0}h^{m-2N_1-n}\alpha _j'(\omega )\sim {\cal
  N}(0,(\frac{1}{\tau _0}h^{m-2N_1-n}\sigma '(h))^2).
}

\par Now, by Proposition \ref{51},
$$
{\bf P}(\vert \beta \vert_{{\bf C}^D}^2>R^2)\le \exp ({\cal
  O}(1)D\frac{h^{m-2N_1-n}\sigma '(h)}{\tau_0}-\frac{R^2\tau _0^2}{{\cal O}(1)(h^{m-2N_1-n}\sigma '(h))^2}).
$$
Here by Weyl's law for the distribution of eigenvalues of elliptic
self-adjoint differential operators, we have $D\asymp (L/h)^n$. Moreover,
$L,R$ behave like certain powers of $h$. 
\begin{itemize}
\item In the case when $\beta =0$, we choose $\tau _0=h^{1/2}$.
Then for any $a>0$ we get 
$$
{\bf P}(\vert \beta \vert_{{\bf C}^D}>R)\le C\exp (-\frac{1}{Ch ^a})
$$  for any given fixed $a$, provided we choose $K$ large enough in 
(\ref{sc.24}).
\item In the case $\beta >0$ we get the same conclusion with $\tau
  _0=h^{-K}\sigma '$ if $K$ is large enough.
\end{itemize}

\par In both cases, we see that the independent random
variables $\beta _j$ in (\ref{sc.25}), (\ref{sc.26}) have a joint
probability density $C(h)e^{\Phi (\alpha ;h)}L(d\alpha )$, satisfying (\ref{int.6.6}) for
some $N_4$ depending on $K$.

\par With $\kappa =1/2$, we put 
$$
\epsilon _0(h)=h^{\kappa } ((\ln \frac{1}{h})^2+\ln \frac{1}{\tau _0}),
$$
where $\tau _0$ is chosen as above. Notice that $\epsilon _0(h)$ is of
the order of magnitude $h^{\kappa -\beta }$ up to a power of $\ln
\frac{1}{h}$. Then Theorem \ref{int1} gives:
\begin{prop}\label{sc2}
There exists a constant $N_4>0$ depending on $\rho ,n,m$ such that the
following holds: Let $\Gamma \Subset \dot{{\bf C}}$ have piecewise
smooth boundary. Then $\exists C>0$ such that for $0<r\le 1/C$,
$\widetilde{\epsilon }\ge C\epsilon _0(h)$, we have with probability
\ekv{sc.27}
{
\ge 1-\frac{C\epsilon _0(h)}{rh^{n+\max
    (n(M+1),N_4+\widetilde{M})}}e^{-\frac{\widetilde{\epsilon
    }}{C\epsilon _0(h)}}-Ce^{-\frac{1}{Ch}},
}
that 
\ekv{sc.28}
{
\vert \# (h^mP^0_\omega )\cap \Gamma )-\frac{1}{(2\pi
  h)^n}\mathrm{vol\,}(p^{-1}(\Gamma ))\vert 
\le \frac{C}{h^n}(\frac{\widetilde{\epsilon }}{r}+\mathrm{vol\,}
(p^{-1}(\partial \Gamma +D(0,r)))).}
\end{prop}

As already noted, this gives Weyl asymptotics provided that 
\ekv{sc.29}
{
\mathrm{vol\,}p^{-1}(\partial \Gamma +D(0,r))={\cal
  O}(r^\alpha ),
}
for some $\alpha \in ]0,1]$ (which would automatically be the case
if $\kappa $ had been larger than $1/2$ instead of being equal to
$1/2$), and we can then choose $r=\widetilde{\epsilon }^{1/(1+\alpha
  )}$, so that the right hand side of (\ref{sc.28}) becomes $\le
C\widetilde{\epsilon }^{\frac{\alpha }{1+\alpha }}h^{-n}$.

\par As in \cite{Sj08a, Sj08b} we also observe that if $\Gamma $ belongs
to a family ${\cal G}$ of domains satisfying the assumptions of the Proposition
uniformly, then with probability
\ekv{sc.30}
{
\ge 1-\frac{C\epsilon _0(h)}{r^2h^{n+\max
    (n(M+1),N_4+\widetilde{M})}}e^{-\frac{\widetilde{\epsilon
    }}{C\epsilon _0(h)}}-Ce^{-\frac{1}{Ch}},
}
the estimate (\ref{sc.28}) holds uniformly and simultaneously for all
$\Gamma \in {\cal G}$.

\subsection{End of the proof}\label{end}

Let $\theta _1,\theta _2,N_0$ be as in Theorem \ref{in1}, so that
$P(\theta _1,N_0)$ and $P(\theta _2,N_0)$ hold. Combining the
propositions \ref{vo1}, \ref{vo2}, \ref{vo3}, we
see that (\ref{sc.29}) holds with $\alpha =1/N_0$ when
$\Gamma =\Gamma ^g_{\theta _1,\theta _2;1,\lambda }$, $\lambda >0$
fixed, and Proposition \ref{sc2} gives:
\begin{prop}\label{end1}
With the parameters as in Proposition \ref{sc2} and for every $\alpha
\in ]0,\frac{1}{N_0}[$, we have with probability
\ekv{end.1}
{
\ge 1 -\frac{C\epsilon _0(h)}{\widetilde{\epsilon }^{\frac{N_0}{1+N_0
    }}h^{n+\max
    (n(M+1),N_4+\widetilde{M})}}e^{-\frac{\widetilde{\epsilon
    }}{C\epsilon _0(h)}}-Ce^{-\frac{1}{Ch}}
}
that
\ekv{end.2}
{
\vert \# (\sigma (h^mP_\omega )\cap \Gamma ^g_{\theta _1,\theta
  _2;1,\lambda })-\frac{1}{(2\pi h)^n}\mathrm{vol\,}(p^{-1}(\Gamma
^g_{\theta _1,\theta _2;1,\lambda }))\vert \le
C\frac{\widetilde{\epsilon }^{\frac{1}{1+N_0 }}}{h^n}.
}
Moreover, the conclusion (\ref{end.2}) is valid simultaneously for all
$\lambda \in [1,2]$ and all $(\theta _1,\theta _2)$ in a set where
$P(\theta _1,N_0)$, $P(\theta _2,N_0)$ hold uniformly, with probability
\ekv{end.3}
{
\ge 1 -\frac{C\epsilon _0(h)}{\widetilde{\epsilon }^{\frac{2N_0}{1+N_0
    }}h^{n+\max
    (n(M+1),N_4+\widetilde{M})}}e^{-\frac{\widetilde{\epsilon
    }}{C\epsilon _0(h)}}-Ce^{-\frac{1}{Ch}}.
}
\end{prop}

\par For $0<\delta \ll 1$, choose $\widetilde{\epsilon
}=h^{-\delta }\epsilon _0\le Ch^{\frac{1}{2}-\beta -\delta }(\ln \frac{1}{h})^2$, so that
$\widetilde{\epsilon }/\epsilon _0=h^{-\delta }$. Then for some $N_5$
we have for every $\alpha \in ]0,1/N_0[$ that 
\ekv{end.4}
{
\vert \# (\sigma (h^mP_\omega )\cap \Gamma ^g_{\theta _1,\theta
  _2;1,\lambda })-\frac{1}{(2\pi h)^n}\mathrm{vol\,}(p^{-1}(\Gamma
^g_{\theta _1,\theta _2;1,\lambda }))\vert \le \frac{C_\alpha
}{h^n}(h^{\frac{1}{2}-\delta -\beta }(\ln \frac{1}{h})^2)^{\frac{1
  }{1+N_0}},
}
simultaneously for $1\le \lambda \le 2$ and all $(\theta _1,\theta _2)$ in a set where
$P(\theta _1,N_0)$, $P(\theta _2,N_0)$ hold uniformly, with probability
\ekv{end.5}
{
\ge 1-\frac{C}{h^{N_5}}e^{-\frac{1}{Ch^\delta }}.
}
The upper bound in (\ref{end.4}) can be replaced by 
$$
\frac{C_\delta }{h^n}h^{(\frac{1}{2}-\beta -2\delta )/(N_0+1)}.
$$

Assuming 
$P(\theta _1,N_0)$, $P(\theta _2,N_0)$, we want to count
the number of eigenvalues of $P_\omega $ in 
$$
\Gamma _{1,\lambda }=\Gamma ^g_{\theta _1,\theta _2;1,\lambda }
$$
when $\lambda \to \infty $. Let $k(\lambda )$ be he largest integer
$k$ for which $2^k\le \lambda $ and decompose
$$
\Gamma _{1,\lambda }=(\bigcup_0^{k(\lambda )-1}\Gamma
_{2^k,2^{k+1}})\cup \Gamma _{2^{k(\lambda )},\lambda }.
$$
In order to count the eigenvalues of $P_\omega ^0$ in $\Gamma
_{2^k,2^{k+1}}$ we define $h$ by $h^m2^k=1$, $h=2^{-k/m}$, so that 
\begin{eqnarray*}
\# (\sigma (P_\omega ^0)\cap \Gamma _{2^k,2^{k+1}})&=&\# (\sigma
(h^mP^0_\omega )\cap \Gamma _{1,2}),\\ 
\frac{1}{(2\pi
  )^n}\mathrm{vol\,}(p^{-1}(\Gamma _{2^k,2^{k+1}}))&=&\frac{1}{(2\pi
 h)^n}\mathrm{vol\,}(p^{-1}(\Gamma _{1,2})).
\end{eqnarray*}
Thus, with probability $\ge 1-C2^{\frac{N_5k}{m}}e^{-2^\frac{\delta k}{m}/C}$ we have
\ekv{end.6}
{
\vert \# (\sigma (P_\omega ^0)\cap \Gamma _{2^k,2^{k+1}})
-\frac{1}{(2\pi )^n}\mathrm{vol\,}p^{-1}(\Gamma _{2^k,2^{k+1}})
\vert \le C_\delta 2^{\frac{kn}{m}}2^{-\frac{k}{m}(\frac{1}{2}-\beta -2\delta
)\frac{1}{N_0+1}}. 
}
Similarly, with probability $\ge 1-C2^{N_5k(\lambda )/m}
e^{-2^{\delta k(\lambda )/m}/C}$, we have
\ekv{end.7}
{
\vert \# (\sigma (P_\omega ^0)\cap \Gamma _{2^{k(\lambda
  )},\widetilde{\lambda }})
-\frac{1}{(2\pi )^n}\mathrm{vol\,}p^{-1}(\Gamma _{2^k(\lambda
  ),\widetilde{\lambda }})
\vert \le C_\delta \lambda ^{\frac{n}{m}}\lambda
^{-\frac{1}{m}(\frac{1}{2}-\beta -2\delta
)\frac{1}{N_0+1}}, 
}
simultaneously for all $\widetilde{\lambda }\in [\lambda ,2\lambda [$.

\par Now, we proceed as in \cite{Bo}, using essentially the Borel--Cantelli lemma. Use that
\begin{eqnarray*}
\sum_\ell^\infty 2^{N_5\frac{k}{m}}e^{-2^{\delta \frac{k}{m}}/C}&=&{\cal
  O}(1) 2^{N_5\frac{\ell}{m}}e^{-2^{\delta \frac{\ell}{m}}/C},\\
\sum_{2^k\le \lambda
}2^{k\frac{n}{m}}2^{-\frac{k}{m}(\frac{1}{2}-\beta -2\delta
  )\frac{1}{N_0+1}}&=&{\cal O}(1)\lambda
^{\frac{n}{m}-\frac{1}{m}(\frac{1}{2}-\beta -2\delta )\frac{1}{N_0+1}},
\end{eqnarray*}
to conclude that with probability $\ge
1-C2^{N_5\frac{\ell}{m}}e^{-2^{\delta \frac{\ell}{m}}/C}$, we have
$$
\vert \# (\sigma (P_\omega ^0)\cap \Gamma _{2^\ell ,\lambda })
-\frac{1}{(2\pi )^n } \mathrm{vol\,}p^{-1}(\Gamma _{2^\ell ,\lambda })
\vert
\le C_\delta \lambda ^{\frac{n}{m}-\frac{1}{m}(\frac{1}{2}-\beta -2\delta
  )\frac{1}{N_0+1}}
$$
for all $\lambda \ge 2^\ell$. This statement implies Theorem
\ref{in1}. \hfill{$\Box$}

\begin{proofof} Theorem \ref{in2}. This is just a minor modification
  of the proof of Theorem \ref{in1}. Indeed, we already used the
  second part of Proposition \ref{sc2}, to get (\ref{end.7}) with the
  probability indicated there. In that estimate we are free to
  vary $(g,\theta _1,\theta _2)$ in ${\cal G}$ and the same holds for
  the estimate (\ref{end.6}). With these modifications, the same proof 
gives Theorem \ref{in2}.
\end{proofof}

\section{Some open problems.}\label{open}
\setcounter{equation}{0}
\begin{itemize}
\item The distribution of {\it resonances} or scattering poles for certain self-adjoint operators like the Schr\"odinger operator is a very intriguing and difficult problem where many basic questions remain unanswered. 

Resonances can be viewed as eigenvalues of a non-self-adjoint operator, obtained from the original self-adjoint one by changing the Hilbert space. If we take a Schr\"odinger operator and add a random perturbation to the potential with support in some fixed compact set, we may ask wether with probability close to one the resonances obey some kind of Weyl asymptotics. In dimension 1 there is a classical result of Zworski \cite{Zw87}, saying that we do have Weyl asymptotics without any random perturbations, but assuming a non-flatness condition near the end points of the convex hull of the support of the potential. In higher dimensions there are results by T.~Christiansen \cite{Chr05, Chr06} and Christiansen and P.~Hislop \cite{ChrHi05} saying that in the ``generic case'' there is (roughly) a lower bound on the number of scattering poles in a sequence of large discs which is of the same order of magnitude as would be prescribed by a reasonable Weyl law. The methods of Christiansen and Hislop use the analysis of several complex variables and it would be interesting to understand the relation with the methods developed in our lectures. 
\item The {\it damped wave equation} in its stationary version is an example of a non-self-adjoint operator which is close to a self-adjoint one. The eigenvalues are confined to a band
  parallel to the real axis and since the work of 
Marcus-Matseev \cite{MaMa79} we know that the real parts obey the Weyl law with a good remainder estimate.
As for the distribution  of imaginary parts many results are known, G.~Lebeau \cite{Leb96}, Sj\"ostrand \cite{Sj00}, N.~Anantharaman \cite{An09}, S.~Nonnenmacher and E.~Schenk \cite{Sch08}, M.~Hitrik, Sj\"ostrand, S.~V\~u Ng\d{o}c, \cite{HiSjVu07}, \cite{HiSj08}.

\par The following problem seems to be open: Suppose we add some randomness to the damping term, and that the underlying geodesic flow  is not ergodic (assuming that we work on a compact manifold without boundary for simplicity). Then with probability close to 1, do the imaginary parts of the eigenvalues distribute according to a Weyl law formulated with the help of the time averages of the damping term? 

\item  What about {\it statistical properties of eigenvalues}? Can it be true, for instance in the simplest one dimensional situations with Gaussian random variables in the perturbation as in \cite{HaSj08}, that we have Poisson distribution of the eigenvalues? What about correlations between eigenvalues? Can we have results similar to the ones that are known for the zeros of random polynomials? (Cf \cite{ShZe99, BlShZe00}.) A step in this direction may be the recent preprint of T.~Christiansen and Zworski \cite{ChrZw09} where the authors study the expectation value of the number of eigenvalues in a domain for certain one-dimensional pseudodifferential operators with doubly periodic symbol.
\item It would be of great interest to have the sharpest possible {\it bounds on the norm of the resolvent}. The general theory of non-self-adjoint operators only gives very weak upper bounds which can be sharpened near the boundary of the range of the symbol (like for instantance in Section \ref{bdy}). In the proofs of the various results on Weyl asymptotics for randomly perturbed operators it is quite clear that the random perturbation has the effect of improving the upper bounds on the resolvent and sometimes in dimension one we can even get a polynomial upper bound, as was observed by Bordeaux Montrieux \cite{Bo}. It is currently not clear how far these improvements go in higher dimensions. Notice that for random matrices there are results showing that we can have a polynomial bound in terms of the size of the matrix. See M.~Rudelson \cite{Ru09}.  
\item Sometimes it is natural to have additional {\it symmetries}. For instance in the case of the Kramers-Fokker-Planck operator one would like to restrict the random perturbations to the class of such operators. Another such class is that of PT-symmetric operators where the question of reality of the spectrum seems to be of great importance. I recently showed that for elliptic PT symmetric operators with PT symmetric random perturbations, we have Weyl asymptotics with probabality close to 1. In particular {\it most PT symmetric operators have most of their eigenvalues away from the real axis}. 
\item A long term project may also be to apply the theory to {\it non-linear problems}.
\end{itemize}

\end{document}